\documentclass[reqno,centertags,12pt]{amsart}

\usepackage[normalem]{ulem}

\usepackage{amsmath,amsthm}
\usepackage{amssymb,amsbsy}
\usepackage{bbm}
\usepackage{tikz}

\usepackage{enumitem}
\usepackage{mleftright}
\usepackage{xcolor}
\usepackage{color}
\usepackage{hyperref}
\usepackage{cancel}
\usepackage{amsfonts}
\usepackage{graphicx}
\usepackage{bookmark}

\long\def\metanote#1#2{{\color{#1}\
		\ifmmode\hbox\fi{\sffamily\mdseries\upshape [#2]}\ }}

\newif\ifcomments
\commentstrue   

-\marginparwidth \oddsidemargin -4mm \evensidemargin \oddsidemargin

\makeatletter
\newcommand\xleftrightarrow[2][]{%
	\ext@arrow 9999{\longleftrightarrowfill@}{#1}{#2}}
\newcommand\longleftrightarrowfill@{%
	\arrowfill@\leftarrow\relbar\rightarrow}
\makeatother

\makeatletter
\newcommand{\xRightarrow}[2][]{\ext@arrow 0359\Rightarrowfill@{#1}{#2}}
\makeatother

\newcounter{smalllist}

\DeclareMathOperator*{\Lip}{Lip}

\allowdisplaybreaks
\numberwithin{equation}{section}

\newcommand{\al}{\alpha}

\newcommand{\ga}{\gamma}
\newcommand{\Ga}{\Gamma}
\newcommand{\de}{\delta}
\newcommand{\De}{\Delta}

\newcommand{\ta}{\theta}
\newcommand{\Ta}{\Theta}
\newcommand{\ka}{\kappa}
\newcommand{\la}{\lambda}
\newcommand{\La}{\Lambda}

\newcommand{\om}{\omega}
\newcommand{\Om}{\Omega}

\newcommand{\vphi}{\varphi}

\newcommand{\rb}[1]{\left(#1\right)}
\newcommand{\sqb}[1]{\left[#1\right]}
\newcommand{\cb}[1]{\left\{#1\right\}}

\newcommand{\inn}[1]{\left\langle #1 \right\rangle}

\newcommand{\tr}{\mathrm{tr}}

\newcommand{\mcl}{\mathcal}

\newcommand{\mbe}{\mathbb{E}}
\newcommand{\mbp}{\mathbb{P}}

\newcommand{\mcP}{\mathcal{P}}

\newcommand{\mcc}{\mathcal{C}}
\newcommand{\mcg}{\mathcal{G}}
\newcommand{\mcp}{\mathcal{P}}
\newcommand{\mA}{\mathcal{A}}
\newcommand{\mB}{\mathcal{B}}

\newcommand{\mJ}{\mathcal{J}}
\newcommand{\mX}{\mathcal{X}}

\newcommand{\mbn}{\mathbb N}

\newcommand{\mbr}{\mathbb R}

\newcommand{\f}{\frac}

\newcommand{\bs}{\boldsymbol}

\newcommand{\ID}{\mathrm{ID}}
\newcommand{\levyG}{\mathcal{G}^\La}

\newtheorem{theorem}{Theorem}[section]
\newtheorem{proposition}[theorem]{Proposition}
\newtheorem{lemma}[theorem]{Lemma}

\theoremstyle{definition}
\newtheorem{definition}[theorem]{Definition}

\newtheorem{remark}[theorem]{Remark}

\theoremstyle{definition}

\newtheorem{assumptionC}{Assumption}

\begin{document}
\title[]{Optimal Couplings of L\'evy Processes in the Class of Immersion Couplings}

\author{Tau Shean Lim, Ray Shua Ooi}

\address{\noindent Department of Mathematics \\ Xiamen University Malaysia}
\begin{abstract}
	We prove that any two L\'evy processes on $\mathbb{R}^d$ with finite second moments admit a coupling that is itself a L\'evy process and is both locally and globally optimal for the quadratic cost among all immersion couplings. This strengthens the optimality result of \cite{KangLim2025}, which was established among Feller couplings. Our approach uses a martingale criterion for local optimality that applies more generally to Markov processes. The argument rests on the existence of an optimal L\'evy coupling and a dual formulation of the associated optimization problem, both previously established in \cite{KangLim2025}. We give simpler, self-contained proofs of these two results using optimal transport between infinitely divisible measures and Lassonde's minimax principle.


	\vspace{6ex}
	
\end{abstract}
\maketitle
\section{Introduction}

Given two Markov processes $\{X_t\}_{t\ge0}$ and $\{Y_t\}_{t\ge0}$ on state spaces $\Pi$ and $\tilde\Pi$, a natural question is to construct a coupling — typically also Markovian — that is optimal in a certain sense. This is commonly referred to as the \emph{optimal Markovian coupling problem}. Such optimal couplings provide a quantitative way to compare the evolution of stochastic processes at fixed times and play an important role in probability theory and its applications; see, e.g., \cite{doeblin1937theorie, Griffeath1975, LindvallRogers1986, kendall1986nonnegative, kendall2015coupling, levin2017markov, wang2013harnack, stroock2006multidimensional, Chaintron_2022_1, Chen2020OptimalCouplings}.

A classical instance of this problem is the \emph{maximal coupling} of two Markov processes, which attains the total variation lower bound on the probability that the two processes differ at every time. For two processes $\{X_t\}$ and $\{Y_t\}$, a coupling $\mathbf{P}_*$ is maximal if
\[
\mathbf{P}_*(X_t \neq Y_t) = d_{\mathrm{TV}}(\mu_t, \nu_t),\qquad \mbox{for all }t\ge 0,
\]
where $\mu_t$ and $\nu_t$ are the laws of $X_t$ and $Y_t$, respectively, and $d_{\mathrm{TV}}$ denotes the total variation distance. It is well known that a maximal coupling of the time-marginal measures exists at each fixed time \cite{doeblin1937theorie,Griffeath1975,levin2017markov,propp1996exact}. In the context of Markov processes, such couplings are widely used to study mixing times, convergence rates, and perfect sampling algorithms.

Beyond maximal coupling, optimal couplings of processes with respect to a cost function $c: \Pi \times \tilde\Pi \to \mathbb{R}_+$ are studied within the framework of \emph{optimal transport} extended to stochastic processes. For instance, for two Markov processes $\{X_t\}$ and $\{Y_t\}$, one seeks a coupling that minimizes the expected cost 
\begin{align*}
	\mathbf{E}^{(x,y)}[c(X_t, Y_t)],\qquad \text{for all } t \ge 0,\; (x,y) \in \Pi\times\tilde\Pi.
\end{align*}
This condition will be called \emph{global $c$-optimality.}
For Brownian motion, the classical \emph{reflection coupling} (see \cite{LindvallRogers1986}) minimizes the meeting time and is optimal for concave costs, while the \emph{synchronous coupling} is optimal for convex costs. More recently, \cite{KendallMajkaMijatovic2024} studied this problem for identical continuous-time random walks (finite-activity L\'evy processes) on $\mathbb{R}$, under the assumption of unimodal jump distributions and a cost function $c(x,y) = \phi(|x-y|)$, where $\phi: \mathbb{R}_+ \to \mathbb{R}_+$ is concave and increasing.

One issue with the global optimality condition above, as pointed out by Chen \cite{Chen1994}, is that a globally $c$-optimal coupling may not exist in general, even for simple Markov processes. The requirement that the coupling minimizes the expected cost at all times $t \ge 0$ simultaneously is extremely strong, and existence is known only in some special cases, such as those mentioned above.
To address this difficulty, Chen \cite{Chen1994} proposed an alternative, weaker notion of optimality based on the infinitesimal generator of the coupled process. A Markov coupling process with generator $\mathcal{J}_*$ is called \emph{locally $c$-optimal} if
\begin{align}\label{def:gen-opt}
	(\mathcal{J}_* c)(x,y) = \inf_{\mathcal{J}} (\mathcal{J} c)(x,y), \qquad \text{for all } (x,y) \in \Pi \times \tilde\Pi,
\end{align}
where the infimum is taken over all generators $\mathcal{J}$ of coupling processes of $\{X_t\}$ and $\{Y_t\}$. In other words, the coupling minimizes the \emph{derivative} of the expected cost $\mathbf{E}^{(x,y)}[c(X_t, Y_t)]$ at $t = 0$. This approach shifts the focus from global pathwise optimality to the infinitesimal level, where existence is more tractable. 
Chen showed that, under suitable assumptions, global $c$-optimality implies local $c$-optimality. He also proved the existence of locally $c$-optimal couplings for drift-diffusion processes on $\mathbb{R}^d$ with concave cost $c(x,y) = \phi(|x-y|)$, where $\phi$ is concave. Later, Zhang \cite{Zhang2000} extended the existence result to bounded Markov jump processes with general lower semicontinuous cost $c$.
The existence of a locally $c$-optimal coupling for general Markov (or Feller) processes, to the best of the authors' knowledge, remains open.

Recently, the first author (together with Kang \cite{KangLim2025}) resolved the case for general unbounded L\'evy processes with the quadratic cost $c(x,y) = \frac12|x-y|^2$, establishing the existence of a locally $c$-optimal coupling that is itself a L\'evy process. Moreover, this coupling is also globally $c$-optimal. Specifically, the optimality (both local and global) is established among all \emph{Feller} couplings.

The main result of the present paper establishes that, for any two L\'evy processes with finite second moments, the aforementioned L\'evy coupling is both locally and globally optimal for the quadratic cost among all \emph{immersion couplings}, extending the optimality among Feller couplings proved in \cite{KangLim2025}. Immersion couplings preserve the martingale property of each component process with respect to the joint filtration. They allow coupling strategies to depend on the observed past while preserving the marginal dynamics, without requiring the coupled process to be Markovian. In particular, every Markovian coupling is an immersion coupling, but the converse is not true.

The importance of this class in optimal coupling problems is illustrated by Kendall's study of Brownian motion together with its local time \cite{kendall2015coupling}, as well as by the finite-activity L\'evy coupling problem \cite{KendallMajkaMijatovic2024} discussed above. Even when the proposed optimal coupling is itself Markovian, optimality among immersion couplings requires comparison with strategies that may depend on the past beyond the current state. Our result shows that such strategies cannot improve upon the optimal L\'evy coupling for the quadratic cost. Since competing immersion couplings need not admit a joint generator description, we establish this optimality through a martingale formulation. More generally, we establish a martingale criterion for local optimality of immersion couplings of Markov processes under suitable regularity and integrability assumptions.

The optimal coupling problem for L\'evy-type processes, whose characteristics depend on the current state, is a natural direction for further investigation, with potential applications to the comparison and long-time behavior of such processes. In this setting, a generator-based approach requires not only the construction of a suitable coupling generator but also a proof that it defines a well-posed Markov process, which can be analytically challenging. The immersion framework offers additional flexibility by allowing coupling strategies that need not yield a Markovian joint process and by formulating local optimality through martingale properties. The present work develops this approach in the translation-invariant L\'evy setting and provides a starting point for studying the more general problem.

Our proof is based on two results: the existence of a coupling that is optimal among all \emph{L\'evy} couplings, and the associated strong duality, both previously established in \cite{KangLim2025}. We provide simpler proofs of these two results using optimal transport between infinitely divisible measures and Lassonde's minimax principle, making the present paper self-contained and readable independently of \cite{KangLim2025}. We then establish that this L\'evy coupling is both locally and globally optimal among all immersion couplings.

We begin in the next section with the basic setting and the statement of the main results. The plan for the remainder of the paper and further discussion of the results are given at the end of Section \ref{sec:1}, after all the necessary notation and assumptions have been introduced.

\section{Settings and Main Results}\label{sec:1}
\subsection{Immersion couplings}\label{sec:1.1}
In the literature \cite{kendall2015coupling,KendallMajkaMijatovic2024}, the notion of \emph{immersion couplings} is defined for two stochastic processes. 
Formally, let $X=\{X_t\}_{t\ge0}$ and $Y=\{Y_t\}_{t\ge0}$ be processes defined on a common probability space. 
An \emph{immersion coupling} $(X', Y') = \{(X'_t, Y'_t)\}_{t\ge 0}$ is a process on some filtered space, adapted to a common filtration $\mathfrak{F} = \{\mathcal{F}_t\}_{t\ge 0}$, satisfying the following conditions:
\begin{enumerate}
	\item The marginal laws agree: $\operatorname{law}(X') = \operatorname{law}(X)$, $\operatorname{law}(Y') = \operatorname{law}(Y)$.
	\item Any martingale with respect to the natural filtration of $X'$ is also an $\mathfrak{F}$-martingale, and similarly for $Y'$.
\end{enumerate}
Intuitively speaking, an immersion coupling preserves the individual dynamics of each process while allowing them to be coupled in such a way that no ``extra information'' is introduced beyond the joint filtration.
In the present work, we use the framework of \emph{filtered probability spaces}, also called \emph{stochastic bases}. We briefly review the notion of immersion and its equivalent formulations in this setting.

\subsubsection{Basics of filtered probability spaces}
Before presenting the definition, we recall some relevant terminology.
A \emph{filtered measurable space} is a triplet $(\Omega,\mathcal{F},\mathfrak{F})$, where $(\Omega,\mathcal{F})$ is a measurable space and $\mathfrak{F}=\{\mathcal{F}_t\}_{t\ge 0}$ is a filtration, i.e., an increasing family of sub-$\sigma$-algebras of $\mathcal{F}$.
In particular, we will often take $(\Omega,\mathcal{F},\mathfrak{F})$ to be \emph{the canonical path space of c\`adl\`ag functions from $\mathbb{R}_+$ to $\Pi$}, where $\Pi$ is a Polish space. That is,
\begin{align*}
	\Omega = D([0,\infty); \Pi) := \{ \omega: [0,\infty) \to \Pi \mid \omega \text{ is c\`adl\`ag} \},
\end{align*}
equipped with the $\sigma$-algebra $\mathcal{F}$ generated by the coordinate projections $\omega \mapsto \omega(t)$ for $t \ge 0$ (the \emph{Skorokhod $\sigma$-algebra}). The filtration $\mathfrak{F} = \{\mathcal{F}_t\}_{t\ge 0}$ is the raw natural filtration, defined by $\mathcal{F}_t := \sigma(\omega(s):0\le s\le t)$, the $\sigma$-algebra generated by the paths up to time $t$.

Given a measurable state space $(\Pi,\mathcal{B})$, a \emph{$\Pi$-valued stochastic process} on $(\Omega,\mathcal{F},\mathfrak{F})$ is a measurable function
\[
S:\mathbb{R}_+\times \Omega \to \Pi
\]
such that $S_t(\omega)=S(t,\omega)$ is $\mathcal{F}_t$-measurable for all $t\ge 0$, i.e., $\{S_t\}_{t\ge0}$ is $\mathfrak{F}$-adapted.
On the canonical path space $\Omega = D([0,\infty);\Pi)$, the canonical process is given by $X_t(\omega) = \omega(t)$, which is c\`adl\`ag and adapted to the natural filtration.

A \emph{filtered probability space}, also known as a \emph{stochastic basis}, is a quadruple $(\Omega,\mathcal{F},\mathfrak{F},\mathbb{P})$, sometimes also denoted $(\Omega,\mathcal{F},\mathfrak{F},\mathbb{P},\mathbb{E})$ to emphasize the expectation operator
\[
\mathbb{E}[X] = \int_\Omega X \, d\mathbb{P},
\]
which allows us to distinguish expectations under different probability measures on the same filtered measurable space.
Whenever a usual augmentation is needed, it is taken under the particular probability measure under consideration.

Given two filtered measurable spaces $(\Omega,\mathcal{F},\mathfrak{F})$ and $(\tilde\Omega,\tilde{\mathcal{F}},\tilde{\mathfrak{F}})$, their \emph{product filtered measurable space} is defined as
\begin{align}\label{def:prod-filtered}
	(\boldsymbol{\Omega},\boldsymbol{\mathcal{F}},\boldsymbol{\mathfrak{F}}) := (\Omega \times \tilde\Omega, \mathcal{F} \otimes \tilde{\mathcal{F}}, \mathfrak{F} \otimes \tilde{\mathfrak{F}}),
\end{align}
where $\Omega \times \tilde\Omega$ is the Cartesian product, $\mathcal{F} \otimes \tilde{\mathcal{F}}$ is the product $\sigma$-algebra, and
\[
\boldsymbol{\mathfrak{F}} = \mathfrak{F} \otimes \tilde{\mathfrak{F}} := \{\mathcal{F}_t \otimes \tilde{\mathcal{F}}_t\}_{t \ge 0}
\]
is the \emph{product filtration}.
We adopt the following notational convention throughout the present work: boldface symbols such as $\boldsymbol{\mathcal{F}}, \boldsymbol{\mathfrak{F}}, \mathbf{P}$ denote objects associated with the product or coupling space, while regular symbols such as $\mathcal{F}, \mathfrak{F}, \mathbb{P}$ refer to the marginal spaces.

Given two $\mathbb{R}$-valued processes $S=\{S_t\}_{t\ge0}$ on $(\Omega,\mathcal{F},\mathfrak{F})$ and $\tilde S=\{\tilde S_t\}_{t\ge0}$ on $(\tilde\Omega,\tilde{\mathcal{F}},\tilde{\mathfrak{F}})$, we define their \emph{product} and \emph{sum} processes on the product filtered space $(\bs\Om,\bs{\mcl{F}},\bs{\mathfrak{F}})$ by
\[
(S \otimes \tilde S)_t(\omega, \tilde\omega) := S_t(\omega) \cdot \tilde S_t(\tilde\omega), 
\qquad
(S \oplus \tilde S)_t(\omega, \tilde\omega) := S_t(\omega) + \tilde S_t(\tilde\omega), 
\qquad (\omega, \tilde\omega) \in \Omega \times \tilde\Omega.
\]
In particular, $S \otimes 1$ denotes the process $S$ lifted from $(\Omega,\mathcal{F},\mathfrak{F})$ to the product space $(\bs\Om,\bs{\mcl{F}},\bs{\mathfrak{F}})$, and similarly $1 \otimes \tilde S$ denotes the lifting of $\tilde S$.

Finally, we briefly recall the notion of martingales, which plays a critical role in the definition of immersion couplings. 
Let $(\Omega,\mathcal{F},\mathfrak{F},\mathbb{P},\mathbb{E})$ be a filtered probability space. 
A real-valued process $M=\{M_t\}_{t\ge0}$ on $(\Omega,\mathcal{F},\mathfrak{F})$ is called a \emph{$(\mathfrak{F},\mathbb{P})$-martingale} (or \emph{$(\mathfrak{F},\mathbb{P},\mathbb{E})$-martingale}, to emphasize the expectation operator) if, for all $0\le s\le t$, 
$M_t \in L^1(\mathbb{P})$, and
\[
\mathbb{E}[M_t \mid \mathcal{F}_s] = M_s, \qquad \mathbb{P}\text{-a.s.}
\]
Similarly, $(\mathfrak{F},\mathbb{P})$-submartingales and supermartingales are defined by replacing the equality with ``$\ge$'' and ``$\le$,'' respectively.  
We stress that the martingale property depends critically on both the filtration $\mathfrak{F}$ and the probability measure $\mathbb{P}$ (hence the associated expectation $\mathbb{E}$). 
In particular, given two probability measures $\mathbb{P}$ and $\mathbb{P}'$ on the same filtered measurable space $(\Omega,\mathcal{F},\mathfrak{F})$, a process $M$ may be a martingale under $(\mathfrak{F},\mathbb{P})$ but fail to be one under $(\mathfrak{F},\mathbb{P}')$.

\subsubsection{Immersion couplings and their equivalent formulations}
Let us now provide the definition of immersion couplings. Given two probability spaces $(\Omega, \mathcal{F}, \mathbb{P})$ and $(\tilde\Omega, \tilde{\mathcal{F}}, \tilde{\mathbb{P}})$, by a \emph{coupling} we mean a probability measure $\mathbf{P}$ on the product space $(\boldsymbol{\Omega}, \boldsymbol{\mathcal{F}}) = (\Omega \times \tilde\Omega, \mathcal{F} \otimes \tilde{\mathcal{F}})$ such that for all $E \in \mathcal{F}$ and $\tilde E \in \tilde{\mathcal{F}}$,
\begin{align}\label{eq:coupling}
	\mathbf{P}(E \times \tilde\Omega) = \mathbb{P}(E), \qquad
	\mathbf{P}(\Omega \times \tilde E) = \tilde{\mathbb{P}}(\tilde E).
\end{align}
To emphasize the expectation operator, we sometimes denote the coupling by $(\mathbf{P}, \mathbf{E})$.

\begin{definition}[Immersion couplings] \label{def:immersion}
	Consider two filtered probability spaces
	\[
	(\Omega, \mathcal{F}, \mathfrak{F}, \mathbb{P}, \mathbb{E}), 
	\qquad
	(\tilde\Omega, \tilde{\mathcal{F}}, \tilde{\mathfrak{F}}, \tilde{\mathbb{P}}, \tilde{\mathbb{E}}).
	\]	
	An \emph{immersion coupling} of $\mathbb{P}$ and $\tilde{\mathbb{P}}$ is a coupling $(\mathbf{P}, \mathbf{E})$ on the product filtered measurable space $(\boldsymbol{\Omega}, \boldsymbol{\mathcal{F}}, \boldsymbol{\mathfrak{F}})$ (see \eqref{def:prod-filtered}) such that, for every $t \ge 0$, every $\mathcal{F}_\infty$-measurable $L^1$-random variable $M: \Omega \to \mathbb{R}$, and every $\tilde{\mathcal{F}}_\infty$-measurable $L^1$-random variable $\tilde M: \tilde\Omega \to \mathbb{R}$,
	\begin{align}\label{eq:marg-condexp}
		\mathbf{E}[M \otimes 1 \mid \boldsymbol{\mathcal{F}}_t] = \mathbb{E}[M \mid \mathcal{F}_t] \otimes 1, \qquad
		\mathbf{E}[1 \otimes \tilde M \mid \boldsymbol{\mathcal{F}}_t] = 1 \otimes \tilde{\mathbb{E}}[\tilde M \mid \tilde{\mathcal{F}}_t].	
	\end{align}
\end{definition}

\begin{remark}
	When the initial $\sigma$-algebras are trivial, i.e., $\mathcal{F}_0$ contains only events of $\mathbb{P}$-probability $0$ or $1$ (and similarly for $\tilde{\mathcal{F}}_0$ under $\tilde{\mathbb{P}}$), the conditional expectation condition \eqref{eq:marg-condexp} at $t = 0$ reduces to
	\[
	\mathbf{E}[M \otimes 1] = \mathbb{E}[M], \qquad
	\mathbf{E}[1 \otimes \tilde M] = \tilde{\mathbb{E}}[\tilde M],
	\]
	which is equivalent to the marginal condition \eqref{eq:coupling}. Hence, the marginal condition is already contained in \eqref{eq:marg-condexp} in this case.
\end{remark}
In the literature, e.g., \cite{kendall2015coupling,KendallMajkaMijatovic2024}, immersion couplings are often characterized through the \emph{martingale preservation property} stated as follows.

\begin{definition}[Martingale preservation]
	Assume the setting of Definition \ref{def:immersion}. A coupling $(\mathbf{P}, \mathbf{E})$ is said to \emph{preserve martingales} if:
	\begin{itemize}
		\item If $\{M_t\}_{t\ge0}$ is a $(\mathfrak{F}, \mathbb{P})$-martingale, then $\{M_t \otimes 1\}_{t\ge0}$ is a $(\boldsymbol{\mathfrak{F}}, \mathbf{P})$-martingale.
		\item If $\{\tilde M_t\}_{t\ge0}$ is a $(\tilde{\mathfrak{F}}, \tilde{\mathbb{P}})$-martingale, then $\{1 \otimes \tilde M_t\}_{t\ge0}$ is a $(\boldsymbol{\mathfrak{F}}, \mathbf{P})$-martingale.
	\end{itemize}
\end{definition}

The following proposition expresses the classical equivalence between the conditional expectation formulation of immersion and martingale preservation \cite[Theorem 3]{BremaudYor1978} in the present coupling framework. We include a proof for completeness.

\begin{proposition}\label{prop:immersion}
	Let $(\Omega, \mathcal{F}, \mathfrak{F}, \mathbb{P}, \mathbb{E})$ and $(\tilde\Omega, \tilde{\mathcal{F}}, \tilde{\mathfrak{F}}, \tilde{\mathbb{P}}, \tilde{\mathbb{E}})$ be two filtered probability spaces. A coupling $(\mathbf{P}, \mathbf{E})$ on the product filtered space $(\boldsymbol{\Omega}, \boldsymbol{\mathcal{F}}, \boldsymbol{\mathfrak{F}})$ is an immersion coupling of $\mathbb{P}$ and $\tilde{\mathbb{P}}$ if and only if it preserves martingales.
\end{proposition}

\begin{remark}
	The conditional expectation formulation \eqref{eq:marg-condexp} strengthens the usual marginal coupling condition \eqref{eq:coupling}: it requires the coupling to preserve not only the marginal laws but also the marginal conditional expectations at every time. When the initial $\sigma$-algebras are trivial, the identities at $t=0$ reduce to the usual marginal condition. Thus, the optimal immersion coupling problem can be viewed as an optimal transport problem on path spaces with additional constraints preserving the marginal conditional dynamics.
\end{remark}

\begin{proof}
	\textit{Immersion implies martingale preservation.} 
	Assume that $(\mathbf{P}, \mathbf{E})$ is an immersion coupling, i.e., it satisfies the conditional expectation condition \eqref{eq:marg-condexp}. We show that it preserves martingales.
	
	Let $\{M_t\}_{t\ge 0}$ be a $(\mathfrak{F}, \mathbb{P})$-martingale. For any fixed $t \ge 0$, $M_t$ is $\mathcal{F}_\infty$-measurable and $M_t \in L^1(\mathbb{P})$. Applying the conditional expectation condition with $M = M_t$ gives, for any $s \le t$,
	\[
	\mathbf{E}[M_t \otimes 1 \mid \boldsymbol{\mathcal{F}}_s] = \mathbb{E}[M_t \mid \mathcal{F}_s] \otimes 1 = M_s \otimes 1.
	\]
	Hence $\{M_t \otimes 1\}_{t\ge 0}$ is a $(\boldsymbol{\mathfrak{F}}, \mathbf{P})$-martingale. The same argument for martingales under $\tilde{\mathbb{P}}$ shows that $\mathbf{P}$ preserves martingales.
	
	\textit{Martingale preservation implies immersion.} 
	Conversely, suppose $\mathbf{P}$ preserves martingales. For any $t \ge 0$ and any $\mathcal{F}_\infty$-measurable $M \in L^1(\mathbb{P})$, define $M_s := \mathbb{E}[M \mid \mathcal{F}_s]$ for $s \ge 0$. Then $\{M_s\}_{s\ge 0}$ is a $(\mathfrak{F}, \mathbb{P})$-martingale. By the martingale preservation property, $\{M_s \otimes 1\}_{s\ge 0}$ is a $(\boldsymbol{\mathfrak{F}}, \mathbf{P})$-martingale. Hence, for any $T \ge t$,
	\[
	\mathbf{E}[M_T \otimes 1 \mid \boldsymbol{\mathcal{F}}_t] = M_t \otimes 1 = \mathbb{E}[M \mid \mathcal{F}_t] \otimes 1.
	\]
	As $T \to \infty$, $M_T \to \mathbb{E}[M \mid \mathcal{F}_\infty] = M$ almost surely and in $L^1(\mathbb{P})$ by the martingale convergence theorem. 
    By the marginal condition, as $T\to\infty$,
    \begin{align*}
        \mathbf{E}[|M\otimes 1-M_T\otimes 1|]= \mathbf{E}[|M-M_T|\otimes 1]=\mathbb{E}[|M-M_T|]\to0. 
    \end{align*}
    Hence $M_T\otimes 1\to M\otimes 1$ in $L^1(\mathbf{P})$.
    By the $L^1$-continuity of conditional expectations, we obtain
	\[
	\mathbf{E}[M \otimes 1 \mid \boldsymbol{\mathcal{F}}_t] = \lim_{T \to \infty} \mathbf{E}[M_T \otimes 1 \mid \boldsymbol{\mathcal{F}}_t] = \mathbb{E}[M \mid \mathcal{F}_t] \otimes 1.
	\]
	The analogous identity for $\tilde M \in L^1(\tilde{\mathbb{P}})$ follows by the same argument. Thus $\mathbf{P}$ satisfies the conditional expectation condition \eqref{eq:marg-condexp} and is therefore an immersion coupling.
\end{proof}

\subsubsection{Markovian couplings} 
A particularly important class of immersion couplings arises when the two processes are \emph{Markovian}. Let $(\Omega, \mathcal{F}, \mathfrak{F})$ be the canonical path space of c\`adl\`ag functions from $\mathbb{R}_+$ to $\Pi$, carrying the coordinate process $\{X_t\}_{t\ge 0}$ (see the discussion above). Let $\{(\mathbb{P}^x, \mathbb{E}^x)\}_{x\in\Pi}$ be a family of \emph{(time-homogeneous) Markov measures}; that is, for each $x \in \Pi$, any bounded measurable function $f:\Om\to\mbr$ and $t\ge 0$,
\[
\mathbb{P}^x(X_0 = x) = 1,\qquad \mathbb{E}^x\bigl[ f \circ \Theta_t \mid \mathcal{F}_t \bigr] = \mathbb{E}^{X_t}[f] \quad \mathbb{P}^x\text{-a.s.},
\]
where $\Theta_t:\Om\to\Om$ is the shift operator $(\Theta_t \omega)(s) = \omega(t+s)$. 
For any initial distribution $\mu \in \mathcal{P}(\Pi)$, the \emph{Markov measure with initial distribution $\mu$} is defined as
\[
\mathbb{P}^\mu := \int_{\Pi} \mathbb{P}^x \, \mu(dx), \qquad 
\mathbb{E}^\mu := \int_{\Pi} \mathbb{E}^x \, \mu(dx),
\]
so that in particular $(\mathbb{P}^x, \mathbb{E}^x) = (\mathbb{P}^{\delta_x}, \mathbb{E}^{\delta_x})$.

Assume the setting above. 
Let $(\tilde\Omega, \tilde{\mathcal{F}}, \tilde{\mathfrak{F}})$ be another filtered measurable space carrying a $\tilde\Pi$-valued process $\{Y_t\}_{t\ge 0}$, and let $\{(\tilde{\mathbb{P}}^y, \tilde{\mathbb{E}}^y)\}_{y\in \tilde\Pi}$ be a family of Markov measures on this space, defined analogously.
A \emph{(time-homogeneous) Markovian coupling} is a family of \emph{Markov} measures $\{\mathbf{P}^{(x,y)}\}_{(x,y)\in\Pi\times\tilde\Pi}$ on the product filtered space $(\boldsymbol{\Omega}, \boldsymbol{\mathcal{F}}, \boldsymbol{\mathfrak{F}})$ such that for each $(x,y)$, the marginal laws of $\mathbf{P}^{(x,y)}$ are $\mathbb{P}^x$ and $\tilde{\mathbb{P}}^y$; i.e., for all $E \in \mathcal{F}$, $\tilde E \in \tilde{\mathcal{F}}$,
\[
\mathbf{P}^{(x,y)}(E \times \tilde\Omega) = \mathbb{P}^x(E), \qquad 
\mathbf{P}^{(x,y)}(\Omega \times \tilde E) = \tilde{\mathbb{P}}^y(\tilde E).
\]
In particular, the coupling requires that the coupled process $\{(X_t, Y_t)\}_{t\ge 0}$ is Markovian with respect to the product filtration $\boldsymbol{\mathfrak{F}}$, while each marginal process preserves its original Markov law $\mathbb{P}^x$ and $\tilde{\mathbb{P}}^y$, respectively.

Every Markovian coupling is in particular an immersion coupling. Indeed, fix $t\ge0$. For any bounded measurable function $f:\Omega\to\mathbb{R}$, the Markov property of the coupled process and its marginal laws give
\[
\mathbf{E}^{(x,y)}[(f\circ\Theta_t)\otimes1\mid\boldsymbol{\mathcal{F}}_t]
=\mathbf{E}^{(X_t,Y_t)}[f\otimes1]
=\mathbb{E}^{X_t}[f]\otimes1
=\mathbb{E}^x[f\circ\Theta_t\mid\mathcal{F}_t]\otimes1.
\]
Now consider a bounded cylinder function $M=\prod_{j=1}^n a_j(X_{s_j})$, where the $a_j$ are bounded Borel functions. We can write
\[
M=H_t(f_t\circ\Theta_t),
\qquad
H_t:=\prod_{s_j\le t}a_j(X_{s_j}),
\qquad
f_t(\omega):=\prod_{s_j>t}a_j\bigl(\omega(s_j-t)\bigr),
\]
with empty products understood as $1$. Since $H_t$ is $\mathcal{F}_t$-measurable, the preceding identity implies
\begin{align*}
	\mathbf{E}^{(x,y)}[M\otimes1\mid\boldsymbol{\mathcal{F}}_t]
	&=(H_t\otimes1)\,\mathbf{E}^{(x,y)}[(f_t\circ\Theta_t)\otimes1\mid\boldsymbol{\mathcal{F}}_t]\\
	&=(H_t\otimes1)\,\bigl(\mathbb{E}^x[f_t\circ\Theta_t\mid\mathcal{F}_t]\otimes1\bigr)
	=\mathbb{E}^x[M\mid\mathcal{F}_t]\otimes1.
\end{align*}
Cylinder functions generate the coordinate $\sigma$-algebra $\mathcal{F}_\infty$, so a monotone class argument extends this identity to every bounded $\mathcal{F}_\infty$-measurable function $M$. Truncation and continuity of conditional expectation in $L^1$ then extend it to every $M\in L^1(\mathbb{P}^x)$. The same argument applies to the second coordinate. Since $t$ was arbitrary, these are precisely the conditional expectation identities in Definition \ref{def:immersion}. Hence $\mathbf{P}^{(x,y)}$ is an immersion coupling of $\mathbb{P}^x$ and $\tilde{\mathbb{P}}^y$.

More generally, any convex combination (or integral) of Markovian couplings remains an immersion coupling, but such a mixture is typically not Markovian. Hence, immersion couplings form a strictly larger class than Markovian couplings. In fact, one can construct more sophisticated immersion couplings, for instance by ``time gluing'': following one Markovian coupling up to a fixed time $T$, and then switching to another thereafter. The resulting process is still an immersion coupling (by conditioning at the switching time and using martingale preservation on each time interval). 

If the two Markov processes admit a generator description (e.g., they are Feller processes), then the marginal condition for a Markovian coupling can be expressed at the level of generators. Specifically, let $\mathcal{A}$ be the infinitesimal generator of $\{\mathbb{P}^x\}_{x\in\Pi}$ and $\mathcal{B}$ be that of $\{\tilde{\mathbb{P}}^y\}_{y\in\tilde\Pi}$. Then a Markovian coupling $\{\mathbf{P}^{(x,y)}\}_{(x,y)\in\Pi\times\tilde\Pi}$ has generator $\mathcal{J}$ satisfying the following marginal conditions: for all $\varphi \in \mathcal{D}(\mathcal{A})$ and $\psi \in \mathcal{D}(\mathcal{B})$,
\begin{align}\label{def:coupling-gen}
    \mathcal{J}(\varphi \otimes 1) = (\mathcal{A}\varphi) \otimes 1, \qquad 
\mathcal{J}(1 \otimes \psi) = 1 \otimes (\mathcal{B}\psi).    
\end{align}
In other words, the generator of the coupled process acts on functions that depend only on $X$ (resp. $Y$) as the generator of the marginal process. This generator formulation of coupling processes was adopted in \cite{Chen1994, Chen2020OptimalCouplings} in the study of optimal couplings between diffusion processes and jump processes.

\subsection{Optimality of immersion couplings with respect to a cost}
We now introduce two optimality conditions for immersion couplings, based on optimal transport. \emph{Local optimality} extends the generator-based condition \eqref{def:gen-opt} from \cite{Chen1994, KangLim2025}, while \emph{global optimality} follows the formulation of \cite{KendallMajkaMijatovic2024}.

\subsubsection{Local and global $c$-optimality}
Throughout this section, fix two (Polish) state spaces $\Pi$ and $\tilde\Pi$ and a \emph{lower semicontinuous cost function} $c: \Pi \times \tilde\Pi \to \mathbb{R}_+$. 
Consider two Markov processes. 
Let $(\Omega, \mathcal{F}, \mathfrak{F})$ be a filtered measurable space carrying a $\Pi$-valued process $\{X_t\}_{t\ge 0}$ and $(\tilde\Om,\tilde{\mcl{F}},\tilde{\mathfrak{F}})$ be another filtered space carrying a $\tilde\Pi$-valued process $\{Y_t\}_{t\ge 0}$. 
Let $\{(\mbp^x,\mbe^x)\}_{x\in\Pi}$, $\{(\tilde \mbp^y,\tilde\mbe^y)\}_{y\in\tilde \Pi}$ be two families of Markov measures on the two spaces. 


\begin{definition}[Optimality of immersion couplings]\label{def:optimal-immersion}
	Assume the setting introduced above. Let $\mu \in \mathcal{P}(\Pi)$ and $\nu \in \mathcal{P}(\tilde\Pi)$ be two initial distributions, and let $(\mathbf{P}_*, \mathbf{E}_*)$ be an immersion coupling of $\mathbb{P}^\mu$ and $\tilde{\mathbb{P}}^{\nu}$.
	
	\begin{enumerate}[label=(\roman*)]
		\item $(\mathbf{P}_*, \mathbf{E}_*)$ is called \emph{locally $c$-optimal} if there is a measurable function $\om:\Pi\times\tilde\Pi\to\mbr$ such that the following holds:
		\begin{itemize}
			\item For every immersion coupling $(\mathbf{P}, \mathbf{E})$ of $\mathbb{P}^\mu$ and $\tilde{\mathbb{P}}^\nu$, the process $\{S_t\}_{t\ge 0}$ defined below is a $(\boldsymbol{\mathfrak{F}}, \mathbf{P})$-submartingale: 
			\begin{align}\label{def:inf-mart1}
				S_t := c(X_t, Y_t) - \int_0^t \omega(X_s, Y_s) \, ds
			\end{align}
			\item $\{S_t\}_{t\ge 0}$ is a $(\boldsymbol{\mathfrak{F}}, \mathbf{P}_*)$-martingale.
		\end{itemize}
		
		\item $(\mathbf{P}_*, \mathbf{E}_*)$ is called \emph{globally $c$-optimal} if for every immersion coupling $(\mathbf{P}, \mathbf{E})$ of $\mathbb{P}^\mu$ and $\tilde{\mathbb{P}}^{\nu}$ and all $0 \le s \le t$,
		\[
		\mathbf{E}_*\bigl[ c(X_t, Y_t) \mid X_s, Y_s \bigr] \le \mathbf{E}\bigl[ c(X_t, Y_t) \mid X_s, Y_s \bigr].
		\]
	\end{enumerate}
\end{definition}

\begin{remark}
	The inequality of conditional expectations in (ii) is understood in the following sense. By the Doob--Dynkin lemma, for every immersion coupling $(\mathbf{P}, \mathbf{E})$, there exists a measurable function $\rho:\Pi\times\tilde\Pi\to\mathbb{R}$ such that
	\[
	\mathbf{E}[c(X_t, Y_t) \mid X_s, Y_s] = \rho(X_s, Y_s) \quad \mathbf{P}\text{-a.s.}
	\]
	The global optimality condition therefore means: there exist measurable functions $\rho_*$ and $\rho$, chosen as versions associated with $\mathbf{P}_*$ and $\mathbf{P}$, respectively, such that
	\[
	\rho_*(x,y) \le \rho(x,y) \qquad \text{for all } (x,y) \in \Pi \times \tilde\Pi.
	\]
\end{remark}

Local optimality is a natural extension of the generator optimality condition \eqref{def:gen-opt}, as introduced in \cite{Chen2020OptimalCouplings, KangLim2025}, to the process level. In a locally optimal coupling, there exists a measurable function $\omega$ (the minimal drift) such that the process \eqref{def:inf-mart1} is a submartingale for every immersion coupling and a martingale for the optimal one. Heuristically speaking, for any immersion coupling $\mathbf{P}$, the lower right Dini derivative of the conditional expected cost satisfies
\begin{align}\label{bdd:derivative-lower}
	\liminf_{h \searrow 0} \frac{\mathbf{E}[c(X_{t+h}, Y_{t+h}) \mid \boldsymbol{\mathcal{F}}_t] - c(X_t, Y_t)}{h}
	\;\ge\; \omega(X_t, Y_t) \quad \mathbf{P}\text{-a.s.},
\end{align}
with equality holding $\mathbf{P}_*$-a.s. for the optimal coupling. Thus, the drift is minimized at every instant. 
Global optimality, on the other hand, was introduced in \cite{KendallMajkaMijatovic2024} as the standard notion of optimality for immersion couplings, requiring global minimization of conditional expected costs over all future times.

The relationship between local and global optimality is subtle. In many settings, global optimality implies local optimality (though we do not prove this claim here). However, global optimality is a much more restrictive condition, and its existence is not guaranteed in general. For example, \cite{KendallMajkaMijatovic2024} construct a globally $c$-optimal coupling for two identical L\'evy processes on $\mathbb{R}$ with finite activity and unimodal jump distributions, where the cost function $c$ is concave (e.g., $c(x,y) = |x-y|^\gamma$ for $\gamma \in (0,1)$). The coupling they obtain is also locally optimal. Notably, their construction requires the two processes to be identical; the case of two different L\'evy processes is not addressed. For the convex case $c(x,y) = |x-y|^\gamma$ with $\gamma \ge 1$, as pointed out in \cite{KendallMajkaMijatovic2024} for L\'evy processes, the optimal coupling for two identical processes is synchronous (see \cite{JackaMijatovic2015} for the proof in the Brownian motion setting). Beyond these special classes, global optimality may fail to exist due to its strong requirements, whereas local optimality remains a well-defined and often attainable notion. Thus, local optimality is more suitable for studying couplings in broader settings (including non-identical L\'evy processes), as it only requires the minimization of the infinitesimal drift — a condition that is both weaker and more tractable.

\subsubsection{Half-relaxed $c$-transport derivative and local $c$-optimality}
To make the notion of local optimality practical, we need a way to identify the minimal drift function $\omega$ that appears in Definition \ref{def:optimal-immersion}(i). This function should capture the infinitesimal behavior of the optimal transport cost between the time-$t$ marginal laws of the two Markov processes. To this end, we introduce the \emph{half-relaxed transport derivative}, which serves as the canonical candidate for the minimal drift under suitable regularity conditions. 

Recall the setting above, where $c:\Pi\times\tilde\Pi\to\mbr_+$ is a lower semicontinuous cost. 
For any two probability measures $\mu \in \mathcal{P}(\Pi)$ and $\nu \in \mathcal{P}(\tilde\Pi)$, we denote the \emph{$c$-optimal transport cost} by
\[
\mathcal{C}_c(\mu, \nu) := \inf_{\gamma \in \Gamma(\mu, \nu)} \int_{\Pi \times \tilde\Pi} c(x, y) \, d\gamma(x, y),
\]
where $\Gamma(\mu, \nu)$ is the set of couplings between $\mu$ and $\nu$.
For $x \in \Pi$ and $y \in \tilde\Pi$, denote the marginal laws at time $t$ by
\begin{align}\label{def:markov-law}
	\mu_t^x := \operatorname{law}^x(X_t) = (X_t)_\sharp \mathbb{P}^x, \qquad 
	\nu_t^y := \operatorname{law}^y(Y_t) = (Y_t)_\sharp \tilde{\mathbb{P}}^y.
\end{align}
Thus $\mu_t^x \in \mathcal{P}(\Pi)$ and $\nu_t^y \in \mathcal{P}(\tilde\Pi)$ for all $t \ge 0$, $x \in \Pi$, $y \in \tilde\Pi$. The \emph{half-relaxed (lower) $c$-transport derivative} between $\{X_t\}$ and $\{Y_t\}$ is defined by
\begin{align}\label{def:half-relaxed}
	\omega_c^-(x, y) := \liminf_{\substack{t \searrow 0 \\ (x',y') \to (x,y)}} \frac{\mathcal{C}_c(\mu_t^{x'}, \nu_t^{y'}) - c(x',y')}{t}.	
\end{align}
In words, $\omega_c^-(x,y)$ is the limit inferior over all sequences $(t_n, x_n, y_n)$ with $t_n \searrow 0$ and $(x_n, y_n) \to (x,y)$ of the difference quotient.

\begin{remark}
	The notion of \emph{transport derivatives} was first introduced in \cite{LimTeoh2025, KangLim2025}, where the pointwise version (i.e., the limit taken over a fixed $(x,y)$) was studied. Other relevant notions—such as \emph{dual} and \emph{L\'evy} transport derivatives—will be introduced in the coming sections. A key question is whether these different notions coincide under suitable conditions.
\end{remark}

We impose the following assumption on the two Markov processes.

\begin{assumptionC}\label{asp:C}
	There exist nonnegative continuous functions $f: \Pi \to \mathbb{R}_+$ and $g: \tilde\Pi \to \mathbb{R}_+$ such that the following conditions hold:
	\begin{enumerate}[label=(C\arabic*)]
		\item \label{C0} It holds $0\le c(x,y)\le f(x)+g(y)$ for all $(x,y)\in\Pi\times\tilde\Pi$.
		\item \label{C1} \emph{Lipschitz bound:} For all sufficiently small $t \in (0,1]$ and all $(x,y) \in \Pi \times \tilde\Pi$,
		\[
		\frac{|\mathcal{C}_c(\mu_t^x, \nu_t^y) - c(x,y)|}{t} \le f(x) + g(y),
		\]
		where $\mu_t^x,\nu_t^y$ are the laws of the two Markov processes, given in \eqref{def:markov-law}.
	\end{enumerate}
\end{assumptionC}

The following theorem shows that, under Assumption \ref{asp:C} and a supremum integrability condition, the half-relaxed transport derivative $\omega_c^-$ provides a universal lower bound for the infinitesimal drift of the cost process under any immersion coupling. As a consequence, an immersion coupling that attains this bound — i.e., makes the process $\{S_t\}_{t\ge 0}$ a martingale — is locally $c$-optimal.

\begin{theorem}\label{thm:opt-immersion}
	Assume the setting of Definition \ref{def:optimal-immersion} and that Assumption \ref{asp:C} holds. Let $\mu \in \mathcal{P}(\Pi)$ and $\nu \in \mathcal{P}(\tilde\Pi)$ be two initial distributions such that, for every $T>0$,
	\begin{align}\label{eq:sup-integrability}
		\mathbb{E}^\mu\left[\sup_{0\le u\le T} f(X_u)\right]
		+\tilde{\mathbb{E}}^\nu\left[\sup_{0\le u\le T} g(Y_u)\right]<\infty.
	\end{align}
	Then for any immersion coupling $(\boldsymbol{\Omega}, \boldsymbol{\mathcal{F}}, \boldsymbol{\mathfrak{F}}, \mathbf{P})$ of the Markov measures $\mathbb{P}^\mu$ and $\tilde{\mathbb{P}}^\nu$, the following process is a $(\boldsymbol{\mathfrak{F}}, \mathbf{P})$-submartingale:
	\begin{align}\label{def:inf-mart2}
		S_t:= c(X_t, Y_t) - \int_0^t \omega_c^-(X_s, Y_s) \, ds.
	\end{align}
\end{theorem}

\subsection{Optimal immersion couplings of L\'evy processes with respect to the quadratic cost}

The main result of this paper concerns the optimal immersion coupling of two L\'evy processes $\{X_t\}_{t\ge 0}$ and $\{Y_t\}_{t\ge 0}$ on $\mathbb{R}^d$ with respect to the \emph{quadratic cost function} $c: \mathbb{R}^d \times \mathbb{R}^d \to \mathbb{R}_+$,
\begin{align}\label{def:quad-cost}
	c(x,y) := \frac12 |x-y|^2.	
\end{align}

A L\'evy process is a Markov process whose Markov measures are translation-invariant: for any $x \in \mathbb{R}^d$,
\[
\mathbb{P}^x (X_t \in E) = \mathbb{P}^0 (X_t + x \in E),
\]
i.e., the law of the process starting from $x$ is simply the law of the process starting from $0$, shifted by $x$. 
(See Section \ref{sec:2} for more details.)
Consequently, the entire family of Markov measures $\{\mathbb{P}^x\}_{x\in\mathbb{R}^d}$ is uniquely determined by $\mathbb{P}^0$, the law of the process starting at the origin. The same holds for $\tilde{\mathbb{P}}^y$ in terms of $\tilde{\mathbb{P}}^0$.

Our main result establishes the existence of a \emph{L\'evy coupling} for any two (possibly different) L\'evy processes with finite second moments that is both locally and globally $c$-optimal among all immersion couplings. 
As mentioned earlier, the existence of such couplings was first established in \cite{KangLim2025}, where optimality was shown in the generator/semigroup formulation; hence, that coupling is both locally and globally $c$-optimal among all \emph{Markovian (Feller) couplings}. We now extend this optimality to the larger class of all immersion couplings.

\begin{theorem}\label{main}
	Let $\{X_t\}_{t\ge 0}$ and $\{Y_t\}_{t\ge 0}$ be two L\'evy processes on $\mathbb{R}^d$ with finite second moments, defined on filtered probability spaces 
	\(
	(\Omega, \mathcal{F}, \mathfrak{F}, \mathbb{P}^0),
	(\tilde\Omega, \tilde{\mathcal{F}}, \tilde{\mathfrak{F}}, \tilde{\mathbb{P}}^0)
	\)
	respectively. 
	Let $c$ be the quadratic cost defined as in \eqref{def:quad-cost}.
	There exists an immersion coupling $\mathbf{P}_*^{(0,0)}$ of $\mathbb{P}^0$ and $\tilde{\mathbb{P}}^0$ such that:
	\begin{enumerate}[label=(\roman*)]
		\item The coupled process $\{(X_t, Y_t)\}_{t\ge 0}$ is a L\'evy process under $\mathbf{P}_*^{(0,0)}$.
		\item $\mathbf{P}_*^{(0,0)}$ is both locally and globally $c$-optimal.
	\end{enumerate}
\end{theorem}

\begin{remark}
	By the translation invariance of L\'evy processes and Theorem \ref{main}, for any initial distributions $\mu, \nu\in\mcp_2(\mbr^d)$ and any coupling $\gamma \in \Gamma(\mu, \nu)$, the immersion coupling of $\mathbb{P}^\mu$ and $\tilde{\mathbb{P}}^{\nu}$ given by
	\[
	\mathbf{P}_*^{\gamma} := \int_{\mathbb{R}^d \times \mathbb{R}^d} \mathbf{P}_*^{(x,y)} \, \gamma(dx, dy)
	\]
	is both locally and globally $c$-optimal.
\end{remark}

\subsection{Strategy of proofs}

The proof of Theorem~\ref{main} relies on Theorem~\ref{thm:opt-immersion} and two preliminary results established in \cite{KangLim2025}.  
The first is the \emph{existence of a coupling} that is optimal among all \emph{L\'evy couplings}. 
At the level of generators, this asserts that for any pair of L\'evy generators $\mathcal{A}, \mathcal{B}$ with finite second moments, there exists a L\'evy coupling generator $\mathcal{J}_*$ such that
\begin{align}\label{eq:primal-prob}
	(\mathcal{J}_* c)(x,y) = \theta^\Lambda(x,y) := \inf_{\mathcal{J} \in \Gamma^\Lambda(\mathcal{A},\mathcal{B})} (\mathcal{J} c)(x,y), \qquad \text{for all } (x,y) \in \mathbb{R}^{2d},	
\end{align}
where $\Gamma^\Lambda(\mathcal{A},\mathcal{B})$ denotes the set of all L\'evy coupling generators of $(\mathcal{A},\mathcal{B})$ in the sense of \eqref{def:coupling-gen}.
The quantity $\ta^\La$ is called \emph{the L\'evy transport derivative} between two L\'evy generators $\mA,\mB$.

The second is the \emph{strong duality}, which concerns the dual formulation of the minimization problem \eqref{eq:primal-prob}. 
The dual quantity is defined as
\begin{align}\label{eq:dual-prob}
	\omega'(x,y) := \sup_{\varphi,\psi} \Big[ (\mathcal{A} \varphi)(x) + (\mathcal{B} \psi)(y) \Big] ,	
\end{align}
where the supremum is taken over all suitable pairs $(\varphi,\psi)$ of $C^2$-functions such that $\varphi \oplus \psi \le c$ with equality at $(x,y)$. That is, the function $\varphi \oplus \psi$ touches the quadratic cost $c$ from below at $(x,y)$. 
The quantity \eqref{eq:dual-prob} will be called \emph{the dual transport derivative}.
Together with the half-relaxed transport derivative $\om^-=\om_c^-$ from \eqref{def:half-relaxed}, the strong duality asserts that
\[
\theta^\Lambda(x,y) = \omega^-(x,y)= \omega'(x,y), \qquad \text{for all } (x,y) \in \mathbb{R}^{2d}.
\]
These two results are, as shown in \cite{KangLim2025}, essential in establishing minimality among Feller couplings and continue to play a central role in the arguments presented in the current work.

With these two results in hand, the proof of Theorem~\ref{main} proceeds as follows. 
By Theorem \ref{thm:opt-immersion}, to show local $c$-optimality it suffices to show that \eqref{def:inf-mart2} is a martingale under the Markov measure associated with the optimal L\'evy coupling. That is, 
\[
\mathbf{E}_*^{(0,0)}\big[c(X_t, Y_t) \mid \boldsymbol{\mathcal{F}}_s\big] 
= c(X_s, Y_s) + \mathbf{E}_*^{(0,0)}\left[ \int_s^t \omega^-(X_u, Y_u) \, du \;\middle|\; \boldsymbol{\mathcal{F}}_s \right], 
\qquad \text{for all } s \le t.
\]
As a consequence of strong duality, all relevant notions of transport derivatives coincide, i.e.,
\[
\omega' = \omega^- = \theta^\Lambda = \mathcal{J}_* c,
\]
where $\mathcal{J}_*$ is the L\'evy generator of the optimal L\'evy coupling. Substituting $\mathcal{J}_* c$ for $\omega_c^-$ and applying Dynkin's formula yields the martingale property. This establishes that the coupling generated by $\mathcal{J}_*$ is locally optimal among all immersion couplings.
Global optimality then follows from the additive (direct sum) structure of the transport derivative; see Proposition \ref{prop:om-affine}. 
Here, Theorem \ref{thm:opt-immersion} constitutes the main technical ingredient.

Let us briefly comment on the proofs of the two preliminary results in \cite{KangLim2025}. 
The central idea is to use the L\'evy--Khintchine decomposition to split a pair of generators $(\mathcal{A},\mathcal{B})$ into drift, diffusion, and jump components.  
Correspondingly, the minimization and duality problems can be treated separately for each component.  
For the drift–diffusion part, the two results follow directly from the classical theory of optimal transport between Gaussian measures. 
The main challenge lies in the jump part.  
In \cite{KangLim2025}, this is addressed by formulating the \emph{L\'evy optimal transport problem}, an analogue of the classical optimal transport problem, except that one transports between L\'evy measures. 
The authors develop a full theory for this problem, establishing the existence of a minimizer and a Kantorovich-type duality, in parallel to the classical setting.  
A final technical step consists in showing the additivity of the results with respect to the L\'evy--Khintchine decomposition, thereby combining the contributions of drift, diffusion, and jump components.

In the present work, we present a new, simpler approach to proving these two preliminary results that is also of independent interest.  
For the minimization problem among L\'evy couplings, rather than working at the generator level—which concerns the derivative of the law of the process at $t=0$—we consider the minimization problem for the law at $t=1$.  
This reduces the problem to a classical optimal transport problem between \emph{infinitely divisible measures}, with the additional constraint that the couplings are themselves infinitely divisible.  
In this setting, classical methods from optimal transport theory, in particular the direct method of the calculus of variations, can be applied more readily, and the existence of minimizers follows straightforwardly.  
Using basic properties of infinitely divisible measures, the minimality can then be extended from $t=1$ to all times.  
This yields the minimality among all L\'evy couplings in a simple and transparent manner.

Regarding the strong duality, our new approach is motivated by the following \emph{heuristic} argument, which frequently appears in optimization theory.  
Suppose we are given two sets $X$ and $Y$, a function $I:X\times Y \to \mathbb{R}$, and subsets $X_0 \subset X$, $Y_0 \subset Y$ representing points where the infima/suprema are attained.  
Then, under suitable regularity conditions on $I$, one often expects that
\[
\inf_{x\in X_0} H(x) = \inf_{x\in X}\sup_{y\in Y} I(x,y) = \sup_{y\in Y}\inf_{x\in X} I(x,y) = \sup_{y\in Y_0} \tilde H(y),
\]
where $H(x):=\sup_{y\in Y} I(x,y)$ and $\tilde H(y) := \inf_{x\in X} I(x,y)$.  
Heuristically, this expresses the equality between the ``primal'' and ``dual'' optimization problems, and serves as a guiding principle for establishing strong duality in our setting.  
In fact, the minimization problem \eqref{eq:primal-prob} and the maximization problem \eqref{eq:dual-prob} can be cast in this framework.  
Hence, the validity of strong duality ultimately rests on a key step: interchanging the infimum and supremum.  
In classical convex optimization, such interchangeability is often guaranteed by a \emph{minimax principle}.  
In our setting, this interchange is justified by a version of the minimax principle due to Lassonde \cite{Lassonde1983}.

\subsection{Organization}
The remainder of the paper is organized as follows.  
In Section \ref{sec:2}, we establish the existence of optimal L\'evy couplings by solving the optimal transport problem between two infinitely divisible measures. 
In Section \ref{sec:4}, we extend the optimality to the class of all immersion couplings. Specifically, the two main results — Theorems \ref{thm:opt-immersion} and \ref{main} — are proved in this section.
The strong duality used in the proof is established in Section \ref{sec:3}, using Lassonde's minimax principle.

\section{Optimal Transport Between Infinitely Divisible Measures} \label{sec:2}

In this section, we establish the existence of an optimal L\'evy coupling between two L\'evy processes with finite second moments. The construction proceeds in two steps: first, we solve the infinitely divisible optimal transport problem for the time-$1$ marginals; second, we lift the resulting optimal coupling to a full L\'evy process using the convolution semigroup property. In the subsequent section, this optimality will be extended to immersion couplings.

This section is organized as follows.
We first recall the necessary background on infinitely divisible measures, L\'evy processes, and L\'evy generators. In Section \ref{sec:2.2}, we formulate the infinitely divisible optimal transport problem and prove the existence of minimizers. The existence of optimal L\'evy couplings is then obtained as a direct consequence in Section \ref{sec:2.3}. 
Next, the optimality of these couplings at the generator level is established in Section \ref{sec:2-opt-gen}. Finally, we prove the additivity of the transport cost with respect to the L\'evy--Khintchine decomposition.

\subsection{Infinitely divisible measures, L\'evy generators and L\'evy processes}

We now briefly review the necessary background on infinitely divisible measures and L\'evy processes. For a comprehensive treatment, we refer the reader to the monographs \cite{Bertoin1996, Sato1999}.

\subsubsection{Infinitely divisible measures and convolution semigroups}
Let $\mcP(\mathbb{R}^d)$ denote the space of Borel probability measures on $\mathbb{R}^d$, and $\mcP_2(\mathbb{R}^d)$ its subspace of measures with finite second moments. A probability measure $\mu \in \mcP(\mathbb{R}^d)$ is \emph{infinitely divisible} (ID) if for every $n \ge 1$ there exists $\tilde\mu_n \in \mcP(\mathbb{R}^d)$ such that $\mu = \tilde\mu_n^{*n}$, where $*$ denotes convolution. We write $\ID(\mathbb{R}^d)$ for the set of all ID measures, and $\ID_2(\mathbb{R}^d) := \ID(\mathbb{R}^d) \cap \mcP_2(\mathbb{R}^d)$ denotes the set of all ID measures with finite second moments. 

A \emph{convolution semigroup} is a family $\{\mu_t\}_{t\ge 0} \subset \mcP(\mathbb{R}^d)$ such that
\[
\mu_t * \mu_s = \mu_{t+s} \quad \text{for }t,s\ge 0,\qquad \mu_0 = \delta_0,
\]
with $t \mapsto \mu_t$ weakly continuous. Every convolution semigroup consists of ID measures, and conversely, every ID measure $\mu$ uniquely determines a convolution semigroup with $\mu_1 = \mu$. Note also that a convolution semigroup $\{\mu_t\}_{t\ge 0}$ has finite second moments for every $t\ge 0$ if and only if $\mu_1$ has a finite second moment.

Associated with a convolution semigroup $\{\mu_t\}_{t\ge 0}$ is the \emph{L\'evy semigroup} $\{T_t\}_{t\ge 0}$ acting on bounded measurable functions $\varphi:\mbr^d\to\mbr$ by
\[
T_t \varphi(x) := (\mu_t * \varphi)(x) = \int_{\mathbb{R}^d} \varphi(x+z) \, d\mu_t(z).
\]
The family $\{T_t\}_{t\ge 0}$ is a \emph{Feller semigroup} on $C_0(\mathbb{R}^d)$. In particular, it satisfies:
\begin{itemize}
    \item $T_t : C_0(\mathbb{R}^d) \to C_0(\mathbb{R}^d)$ for all $t \ge 0$,
    \item $T_0 = \operatorname{id}$, $T_t T_s = T_{t+s}$,
    \item $\lim_{t \searrow 0} \|T_t \varphi - \varphi\|_\infty = 0$ for every $\varphi \in C_0(\mathbb{R}^d)$,
    \item $T_t \mathbf{1} = \mathbf{1}$ (conservation of mass),
    \item $T_t \varphi \ge 0$ whenever $\varphi \ge 0$ (positivity preserving).
\end{itemize}
By duality, $\{T_t\}_{t\ge 0}$ also acts on probability measures $\mu \in \mathcal{P}(\mathbb{R}^d)$ via
\[
\int_{\mathbb{R}^d} \varphi \, d(\mu T_t)= \inn{\mu T_t,\vphi}=\inn{\mu,T_t\vphi} = \int_{\mathbb{R}^d} T_t \varphi \, d\mu, \qquad \varphi \in C_b(\mathbb{R}^d).
\]
In particular, if $\mu = \delta_x$, then $\delta_x T_t = \mu_t(\cdot - x)$.

Every L\'evy semigroup is \emph{translation invariant}, i.e., its operators commute with translation operators.  
Conversely, every translation-invariant Feller semigroup on $C_0(\mathbb{R}^d)$ is a L\'evy semigroup; that is, it arises from a convolution semigroup $\{\mu_t\}_{t\ge 0}$.

\subsubsection{L\'evy exponent, triplets, and generators}\label{sec:3.1.2}
Given an ID measure $\mu \in \ID(\mathbb{R}^d)$, its \emph{L\'evy exponent} is defined as the logarithm of the characteristic function $\hat\mu$ of $\mu$:
\begin{align}\label{eq:levy-char}
    \Psi_\mu(\xi) := \log \hat\mu(\xi), \qquad \hat\mu(\xi) = \int_{\mathbb{R}^d} e^{i x^\top \xi} \, \mu(dx).
\end{align}
By the \emph{L\'evy--Khintchine representation}, the L\'evy exponent of an ID measure takes the form
\begin{align}\label{eq:LK-rep}
    \Psi_\mu(\xi) = i\kappa^\top \xi - \frac12 \xi^\top \alpha \xi + \int_{\mathbb{R}^d \setminus \{0\}} \bigl( e^{i\xi^\top z} - 1 - i\xi^\top z \, b(z) \bigr) \, d\Theta(z),
\end{align}
where
\begin{itemize}
    \item $\kappa \in \mathbb{R}^d$ is the \emph{drift},
    \item $\alpha \in \mathcal{S}_{\ge 0}(\mathbb{R}^d)$\footnote{We denote by $\mathcal{S}_{\ge 0}(\mathbb{R}^d)$ the set of all nonnegative definite $(d \times d)$-matrices, and by $\mathcal{S}_{>0}(\mathbb{R}^d) \subset \mathcal{S}_{\ge 0}(\mathbb{R}^d)$ the subset of positive definite matrices.} is the \emph{Gaussian covariance matrix},
    \item $\Theta$ is the \emph{L\'evy measure}, satisfying 
    \begin{align}\label{bdd:levy-meas}
        \int_{\mbr^d\setminus\{0\}}  \min(|z|^2, 1) \, d\Theta(z) < \infty,   
    \end{align}
    \item $b: \mathbb{R}^d \to [0,1]$ is a \emph{cutoff function} satisfying $b\equiv 1$ in a neighborhood of the origin, and $b(z) = 0$ for large $|z|$. A standard choice is $b(z) = \mathbf{1}_{|z| \le 1}$. 
\end{itemize}
The triple $(\kappa, \alpha, \Theta)_b$ is called the \emph{L\'evy triplet with respect to the cutoff function $b$}. For a fixed cutoff function $b$, the L\'evy triplet uniquely determines the infinitely divisible measure $\mu$, and conversely, $\mu$ uniquely determines the triplet $(\kappa, \alpha, \Theta)_b$. Different choices of $b$ lead to different drifts $\kappa$ but leave $\alpha$ and $\Theta$ unchanged; the triplet $(\kappa, \alpha, \Theta)$ is therefore often understood up to the chosen truncation convention. 

In this work, we focus on ID measures $\mu \in \ID_2(\mathbb{R}^d)$ with finite second moments. For such measures, in addition to \eqref{bdd:levy-meas}, the L\'evy measure $\Theta$ satisfies the stronger integrability condition
\begin{align}\label{eq:second-moment}
    \int_{\mathbb{R}^d \setminus \{0\}} |z|^2 \, d\Theta(z) < \infty.    
\end{align}
In this case, the integral in \eqref{eq:LK-rep} against the L\'evy measure $\Theta$ converges absolutely without any cutoff. Thus, one may adopt the convention $b(z) \equiv 1$, leading to the \emph{untruncated} L\'evy--Khintchine representation
\begin{align}\label{eq:ULK-rep}
    \Psi_\mu(\xi) &= i\tilde\kappa^\top \xi - \frac12 \xi^\top \alpha \xi + \int_{\mathbb{R}^d \setminus \{0\}} \bigl( e^{i\xi^\top z} - 1 - i\xi^\top z \bigr) \, d\Theta(z)\\
    &=: \Psi_\nabla(\xi) + \Psi_\Delta(\xi) + \Psi_J(\xi),\nonumber
\end{align}
where, for the standard cutoff $b(z)=\mathbf{1}_{|z|\le 1}$, $\tilde\kappa = \kappa + \int_{|z|>1} z \, d\Theta(z)$ is the modified drift that absorbs the truncation. 
Here $\Psi_\nabla$, $\Psi_\Delta$, $\Psi_J$ are called the \emph{drift}, \emph{diffusion}, and \emph{jump parts} of the L\'evy exponent, respectively.
This representation simplifies the analysis and \emph{will be our default convention throughout the paper for all measures with finite second moments.}

Following the decomposition \eqref{eq:ULK-rep}, an ID measure $\mu \in \ID_2(\mathbb{R}^d)$ can be decomposed into a convolution of three independent ID measures:
\begin{align}\label{eq:LK-decomp}
    \mu = \mu_\nabla * \mu_\Delta * \mu_J,
\end{align}
where $\mu_\nabla$, $\mu_\Delta$, $\mu_J$ correspond to the drift, diffusion, and jump parts with L\'evy exponents $\Psi_\nabla$, $\Psi_\Delta$, $\Psi_J$, respectively. Concretely, $\mu_\nabla = \delta_{\tilde\kappa}$ (a Dirac mass at $\tilde\kappa$), $\mu_\Delta$ is the centered Gaussian measure with covariance $\alpha$, and $\mu_J$ is a centered pure-jump ID measure with L\'evy measure $\Theta$. We shall call this the \emph{global L\'evy--Khintchine decomposition}.

\begin{remark}
    The global decomposition is defined with respect to the untruncated representation $b \equiv 1$, which is our default convention for measures with finite second moments. Under this convention, the decomposition is unique and the components are independent. If a different cutoff function $b$ were used, the decomposition would involve a different splitting of the drift between $\mu_\nabla$ and $\mu_J$, and the term ``global decomposition'' would not apply.
\end{remark}

Related to this is the notion of \emph{L\'evy generators}. 
Given an ID measure $\mu$ that generates the convolution semigroup $\{\mu_t\}_{t\ge 0}$ and the L\'evy semigroup $\{T_t\}_{t\ge 0}$, the associated generator is the operator $\mathcal{A}: C_b^2(\mathbb{R}^d) \to C_b(\mathbb{R}^d)$ given by
\begin{align}\label{eq:generator}
    \mathcal{A} \varphi(x) := \lim_{t \searrow 0} \frac{T_t \varphi(x) - \varphi(x)}{t} = \lim_{t\searrow0} \frac{(\mu_t * \varphi)(x) - \varphi(x)}{t}.    
\end{align}
We denote by $\mathcal{G}^{\Lambda}(\mathbb{R}^d)$ the class of all \emph{L\'evy generators} arising in this way, and $\mathcal{G}_2^{\Lambda}(\mathbb{R}^d) \subset \mathcal{G}^{\Lambda}(\mathbb{R}^d)$ the subclass for which $\mu_1$ admits a finite second moment.
By the L\'evy--Khintchine representation with the triplet $(\kappa, \alpha, \Theta)_b$, the generator admits the form
\begin{align}\label{eq:LK-rep-op1}
    \mathcal{A} \varphi(x) = \kappa^\top \nabla \varphi(x) + \frac12 \operatorname{tr}\bigl(\alpha \nabla^2 \varphi(x)\bigr) + \int_{\mathbb{R}^d \setminus \{0\}} \bigl( \varphi(x+z) - \varphi(x) - z^\top \nabla \varphi(x) b(z) \bigr) \, d\Theta(z).
\end{align}

Following our convention, if $\mathcal{A} \in \mathcal{G}_2^{\Lambda}(\mathbb{R}^d)$, the corresponding ID measure $\mu$ and L\'evy measure $\Ta$ have finite second moments. In this case, we may take $b \equiv 1$, yielding the \emph{untruncated} representation
\begin{align}\label{eq:LK-rep-op}
    \mathcal{A} \varphi(x) &= \tilde\kappa^\top \nabla \varphi(x) + \frac12 \operatorname{tr}\bigl(\alpha \nabla^2 \varphi(x)\bigr) + \int_{\mathbb{R}^d \setminus \{0\}} \bigl( \varphi(x+z) - \varphi(x) - z^\top \nabla \varphi(x) \bigr) \, d\Theta(z)\\
    &=: \mathcal{A}_\nabla \varphi(x) + \mathcal{A}_\Delta \varphi(x) + \mathcal{A}_J \varphi(x). \nonumber
\end{align}
The operators $\mathcal{A}_\nabla$, $\mathcal{A}_\Delta$, $\mathcal{A}_J$ are the \emph{drift}, \emph{diffusion}, and \emph{jump} parts of the generator $\mathcal{A}$ under the global L\'evy--Khintchine decomposition. 
As mentioned, throughout this paper, \emph{this will be the default representation for generators $\mathcal{A} \in \mathcal{G}_2^{\Lambda}(\mathbb{R}^d)$}.

Conversely, for any operator $\mathcal{A}: C_b^2(\mathbb{R}^d) \to C_b(\mathbb{R}^d)$ of the form \eqref{eq:LK-rep-op1}, there exists a L\'evy semigroup $\{T_t\}_{t\ge 0}$ (and hence an ID measure) such that \eqref{eq:generator} holds, and consequently $\mathcal{A} \in \mathcal{G}^{\Lambda}(\mathbb{R}^d)$. 
One conventionally writes $T_t = e^{t\mathcal{A}}$ to indicate this generator--semigroup relation.
Likewise, if $\mathcal{A}$ takes the form \eqref{eq:LK-rep-op} with the L\'evy measure $\Theta$ satisfying the finite second moment condition \eqref{eq:second-moment},
then $\mathcal{A} \in \mathcal{G}_2^{\Lambda}(\mathbb{R}^d)$.

For any L\'evy generator $\mathcal{A} \in \mathcal{G}_2^\Lambda(\mathbb{R}^d)$, the representation \eqref{eq:LK-rep-op} extends naturally to the larger domain $C_2^2(\mathbb{R}^d)$ of twice continuously differentiable functions satisfying the growth bounds
\[
|\varphi(x)| \leq C(1+|x|^2), \qquad \|D^2\varphi(x)\| \leq C, \qquad x \in \mathbb{R}^d,
\]
for some constant $C \geq 0$ (depending on $\varphi$). Indeed, the integral term in \eqref{eq:LK-rep-op} is absolutely convergent for such $\varphi$ thanks to the finite second-moment condition \eqref{eq:second-moment}. Moreover, for $\varphi \in C_2^2(\mathbb{R}^d)$, the function $\mathcal{A}\varphi$ is continuous and has at most linear growth: $|\mathcal{A}\varphi(x)| \le C_\varphi(1+|x|)$ for all $x\in\mathbb{R}^d$ and some constant $C_\varphi\ge0$. The semigroup identity
\[
\lim_{t \searrow 0} \frac{e^{t\mathcal{A}}\varphi(x) - \varphi(x)}{t} = \mathcal{A}\varphi(x)
\]
holds pointwise (and locally uniformly in $x$) for every $x \in \mathbb{R}^d$.

\subsubsection{L\'evy processes}\label{sec:2.1.3}
Let $(\Omega, \mathcal{F}, \mathfrak{F})$ be the canonical path space on the state space $\mathbb{R}^d$. That is, $\Omega = D(\mathbb{R}_+; \mathbb{R}^d)$ is the space of c\`adl\`ag functions from $\mathbb{R}_+$ to $\mathbb{R}^d$, equipped with the $\sigma$-algebra $\mathcal{F}$ generated by the coordinate projections and the natural filtration $\mathfrak{F} = \{\mathcal{F}_t\}_{t\ge 0}$. The coordinate process is defined by $X_t(\omega) = \omega(t)\in\mbr^d$ for $t \ge 0$, $\omega \in \Omega$.
A \emph{L\'evy process} is a probability measure $(\mathbb{P}^0, \mathbb{E}^0)$ on $(\Omega, \mathcal{F}, \mathfrak{F})$ such that the coordinate process $\{X_t\}_{t\ge 0}$ satisfies:
\begin{itemize}
	\item $X_0 = 0$ $\mathbb{P}^0$-almost surely,
	\item \emph{Independent increments}: for any $0 \le t_1 < t_2 < \dots < t_n$, the random variables $X_{t_1}, X_{t_2} - X_{t_1}, \dots, X_{t_n} - X_{t_{n-1}}$ are independent under $\mathbb{P}^0$,
	\item \emph{Stationary increments}: for any $0 \le s \le t$, the distribution of $X_t - X_s$ depends only on $t-s$,
	\item \emph{Stochastic continuity}: for every $t \ge 0$ and $\varepsilon > 0$, $\lim_{s \to t} \mathbb{P}^0(|X_t - X_s| > \varepsilon) = 0$.
\end{itemize}
The marginal laws $\{\mu_t = \operatorname{law}(X_t)\}_{t\ge 0}$ form a convolution semigroup of measures on $\mbr^d$, and hence $\mu_1$ is infinitely divisible. Conversely, every convolution semigroup, and hence ID measure, uniquely determines the law of a L\'evy process.

A L\'evy process $\mathbb{P}^0$ induces a family of Markov measures $\{\mathbb{P}^x\}_{x\in\mathbb{R}^d}$ by spatial translation: $\mathbb{P}^x := (X_t + x)_\sharp \mathbb{P}^0$. This family defines a time-homogeneous Markov process with translation-invariant semigroup, and each $\mathbb{P}^x$ is the law of the process starting from $x$.

Throughout this work, we consider L\'evy processes with finite second moments, i.e., $\mathbb{E}^0[|X_t|^2] < \infty$ for all $t \ge 0$ (equivalently, for some $t > 0$, e.g., $t = 1$).
At the process level, the finite second moment condition yields the (untruncated) \emph{L\'evy--It\^o decomposition}, expressing $\{X_t\}_{t\ge 0}$ directly as the sum of three \emph{mutually independent} processes:
\begin{align*}
	X_t &= X_t^\nabla +X_t^\De + X_t^J \\
	&:= \tilde\kappa t + \sigma B_t + \int_0^t \int_{\mathbb{R}^d \setminus \{0\}} z \, \bigl(N(ds, dz) - ds \, d\Theta(z)\bigr).
\end{align*}
Here, $\{B_t\}_{t\ge 0}$ is a standard Brownian motion ($\sigma\sigma^\top = \alpha$), and $N$ is an independent Poisson random measure with intensity $ds \, d\Theta(z)$. These independent pathwise components correspond precisely to the drift, diffusion, and jump operators $(\mathcal{A}_\nabla, \mathcal{A}_\Delta, \mathcal{A}_J)$ of $\mathcal{A}$ in \eqref{eq:LK-rep-op}.

\subsubsection{One-to-one correspondence}
In summary, from the classical theory of L\'evy processes, the following objects are in one-to-one correspondence:
\begin{itemize}
    \item Infinitely divisible measures $\mu \in \ID(\mathbb{R}^d)$,
    \item Convolution semigroups $\{\mu_t\}_{t\ge 0}$ with $\mu_1 = \mu$,
    \item L\'evy semigroups $\{T_t\}_{t\ge 0} = \{e^{t\mathcal{A}}\}_{t\ge 0}$,
    \item L\'evy exponents $\Psi_\mu$,
    \item L\'evy generators $\mathcal{A} \in \mathcal{G}^{\Lambda}(\mathbb{R}^d)$,
    \item L\'evy triplets $(\kappa, \alpha, \Theta)$,
    \item L\'evy processes $\{X_t\}_{t\ge 0}$ (in law).
\end{itemize}
The same one-to-one correspondence holds under the finite second-moment assumption, i.e., for $\mu \in \ID_2(\mathbb{R}^d)$ and $\mathcal{A} \in \mathcal{G}_2^{\Lambda}(\mathbb{R}^d)$. 
Throughout this paper we will freely identify these objects.

\subsubsection{Quadratic moments of infinitely divisible measures}

It is a well-known fact that the \emph{quadratic moments} of a convolution semigroup of infinitely divisible measures are determined by the mean vector and covariance matrix of the measure at time $1$. In this subsection, we collect the relevant preliminaries for later use.

Let $\{\gamma_t\}_{t\ge 0}$ be a convolution semigroup on $\mathbb{R}^d$ with $\gamma_1 \in \mathcal{P}_2(\mathbb{R}^d)$. The \emph{mean vector} $m \in \mathbb{R}^d$ and \emph{covariance matrix} $\Sigma \in \mathcal{S}_{\ge 0}(\mathbb{R}^d)$ of $\gamma_1$ are defined by
\begin{align}\label{def:mean-vector}
    m := \int_{\mathbb{R}^d} x \, d\gamma_1(x), \qquad 
\Sigma := \int_{\mathbb{R}^d} (x - m)(x - m)^\top \, d\gamma_1(x).    
\end{align}
We shall use the notations $m_\mu$, $m_X$, $m_\mathcal{A}$ and $\Sigma_\mu$, $\Sigma_X$, $\Sigma_\mathcal{A}$ interchangeably to denote the mean and covariance of an ID measure, a L\'evy process, or its generator; all refer to the first two moments of the time-$1$ marginal.

For a symmetric matrix $Q \in \mathcal{S}(\mathbb{R}^d)$, we denote by $\varphi_Q(x) := x^\top Q x$ the associated quadratic form. The following lemma expresses the expectation of $\varphi_Q$ under $\gamma_t$ in terms of the mean vector $m$, covariance matrix $\Sigma$, and time $t\ge 0$.

\begin{lemma}\label{lem:quad-linear}
	Let $\{\gamma_t\}_{t\ge 0}$ be a convolution semigroup on $\mathbb{R}^d$ with $\gamma_1\in\mathcal{P}_2(\mathbb{R}^d)$. Let $m\in\mathbb{R}^d$ and $\Sigma\in\mathcal{S}_{\ge 0}(\mathbb{R}^d)$ be as in \eqref{def:mean-vector}. Let $Q\in\mathcal{S}(\mathbb{R}^d)$ be a symmetric matrix, and denote by $\varphi_Q(x)=x^\top Q x$ the associated quadratic form. Then, for all $t\ge 0$,
	\begin{align}
		\int_{\mathbb{R}^d} \varphi_Q(x) \, d\gamma_t(x) 
        &= t \, \operatorname{tr}(Q\Sigma) + t^2 \, \varphi_Q(m) 
		= t \int_{\mathbb{R}^d} \varphi_Q(x) \, d\gamma_1(x) + (t^2 - t) \varphi_Q(m). \label{eq:quad-moment-2}
	\end{align}
\end{lemma}

\begin{proof}
    Let $\{X_t\}_{t\ge 0}$ be the L\'evy process associated with the convolution semigroup $\{\gamma_t\}_{t\ge 0}$; recall that $\{\gamma_t\}_{t\ge 0}$ has finite second moments for every $t \ge 0$, since $\gamma_1 \in \mathcal{P}_2(\mathbb{R}^d)$.
    We derive the scaling of the first two moments from the characteristic functions. Since $\gamma_1$ is infinitely divisible, its characteristic function $\hat\gamma_1$ is nowhere zero \cite[Lemma 7.5]{Sato1999}, and since $\gamma_1$ has a finite second moment, $\hat\gamma_1$ is twice continuously differentiable. Hence the L\'evy exponent $\Psi_{\gamma_1} = \log \hat\gamma_1$ from \eqref{eq:levy-char} (taken with the distinguished logarithm, so that $\Psi_{\gamma_1}(0) = 0$) is $C^2$ in a neighborhood of the origin, and $\hat\gamma_t = e^{t \Psi_{\gamma_1}}$ for all $t \ge 0$. Differentiating twice at $\xi = 0$,
    \begin{align*}
        \nabla \hat\gamma_t(0) = t\, \nabla \Psi_{\gamma_1}(0), \qquad
        D^2 \hat\gamma_t(0) = t\, D^2 \Psi_{\gamma_1}(0) + t^2\, \nabla \Psi_{\gamma_1}(0)\, \nabla \Psi_{\gamma_1}(0)^\top.
    \end{align*}
    Recalling that $\mathbb{E}[X_t] = -i\, \nabla \hat\gamma_t(0)$ and $\mathbb{E}[X_t X_t^\top] = -D^2 \hat\gamma_t(0)$, the case $t = 1$ identifies $\nabla \Psi_{\gamma_1}(0) = i m$ and $D^2 \Psi_{\gamma_1}(0) = -\Sigma$; substituting back, for every $t \ge 0$,
    \begin{align*}
        \mathbb{E}[X_t] = t\, m, \qquad \mathbb{E}[X_t X_t^\top] = t\, \Sigma + t^2\, m m^\top.
    \end{align*}

    Evaluating the expectation of the quadratic form via the trace operator gives
    \begin{align*}
        \int_{\mathbb{R}^d} x^\top Q x \, d\gamma_t(x)
        &= \mathbb{E}\bigl[X_t^\top Q X_t\bigr]
        = \operatorname{tr}\bigl(Q \, \mathbb{E}[X_t X_t^\top]\bigr)
        = \operatorname{tr}\bigl(Q \bigl(t\Sigma + t^2 m m^\top\bigr)\bigr).
    \end{align*}
    Applying the linearity of the trace and using $\Sigma = \int_{\mathbb{R}^d} x x^\top \, d\gamma_1(x) - m m^\top$ yields
    \begin{align*}
        \int_{\mathbb{R}^d} \varphi_Q(x) \, d\gamma_t(x) 
        &= t \biggl( \int_{\mathbb{R}^d} \varphi_Q(x) \, d\gamma_1(x) - m^\top Q m \biggr) + t^2 m^\top Q m \\
        &= t \int_{\mathbb{R}^d} \varphi_Q(x) \, d\gamma_1(x) + (t^2 - t) m^\top Q m.
    \end{align*}
    This completes the proof.
\end{proof}

\subsection{Optimal coupling problem for infinitely divisible measures}\label{sec:2.2}

We now formulate the central problem of this section, a variant of the classical optimal transport problem. In the classical setting, given two probability measures $\mu,\nu\in\mcP(\mathbb{R}^d)$, one seeks a coupling measure $\gamma\in\mcP(\mathbb{R}^{2d})$ minimizing the \emph{quadratic} transport cost:
\[
\langle \gamma, c\rangle := \int_{\mathbb{R}^{2d}} c(x,y)\,d\gamma(x,y), \qquad c(x,y):=\frac12 |x-y|^2.
\]
In our setting, we consider the same problem with $\mu,\nu$ infinitely divisible, but with the additional constraint that the coupling $\gamma$ is also infinitely divisible. We call this the \emph{infinitely divisible optimal transport problem}.

\begin{definition}[Infinitely divisible couplings, transport cost, and optimality]
	Let $\mu,\nu \in \mathrm{ID}(\mathbb{R}^d)$ and $c:\mbr^d\times\mbr^d\to\mbr_+$ be the quadratic cost $c(x,y)=\f 12 |x-y|^2$.
	\begin{enumerate}[label=(\roman*)]
		\item An \emph{infinitely divisible (ID) coupling} of $\mu$ and $\nu$ is an ID measure $\gamma \in \mathrm{ID}(\mathbb{R}^{2d})$ such that for all measurable $E \subset \mathbb{R}^d$,
		\[
		\gamma(E \times \mathbb{R}^d) = \mu(E), \qquad 
		\gamma(\mathbb{R}^d \times E) = \nu(E).
		\]
		We denote by $\Gamma^{\mathrm{ID}}(\mu,\nu)$ the set of all ID couplings of $\mu$ and $\nu$.
		
		\item Given $\mu,\nu \in \mathrm{ID}(\mathbb{R}^d)$, the \emph{infinitely divisible quadratic transport cost} between $\mu$ and $\nu$ is defined as
		\[
		\mathcal{C}_2^{\mathrm{ID}}(\mu,\nu) 
		:= \inf_{\gamma \in \Gamma^{\mathrm{ID}}(\mu,\nu)} 
		\int_{\mathbb{R}^{2d}} c(x,y) \, d\gamma(x,y).
		\]
		
		\item An ID coupling $\gamma_* \in \Gamma^{\mathrm{ID}}(\mu,\nu)$ is called \emph{infinitely divisible optimal}, or \emph{ID-optimal}, if it attains the infimum, i.e.,
		\[
		\int_{\mathbb{R}^{2d}} c(x,y) \, d\gamma_*(x,y) = \mathcal{C}_2^{\mathrm{ID}}(\mu,\nu).
		\]
	\end{enumerate}
\end{definition}

\begin{remark}
	The set $\Gamma^{\mathrm{ID}}(\mu,\nu)$ is nonempty; for instance, the product measure $\mu\otimes\nu$ belongs to $\Gamma^{\mathrm{ID}}(\mu,\nu)$. Consequently, the cost $\mathcal{C}_2^{\mathrm{ID}}(\mu,\nu)$ is finite whenever $\mu,\nu\in\ID_2(\mbr^d)$.
\end{remark}

Let us begin by establishing a simple characterization of ID couplings in terms of their L\'evy exponents.
\begin{lemma}\label{lem:marg-char}
    Let $\mu,\nu \in \ID(\mathbb{R}^d)$ and $\gamma \in \ID(\mathbb{R}^{2d})$. Then $\gamma \in \Gamma^{\mathrm{ID}}(\mu,\nu)$ if and only if their L\'evy exponents satisfy for all $\xi,\eta \in \mathbb{R}^d$:
    \begin{align*}
        \Psi_\gamma(\xi,0) = \Psi_\mu(\xi), \qquad \Psi_\gamma(0,\eta) = \Psi_\nu(\eta).
    \end{align*}
\end{lemma}

\begin{proof}
	Let $\gamma \in \Gamma^{\mathrm{ID}}(\mu,\nu)$. The characteristic function of the first marginal $\gamma_1 = \mu$ is
	\begin{align*}
		\hat\mu(\xi) = \int_{\mathbb{R}^d} e^{i\xi^\top x} \, d\mu(x) = \int_{\mathbb{R}^{2d}} e^{i\xi^\top x} \, d\gamma(x,y) = \hat\gamma(\xi,0).
	\end{align*}
	Similarly, $\hat\nu(\eta) = \hat\gamma(0,\eta)$. Taking the continuous logarithms normalized to vanish at the origin gives $\Psi_\mu(\xi) = \Psi_\gamma(\xi,0)$, and likewise $\Psi_\nu(\eta) = \Psi_\gamma(0,\eta)$. Conversely, if these equalities hold for the characteristic exponents, then $\hat\gamma(\xi,0) = \hat\mu(\xi)$ and $\hat\gamma(0,\eta) = \hat\nu(\eta)$, which implies that $\gamma$ has marginals $\mu$ and $\nu$. Since $\gamma$ is infinitely divisible, we conclude that $\gamma \in \Gamma^{\mathrm{ID}}(\mu,\nu)$. This completes the proof.
\end{proof}

\subsubsection{Existence of ID-optimal couplings}
The \emph{infinitely divisible optimal transport problem} therefore asks: given two ID measures $\mu,\nu\in\mcP(\mathbb{R}^d)$, does an ID-optimal coupling exist? As a first step, we provide an affirmative answer under a suitable second-moment assumption.

\begin{proposition}\label{prop:minimizer}
	Let $\mu, \nu \in \ID_2(\mathbb{R}^d)$ be two infinitely divisible measures with finite second moments. Then there exists an ID-optimal coupling $\gamma_*$ of $\mu$ and $\nu$, and moreover $\gamma_*$ itself has finite second moment.
\end{proposition}

The proof follows the same lines as in the classical optimal transport problem, where one constructs a minimizing sequence and extracts a convergent subsequence using tightness arguments. The only additional technical point is to ensure that the limiting coupling remains infinitely divisible; this follows from the fact that the class of infinitely divisible measures is closed under weak convergence.

\begin{proof}
    Set $c(x,y)=\frac12 |x-y|^2$. 
    We first note that $\mathcal{C}_2^{\mathrm{ID}}(\mu,\nu)$ is finite. 
    Indeed, since $c(x,y)\le |x|^2+|y|^2$ and $\mu,\nu$ have finite second moments,
    for any $\gamma\in\Gamma^{\mathrm{ID}}(\mu,\nu)$ we have 
    \[
     \langle \gamma, c\rangle \le \int |x|^2\,d\mu(x) + \int |y|^2\,d\nu(y) < \infty.
    \]

    Let $\{\gamma_n\}_{n\ge 1}\subset \Gamma^{\mathrm{ID}}(\mu,\nu)$ be a minimizing sequence, so that 
    $\langle \gamma_n, c\rangle \searrow \mathcal{C}_2^{\mathrm{ID}}(\mu,\nu)$. 
    The above bound yields uniform second-moment control:
    \[
    \sup_n \int \big(|x|^2 + |y|^2\big)\,d\gamma_n(x,y) < \infty.
    \]
    Hence $\{\gamma_n\}$ is tight in $\mathcal{P}(\mathbb{R}^{2d})$. By Prokhorov's theorem, there exists a subsequence (still denoted $\gamma_n$) and a probability measure $\gamma_*\in\mathcal{P}(\mathbb{R}^{2d})$ such that $\gamma_n \to \gamma_*$ weakly.

    We now show that the limit $\gamma_*$ belongs to $\Gamma^{\mathrm{ID}}(\mu,\nu)$. 
    First, $\gamma_*$ is a coupling of $\mu$ and $\nu$. Indeed, for any bounded continuous function $\varphi \in C_b(\mathbb{R}^d)$, weak convergence $\gamma_n \to \gamma_*$ gives
    \[
    \int_{\mathbb{R}^{2d}} \varphi(x)\,d\gamma_*(x,y)
    = \lim_{n\to\infty} \int_{\mathbb{R}^{2d}} \varphi(x)\,d\gamma_n(x,y)
    = \int_{\mathbb{R}^d} \varphi(x)\,d\mu(x),
    \]
    and analogously for the second marginal. Therefore $\gamma_* \in \Gamma(\mu,\nu)$.

    The infinite divisibility of $\gamma_*$ follows from the fact that the class of infinitely divisible measures is closed under weak convergence \cite[Lemma~7.8]{Sato1999}.
    Moreover, since $\{\gamma_n\}$ has uniformly bounded second moments and $c(x,y)$ is lower semicontinuous, the Portmanteau lemma implies
    \[
    \int_{\mathbb{R}^{2d}} (|x|^2+|y|^2)\,d\gamma_*(x,y)
    \le \liminf_{n\to\infty} \int_{\mathbb{R}^{2d}} (|x|^2+|y|^2)\,d\gamma_n(x,y) < \infty.
    \]
    Thus $\gamma_*$ has finite second moment, and consequently $\gamma_* \in \Gamma^{\mathrm{ID}}(\mu,\nu)$.

    Finally, applying the Portmanteau lemma again yields
    \[
    \langle \gamma_*, c \rangle \le \liminf_{n\to\infty} \langle \gamma_n, c \rangle = \mathcal{C}_2^{\mathrm{ID}}(\mu,\nu).
    \]
    Hence $\gamma_*$ attains the infimum and is therefore an ID-optimal coupling.
\end{proof}

\subsubsection{Stability of ID couplings under L\'evy flows}
The following proposition establishes that the ID-optimal coupling is stable under the action of the associated L\'evy flows. More precisely, if two convolution semigroups $\{\mu_t\}_{t\ge 0}$ and $\{\nu_t\}_{t\ge 0}$ are coupled via a convolution semigroup $\{\gamma_t\}_{t\ge 0}$ on $\mathbb{R}^{2d}$, then the optimality property at a single time $t=1$ propagates to all times $t\ge 0$.

\begin{proposition}[Stability under L\'evy flows]\label{prop:stab}
    Let $\{\mu_t\}_{t\ge 0}$ and $\{\nu_t\}_{t\ge 0}$ be convolution semigroups of probability measures on $\mathbb{R}^d$ with $\mu_1,\nu_1\in\mcp_2(\mbr^d)$, and let $\{\gamma_t\}_{t\ge 0}$ be a convolution semigroup on $\mathbb{R}^{2d}$ such that for each $t\ge 0$ (equivalently, for some $t>0$), $\gamma_t$ is a coupling of $\mu_t$ and $\nu_t$. If $\gamma_1$ is an ID-optimal coupling of $\mu_1$ and $\nu_1$, then $\gamma_t$ is an ID-optimal coupling of $\mu_t$ and $\nu_t$ for every $t \ge 0$.
 	Moreover, the optimal cost satisfies
 	\[
 		\mathcal{C}_2^{\mathrm{ID}}(\mu_t,\nu_t) = t \, \mathcal{C}_2^{\mathrm{ID}}(\mu_1,\nu_1) + (t^2 - t) \, c(m_\mu, m_\nu),
 	\]
    where $m_\mu, m_\nu \in \mathbb{R}^d$ are the mean vectors of the convolution semigroups $\{\mu_t\}_{t\ge 0}$ and $\{\nu_t\}_{t\ge 0}$, as defined in \eqref{def:mean-vector}.
\end{proposition}

\begin{proof}
Since $\mu_0 = \nu_0 = \delta_0$ and $\gamma_0 = \delta_{(0,0)}$, the conclusion holds trivially for $t = 0$.

Consider $t > 0$. Assume, towards a contradiction, that there exists $t_0 > 0$ for which $\gamma_{t_0}$ is not optimal. Then there exists an infinitely divisible coupling
$\tilde\gamma_{t_0} \in \Gamma^{\mathrm{ID}}(\mu_{t_0}, \nu_{t_0})$
such that 
\begin{equation}\label{eq:contradiction-t0}
 \int_{\mathbb{R}^{2d}} c(x,y) \, d\tilde\gamma_{t_0}(x,y)
<
 \int_{\mathbb{R}^{2d}} c(x,y) \, d\gamma_{t_0}(x,y).
\end{equation}

Since $\tilde\gamma_{t_0} \in \ID(\mbr^{2d})$, it embeds into a unique convolution semigroup $\{\tilde\gamma_t\}_{t\ge 0}$ on $\mathbb{R}^{2d}$ whose distribution at time $t_0$ is $\tilde\gamma_{t_0}$; in terms of L\'evy exponents, $\Psi_{\tilde\gamma_t} = (t/t_0)\, \Psi_{\tilde\gamma_{t_0}}$ for all $t \ge 0$. We claim that $\tilde\gamma_t \in \Gamma^{\mathrm{ID}}(\mu_t, \nu_t)$ for every $t \ge 0$. Indeed, since $\tilde\gamma_{t_0} \in \Gamma^{\mathrm{ID}}(\mu_{t_0}, \nu_{t_0})$, Lemma \ref{lem:marg-char} gives $\Psi_{\tilde\gamma_{t_0}}(\xi,0) = \Psi_{\mu_{t_0}}(\xi) = t_0\, \Psi_{\mu_1}(\xi)$, and hence
\[
\Psi_{\tilde\gamma_t}(\xi,0) = \frac{t}{t_0}\, \Psi_{\tilde\gamma_{t_0}}(\xi,0) = t\, \Psi_{\mu_1}(\xi) = \Psi_{\mu_t}(\xi), \qquad \xi \in \mathbb{R}^d,
\]
with the symmetric identity in $\eta$ for the second marginal. By Lemma \ref{lem:marg-char} again, $\tilde\gamma_t \in \Gamma^{\mathrm{ID}}(\mu_t, \nu_t)$ for all $t \ge 0$; in particular, $\tilde\gamma_1 \in \Gamma^{\mathrm{ID}}(\mu_1, \nu_1)$.

We now apply Lemma~\ref{lem:quad-linear} on $\mathbb{R}^{2d}$ to the convolution semigroups $\{\tilde\gamma_t\}$ and $\{\gamma_t\}$ with the quadratic form $c(x,y) = \f 12 |x-y|^2$. For any $t \ge 0$, the lemma yields the linear scaling
\[
\int_{\mathbb{R}^{2d}} c(x,y) \, d\tilde\gamma_t(x,y)
= t \int_{\mathbb{R}^{2d}}c(x,y) \, d\tilde\gamma_1(x,y) + (t^2-t) c(m_{\mu},m_{\nu}),
\]
and similarly for $\gamma_t$ with $\gamma_1$ in place of $\tilde\gamma_1$. 
Since the quadratic term $(t^2-t) c(m_\mu,m_\nu)$ is identical for both $\tilde\gamma_t$ and $\gamma_t$, the strict inequality \eqref{eq:contradiction-t0} at time $t_0$ propagates to time $1$:
\[
 \int_{\mathbb{R}^{2d}} c(x,y) \, d\tilde\gamma_1(x,y)
<
 \int_{\mathbb{R}^{2d}} c(x,y) \, d\gamma_1(x,y).
\]
This contradicts the assumed optimality of $\gamma_1$ at time $1$. Hence, $\gamma_t$ must be ID-optimal for all $t \ge 0$.

Finally, we verify the asserted identity of the optimal cost. Since $\gamma_t$ is ID-optimal for $\mu_t$ and $\nu_t$ for all $t \ge 0$, applying Lemma~\ref{lem:quad-linear} with the quadratic function $c(x,y) = \frac12|x-y|^2$ yields
\begin{align*}
    \mathcal{C}_2^{\mathrm{ID}}(\mu_t, \nu_t) 
    &= \int_{\mathbb{R}^{2d}} c(x,y) \, d\gamma_t(x,y) \\
    &= t \int_{\mathbb{R}^{2d}} c(x,y) \, d\gamma_1(x,y) + (t^2 - t) \, c(m_\mu, m_\nu) \\
    &= t \, \mathcal{C}_2^{\mathrm{ID}}(\mu_1, \nu_1) + (t^2 - t) \, c(m_\mu, m_\nu).
\end{align*}
This completes the proof of Proposition \ref{prop:stab}. 
\end{proof}

\subsection{Optimal couplings of L\'evy processes among L\'evy couplings}\label{sec:2.3}

Propositions \ref{prop:minimizer} and \ref{prop:stab} imply that for any two convolution semigroups $\{\mu_t\}_{t\ge 0}$ and $\{\nu_t\}_{t\ge 0}$ with finite second moments, there exists a coupling convolution semigroup $\{\gamma_t^*\}_{t\ge 0}$ satisfying
\[
\mathcal{C}_2^{\mathrm{ID}}(\mu_t,\nu_t) = \int_{\mathbb{R}^{2d}} c(x,y) \, d\gamma_t^*(x,y), \qquad \forall\, t \ge 0.
\]
Consequently, the optimality lifts to the process level, which we formulate next.

In what follows, we adopt the setting of stochastic bases from Section \ref{sec:2.1.3} (see also Section \ref{sec:1.1}). 
Let $(\Omega, \mathcal{F}, \mathfrak{F})$ be the canonical path space on $\mathbb{R}^d$, and let $(\mathbb{P}^0, \mathbb{E}^0)$ and $(\tilde{\mathbb{P}}^0, \tilde{\mathbb{E}}^0)$ be two measures on the path space associated with two L\'evy processes $\{X_t\}_{t\ge 0}$ and $\{Y_t\}_{t\ge 0}$ on $\mathbb{R}^d$, respectively.
A \emph{L\'evy coupling process} is a measure $(\mathbf{P}^{(0,0)}, \mathbf{E}^{(0,0)})$ on the product path space $(\boldsymbol{\Omega}, \boldsymbol{\mathcal{F}}, \boldsymbol{\mathfrak{F}})$ (see \eqref{def:prod-filtered}) such that:
\begin{itemize}
	\item The coupled process $\{(X_t, Y_t)\}_{t\ge 0}$ is a L\'evy process on $\mathbb{R}^{2d}$ under $\mathbf{P}^{(0,0)}$,
	\item The marginal laws of $\mathbf{P}^{(0,0)}$ are $\mathbb{P}^0$ and $\tilde{\mathbb{P}}^0$.
\end{itemize}
Recall also that every L\'evy process measure, e.g., $(\mathbb{P}^0, \mathbb{E}^0)$, $(\tilde{\mathbb{P}}^0, \tilde{\mathbb{E}}^0)$, or $(\mathbf{P}^{(0,0)}, \mathbf{E}^{(0,0)})$, induces a family of Markov measures via translation: for instance,
\[
\mathbb{P}^x := (X_t + x)_\sharp \mathbb{P}^0, \qquad 
\tilde{\mathbb{P}}^y := (Y_t + y)_\sharp \tilde{\mathbb{P}}^0, \qquad 
\mathbf{P}^{(x,y)} := (X_t + x, Y_t + y)_\sharp \mathbf{P}^{(0,0)}.
\]

\begin{proposition}\label{levyspecialcase}
	Let $\{X_t\}_{t\ge 0}$ and $\{Y_t\}_{t\ge 0}$ be two L\'evy processes on $\mathbb{R}^d$ with finite second moments under the probability measures $(\mathbb{P}^0, \mathbb{E}^0)$ and $(\tilde{\mathbb{P}}^0, \tilde{\mathbb{E}}^0)$ on the path space. There exists a L\'evy coupling process $(\mathbf{P}_*^{(0,0)}, \mathbf{E}_*^{(0,0)})$ of the two processes such that for every L\'evy coupling process $(\mathbf{P}^{(0,0)}, \mathbf{E}^{(0,0)})$,
	\[
	\mathbf{E}_*^{(x,y)}\!\left[c(X_t, Y_t)\right]
	\;\le\;
	\mathbf{E}^{(x,y)}\!\left[c(X_t, Y_t)\right],
	\qquad \text{for all } (t, x, y) \in [0, \infty) \times \mathbb{R}^{2d}.
	\]
\end{proposition}

\begin{remark}
	In the language of Definition \ref{def:optimal-immersion}, Proposition \ref{levyspecialcase} implies that $\mathbf{P}_*^{(x,y)}$ is globally $c$-optimal among all L\'evy coupling processes. This optimality will be further extended to the class of all immersion couplings in the next section.
\end{remark}

\begin{proof}
	For $t \ge 0$, let $\mu_t = \operatorname{law}(X_t)$ and $\nu_t = \operatorname{law}(Y_t)$ under $\mathbb{P}^0$ and $\tilde{\mathbb{P}}^0$, respectively. 
	Then $\{\mu_t\}_{t\ge 0}$ and $\{\nu_t\}_{t\ge 0}$ are convolution semigroups of measures with finite second moments. 
	By Propositions \ref{prop:minimizer} and \ref{prop:stab}, there exists a convolution semigroup $\{\gamma_t^*\}_{t\ge 0}$ that provides an ID-optimal coupling of $\mu_t$ and $\nu_t$ for all $t \ge 0$. 
	Let $\mathbf{P}_*^{(0,0)}$ be the L\'evy process measure on the product path space uniquely associated with $\{\gamma_t^*\}_{t\ge 0}$; under $\mathbf{P}_*^{(0,0)}$, the coupled process $\{(X_t, Y_t)\}_{t\ge 0}$ is a L\'evy coupling process of $\{X_t\}$ and $\{Y_t\}$. Indeed, the coordinate process $\{X_t\}$ remains a L\'evy process under $\mathbf{P}_*^{(0,0)}$ (its increments are the projections of the increments of the coupled process), its time-$t$ law is the first marginal of $\gamma_t^*$, namely $\mu_t$, and hence its law on the path space is $\mathbb{P}^0$; similarly for $\{Y_t\}$.
	
	Now consider any competing L\'evy coupling process $\mathbf{P}^{(0,0)}$ of $\mbp^0,\tilde\mbp^0$. 
	For $t \ge 0$, set $\gamma_t' = \operatorname{law}_{\mathbf{P}^{(0,0)}}((X_t, Y_t)) \in \mathrm{ID}_2(\mathbb{R}^{2d})$. 
	Denote by $m_X = \mathbb{E}^0[X_1]$, $m_Y = \tilde{\mathbb{E}}^0[Y_1]$ the mean vectors of the two processes. 
	For any $t \ge 0$ and $x, y \in \mathbb{R}^d$, we compute:
	\begin{align*}
		\mathbf{E}^{(x,y)}\!\left[c(X_t, Y_t)\right] 
		&= \int_{\mathbb{R}^{2d}} c(x+u, y+v) \, d\gamma_t'(u,v) \\
		&= \int_{\mathbb{R}^{2d}} c(u,v) \, d\gamma_t'(u,v) + t(x-y)^\top (m_X - m_Y) + c(x,y) \\
		&\ge \int_{\mathbb{R}^{2d}} c(u,v) \, d\gamma_t^*(u,v) + t(x-y)^\top (m_X - m_Y) + c(x,y) \\
		&= \mathbf{E}_*^{(x,y)}\!\left[c(X_t, Y_t)\right].
	\end{align*}
	The inequality follows from the optimality of $\gamma_t^*$, which minimizes the quadratic cost among all ID couplings of $\mu_t$ and $\nu_t$. This completes the proof.
\end{proof}

\subsection{Optimality at the generator level}\label{sec:2-opt-gen}

We now demonstrate that the corresponding results extend to the level of generators, thereby providing the analytical framework required for the developments in Sections \ref{sec:4} and \ref{sec:3}.

\begin{definition}[Coupling generators, L\'evy transport derivatives, optimality]\label{def:levy-deriv}
    Let $\mathcal{A}, \mathcal{B} \in \mathcal{G}_2^\Lambda(\mathbb{R}^d)$ be two L\'evy generators with finite second moments. 
    
    \begin{enumerate}[label=(\alph*)]
        \item A \emph{L\'evy coupling generator} is a L\'evy generator $\mathcal{J} \in \mathcal{G}_2^\Lambda(\mathbb{R}^{2d})$ such that for all test functions $\varphi, \psi \in C_b^2(\mathbb{R}^d)$,
        \[
        \mathcal{J}[\varphi \otimes 1] = (\mathcal{A}\varphi) \otimes 1 \quad \text{and} \quad \mathcal{J}[1 \otimes \psi] = 1 \otimes (\mathcal{B}\psi).
        \]
        The space of all such L\'evy coupling generators is denoted by $\Gamma^\Lambda(\mathcal{A}, \mathcal{B})$.
        
        \item The \emph{L\'evy transport derivative} between $\mathcal{A}$ and $\mathcal{B}$ with respect to the quadratic cost $c(x,y) = \frac12|x-y|^2$ is the function $\theta^\Lambda: \mathbb{R}^d \times \mathbb{R}^d \to \mathbb{R}$ defined by
        \[
        \theta^\Lambda(x,y) := \inf_{\mathcal{J} \in \Gamma^\Lambda(\mathcal{A}, \mathcal{B})} (\mathcal{J} c)(x,y).
        \]
        
        \item A L\'evy coupling generator $\mathcal{J}_* \in \Ga^\La(\mA,\mB)$ is called \emph{L\'evy $2$-optimal} if
        \[
        (\mathcal{J}_* c)(x,y) = \theta^\Lambda(x,y) \qquad \text{for all } (x,y) \in \mathbb{R}^d \times \mathbb{R}^d.
        \]
    \end{enumerate}
\end{definition}

The generator-level marginal conditions in Definition \ref{def:levy-deriv}(a) correspond exactly to the process-level notion of L\'evy coupling processes from Section \ref{sec:2.3}. Recall from the one-to-one correspondence in Section \ref{sec:3.1.2} that each L\'evy generator carries a L\'evy exponent; we write $\Psi_\mA$ for the exponent associated with $\mA \in \mcg_2^\La(\mbr^d)$, and similarly $\Psi_\mB$, $\Psi_\mJ$.

\begin{lemma}\label{lem:gen-vs-process}
	Let $\mA, \mB \in \mcg_2^\La(\mbr^d)$ be the generators of the two L\'evy processes $(\mbp^0, \mbe^0)$ and $(\tilde\mbp^0, \tilde\mbe^0)$, and let $\mJ \in \mcg_2^\La(\mbr^{2d})$ with associated L\'evy process measure $(\mathbf{P}^{(0,0)}, \mathbf{E}^{(0,0)})$ on the product path space. The following are equivalent:
	\begin{enumerate}[label=(\roman*)]
		\item $\mJ \in \Ga^\La(\mA, \mB)$;
		\item $\Psi_\mJ(\xi, 0) = \Psi_\mA(\xi)$ and $\Psi_\mJ(0, \eta) = \Psi_\mB(\eta)$ for all $\xi, \eta \in \mbr^d$;
		\item for all $\vphi,\psi\in C_b(\mbr^d)$, it holds $e^{t\mJ}[\vphi\otimes 1]=e^{t\mA}\vphi\otimes 1, e^{t\mJ}[1\otimes\psi]=1\otimes e^{t\mB}\psi$;
		\item $(\mathbf{P}^{(0,0)}, \mathbf{E}^{(0,0)})$ is a L\'evy coupling process of $\{X_t\}$ and $\{Y_t\}$.
	\end{enumerate}
\end{lemma}

\begin{remark}\label{rem:C2-extension}
	The assertions in (i) and (iii) extend naturally to $C_2^2(\mathbb{R}^d)$ and $C_2(\mathbb{R}^d)$, respectively, since the generators and semigroups are well-defined on these larger spaces (see the discussion following \eqref{eq:LK-rep-op}). 
\end{remark}

\begin{proof}
	(i) $\Rightarrow$ (ii). For $\xi \in \mbr^d$, consider the exponential function $\varphi_\xi(x) := e^{i\xi^\top x}$, understood componentwise through its real and imaginary parts $\cos(\xi^\top x), \sin(\xi^\top x) \in C_b^2(\mbr^d)$. A direct computation from the representation \eqref{eq:LK-rep-op1} gives $\mA \varphi_\xi = \Psi_\mA(\xi)\, \varphi_\xi$. Moreover, $\varphi_\xi \otimes 1$ is the exponential function on $\mbr^{2d}$ with frequency $(\xi, 0)$, so likewise $\mJ(\varphi_\xi \otimes 1) = \Psi_\mJ(\xi, 0)\, (\varphi_\xi \otimes 1)$. If (i) holds, then
	\[
	\Psi_\mJ(\xi, 0)\, (\varphi_\xi \otimes 1) = \mJ(\varphi_\xi \otimes 1) = (\mA \varphi_\xi) \otimes 1 = \Psi_\mA(\xi)\, (\varphi_\xi \otimes 1),
	\]
	and since $\varphi_\xi \otimes 1$ is nowhere zero, $\Psi_\mJ(\xi, 0) = \Psi_\mA(\xi)$. The identity for the second marginal follows symmetrically.

	(ii) $\Rightarrow$ (iii). Let $\{\gamma_t\}_{t\ge 0}$ be the convolution semigroup generated by $\mJ$. By (ii) and the scaling of L\'evy exponents, $\Psi_{\gamma_t}(\xi, 0) = t\, \Psi_\mJ(\xi, 0) = t\, \Psi_\mA(\xi) = \Psi_{\mu_t}(\xi)$, and similarly in $\eta$; hence, by Lemma \ref{lem:marg-char}, $\gamma_t$ is a coupling of $\mu_t$ and $\nu_t$ for every $t \ge 0$. Consequently, for any $\varphi \in C_b(\mbr^d)$ and $(x, y) \in \mbr^{2d}$,
	\[
	e^{t\mJ}(\varphi \otimes 1)(x, y) = \int_{\mbr^{2d}} \varphi(x + u)\, d\gamma_t(u, v) = \int_{\mbr^d} \varphi(x + u)\, d\mu_t(u) = (e^{t\mA}\varphi)(x).
	\]
	Thus $e^{t\mathcal{J}}(\varphi \otimes 1) = e^{t\mathcal{A}}\varphi \otimes 1$. The same argument gives $e^{t\mathcal{J}}(1 \otimes \psi) = 1 \otimes e^{t\mathcal{B}}\psi$ for all $\psi \in C_b(\mathbb{R}^d)$.
	
	(iii) $\Rightarrow$ (i). Subtracting $\varphi(x)$ from the equality $e^{t\mathcal{J}}(\varphi \otimes 1)(x,y) = e^{t\mathcal{A}}\varphi(x)$, dividing by $t$, and letting $t \searrow 0$, the definition \eqref{eq:generator} of the generators yields $\mathcal{J}(\varphi \otimes 1) = (\mathcal{A}\varphi) \otimes 1$ for all $\vphi\in C_b^2(\mbr^d)$; the second marginal condition follows similarly. Hence $\mathcal{J} \in \Gamma^\Lambda(\mathcal{A}, \mathcal{B})$.

	(ii) $\Leftrightarrow$ (iv). Under $\mathbf{P}^{(0,0)}$, the coupled process $\{(X_t, Y_t)\}_{t\ge 0}$ is by construction a L\'evy process on $\mbr^{2d}$, and the coordinate process $\{X_t\}$ is again a L\'evy process (its increments are the projections of the increments of the coupled process), with L\'evy exponent $\xi \mapsto \Psi_\mJ(\xi, 0)$. Since the law of a L\'evy process on the path space is uniquely determined by its L\'evy exponent, the first marginal law of $\mathbf{P}^{(0,0)}$ equals $\mbp^0$ if and only if $\Psi_\mJ(\cdot, 0) = \Psi_\mA$; similarly for the second marginal. This is precisely the equivalence of (ii) and (iv).
\end{proof}

We now establish the equivalence between the optimality of L\'evy processes (in the sense of Proposition \ref{levyspecialcase}) and the optimality of their generators as defined above.

\begin{proposition}\label{prop:generator_optimality}
	Let $\{X_t\}_{t\ge 0}$ and $\{Y_t\}_{t\ge 0}$ be two L\'evy processes with finite second moments, and let $\mathcal{A}, \mathcal{B} \in \mathcal{G}_2^\Lambda(\mathbb{R}^d)$ be their respective generators. Let $\mathcal{J}_* \in \mathcal{G}_2^\Lambda(\mathbb{R}^{2d})$ be the generator of an optimal L\'evy coupling process $\mathbf{P}_*^{(0,0)}$ given in Proposition \ref{levyspecialcase}. Then $\mathcal{J}_*$ is L\'evy $2$-optimal.
\end{proposition}

\begin{proof}
	Let $\mathcal{J} \in \Gamma^\Lambda(\mathcal{A}, \mathcal{B})$ be any L\'evy coupling generator of $\mathcal{A}, \mathcal{B}$, and let $(\mathbf{P}^{(0,0)},\mathbf{E}^{(0,0)})$ be the L\'evy process generated by $\mathcal{J}$. 
	By Lemma \ref{lem:gen-vs-process}, $\mathbf{P}^{(0,0)}$ is a L\'evy coupling process of $\{X_t\},\{Y_t\}$.
	By Proposition \ref{levyspecialcase}, since $(\mathbf{P}_*^{(0,0)},\mathbf{E}_*^{(0,0)})$ is optimal, it holds for all $t \ge 0$, $x, y \in \mathbb{R}^d$ that
	\[
	\mathbf{E}_*^{(x,y)}[c(X_t, Y_t)] \le \mathbf{E}^{(x,y)}[c(X_t, Y_t)].
	\]
	Subtracting the initial cost $c(x,y) = \frac12|x-y|^2$ from both sides, dividing by $t > 0$, and letting $t\searrow 0$, we obtain from the definition of infinitesimal generators
	\[
	(\mathcal{J}_* c)(x,y) \le (\mathcal{J} c)(x,y).
	\]
	Since $\mathcal{J}$ was an arbitrary element of $\Gamma^\Lambda(\mathcal{A}, \mathcal{B})$, the generator $\mathcal{J}_*$ attains the infimum, i.e., $\mathcal{J}_* c = \theta^\Lambda$. Hence $\mathcal{J}_*$ is L\'evy $2$-optimal.
\end{proof}

\begin{proposition}\label{prop:ta-affine}
	Let $\mathcal{A}, \mathcal{B} \in \mathcal{G}_2^\Lambda(\mathbb{R}^d)$. 
	Let $\{\mu_t\}_{t\ge 0}, \{\nu_t\}_{t\ge 0}$ be the associated convolution semigroups, and let $m_\mathcal{A}, m_\mathcal{B} \in \mathbb{R}^d$ be their mean vectors (see \eqref{def:mean-vector}).
	Let $\theta^\Lambda$ be the L\'evy transport derivative and $\mathcal{J}_*$ a L\'evy $2$-optimal coupling generator between $\mathcal{A}$ and $\mathcal{B}$. Then the following identities hold for all $t \ge 0$, $x, y \in \mathbb{R}^d$:
	\begin{align}
		\theta^\Lambda(x,y) &= \theta^\Lambda(0,0) + (x - y)^\top (m_\mathcal{A} - m_\mathcal{B}), \label{eq:theta-affine}\\
		\mathcal{C}_2^{\mathrm{ID}}(\mu_1, \nu_1) &= \theta^\Lambda(0,0) + c(m_\mathcal{A}, m_\mathcal{B}), \label{eq:cost-identity}\\
		e^{t\mathcal{J}_*}c(x,y) &= c(x,y) + t\,\theta^\Lambda(x,y) + t^2 c(m_\mathcal{A}, m_\mathcal{B}). \label{eq:semigroup-exp}
	\end{align}
\end{proposition}

\begin{proof}
	We first establish \eqref{eq:theta-affine} and \eqref{eq:cost-identity} using the particular $2$-optimal coupling generator $\mathcal{J}_\circ$ constructed in Proposition \ref{prop:generator_optimality}, whose associated convolution semigroup $\{\gamma_t^\circ\}_{t\ge 0}$ is, by construction, ID-optimal for the pair $\mu_t, \nu_t$ for every $t \ge 0$. We then show that the flow of an \emph{arbitrary} L\'evy $2$-optimal generator $\mathcal{J}_*$ is also ID-optimal, from which \eqref{eq:semigroup-exp} follows in full generality. Throughout, the differentiation of $t \mapsto e^{t\mathcal{J}}c$ at $t = 0$ against the quadratic cost $c \in C_2^2(\mathbb{R}^{2d})$ is justified by the extension of L\'evy generators to $C_2^2$ discussed in Section \ref{sec:3.1.2}.

	\emph{Step 1: the identities \eqref{eq:theta-affine} and \eqref{eq:cost-identity}.}
	For any L\'evy coupling generator $\mathcal{J} \in \Gamma^\Lambda(\mathcal{A}, \mathcal{B})$ and $(x,y) \in \mathbb{R}^{2d}$,
	\begin{align}
		(e^{t\mathcal{J}}c)(x,y) &= \mathbf{E}^{(x,y)}[c(X_t, Y_t)] \nonumber \\
		&= \mathbf{E}^{(0,0)}[c(X_t + x, Y_t + y)] \nonumber \\
		&= \mathbf{E}^{(0,0)}\bigl[ c(X_t, Y_t) + (x - y)^\top (X_t - Y_t) + c(x,y) \bigr] \nonumber \\
		&= (e^{t\mathcal{J}}c)(0,0) + t (x - y)^\top (m_\mathcal{A} - m_\mathcal{B}) + c(x,y), \label{eq:calc21}
	\end{align}
	where $(\mathbf{P}^{(0,0)}, \mathbf{E}^{(0,0)})$ denotes the L\'evy coupling process generated by $\mathcal{J}$ (Lemma \ref{lem:gen-vs-process}), and the last equality uses $\mathbf{E}^{(0,0)}[X_t - Y_t] = t (m_\mathcal{A} - m_\mathcal{B})$.
	Applying \eqref{eq:calc21} with $\mathcal{J} = \mathcal{J}_\circ$, the $2$-optimality of $\mathcal{J}_\circ$ gives
	\[
	\theta^\Lambda(x,y) = (\mathcal{J}_\circ c)(x,y) = \lim_{t \searrow 0} \frac{(e^{t\mathcal{J}_\circ}c)(x,y) - c(x,y)}{t} = \theta^\Lambda(0,0) + (x - y)^\top (m_\mathcal{A} - m_\mathcal{B}),
	\]
	which establishes \eqref{eq:theta-affine}. Moreover, since $\{\gamma_t^\circ\}_{t\ge 0}$ is ID-optimal for every $t \ge 0$, Proposition \ref{prop:stab} yields
	\[
	(e^{t\mathcal{J}_\circ}c)(0,0) = \int_{\mathbb{R}^{2d}} c(x,y) \, d\gamma_t^\circ(x,y) = \mathcal{C}_2^{\mathrm{ID}}(\mu_t, \nu_t) = t \, \mathcal{C}_2^{\mathrm{ID}}(\mu_1, \nu_1) + (t^2 - t) c(m_\mathcal{A}, m_\mathcal{B}).
	\]
	Differentiating at $t = 0$ and using $c(0,0) = 0$ gives
	\(
	\theta^\Lambda(0,0) = \mathcal{C}_2^{\mathrm{ID}}(\mu_1, \nu_1) - c(m_\mathcal{A}, m_\mathcal{B}),
	\)
	which is equivalent to \eqref{eq:cost-identity}.

	\emph{Step 2: the flow of an arbitrary $2$-optimal generator is ID-optimal.}
	Let $\mathcal{J}_*$ be any L\'evy $2$-optimal coupling generator, and let $\{\gamma_t^*\}_{t\ge 0}$ be its convolution semigroup; by Lemma \ref{lem:gen-vs-process}, $\gamma_t^* \in \Gamma^{\mathrm{ID}}(\mu_t, \nu_t)$ for all $t \ge 0$. Set
	\[
	\delta := \int_{\mathbb{R}^{2d}} c \, d\gamma_1^* - \mathcal{C}_2^{\mathrm{ID}}(\mu_1, \nu_1) \ge 0.
	\]
	By Lemma \ref{lem:quad-linear} applied to $\{\gamma_t^*\}_{t\ge 0}$,
	\begin{align*}
		(e^{t\mathcal{J}_*}c)(0,0) = \int_{\mathbb{R}^{2d}} c \, d\gamma_t^*
		= t \int_{\mathbb{R}^{2d}} c \, d\gamma_1^* + (t^2 - t)\, c(m_\mathcal{A}, m_\mathcal{B})
		= \mathcal{C}_2^{\mathrm{ID}}(\mu_t, \nu_t) + t \delta,
	\end{align*}
	where the last equality uses the definition of $\delta$ and the cost identity of Proposition \ref{prop:stab}. Differentiating at $t = 0$ and invoking \eqref{eq:cost-identity},
	\[
	(\mathcal{J}_* c)(0,0) = \mathcal{C}_2^{\mathrm{ID}}(\mu_1, \nu_1) - c(m_\mathcal{A}, m_\mathcal{B}) + \delta = \theta^\Lambda(0,0) + \delta.
	\]
	Since $\mathcal{J}_*$ is $2$-optimal, $(\mathcal{J}_* c)(0,0) = \theta^\Lambda(0,0)$, and hence $\delta = 0$: the coupling $\gamma_1^*$ is ID-optimal. By Proposition \ref{prop:stab}, $\gamma_t^*$ is ID-optimal for every $t \ge 0$.

	\emph{Step 3: the identity \eqref{eq:semigroup-exp}.}
	By Step 2 and \eqref{eq:cost-identity},
	\[
	(e^{t\mathcal{J}_*}c)(0,0) = \mathcal{C}_2^{\mathrm{ID}}(\mu_t, \nu_t) = t\, \theta^\Lambda(0,0) + t^2 c(m_\mathcal{A}, m_\mathcal{B}).
	\]
	Substituting this and \eqref{eq:theta-affine} into \eqref{eq:calc21} with $\mathcal{J} = \mathcal{J}_*$ yields \eqref{eq:semigroup-exp}, completing the proof.
\end{proof}

\subsection{Additivity of optimal coupling with respect to the L\'evy--Khintchine decomposition}

We now establish an auxiliary result that is of independent interest: the infinitely divisible optimal quadratic transport cost $\mathcal{C}_2^{\mathrm{ID}}$ is additive with respect to the global L\'evy--Khintchine decomposition \eqref{eq:LK-decomp}. Although this additivity property is not directly needed for the proof of our main theorem, it sheds light on the structure of optimal ID couplings and may find applications in other contexts. The main result of this subsection is stated as follows.

\begin{theorem}\label{thm:additivity}
	Let $\mu,\nu \in \ID_2(\mathbb{R}^d)$, and let $\mu_\bullet,\nu_\bullet$ for $\bullet \in \{\nabla, \Delta, J\}$ be the drift, diffusion, and jump parts of $\mu,\nu$ under the global L\'evy--Khintchine decomposition. Then the infinitely divisible optimal quadratic transport cost satisfies
	\begin{equation}\label{eq:additivity}
		\mathcal{C}_2^{\mathrm{ID}}(\mu, \nu) = \mathcal{C}_2^{\mathrm{ID}}(\mu_{\nabla}, \nu_{\nabla}) + \mathcal{C}_2^{\mathrm{ID}}(\mu_{\Delta}, \nu_{\Delta}) + \mathcal{C}_2^{\mathrm{ID}}(\mu_{J}, \nu_{J}).
	\end{equation}
\end{theorem}

To prove the additivity theorem, we need the following lemma. 

\begin{lemma}\label{lem:ga-decomp}
	Let $\mu,\nu\in \ID_2(\mathbb{R}^d)$ and $\ga\in \ID_2(\mbr^{2d})$. With $\gamma = \gamma_\nabla * \gamma_\Delta * \gamma_J$, we have $\gamma \in \Gamma^{\mathrm{ID}}(\mu,\nu)$ if and only if $\gamma_{\bullet} \in \Gamma^{\mathrm{ID}}(\mu_{\bullet},\nu_{\bullet})$ for $\bullet = \nabla, \Delta, J$.
\end{lemma}

\begin{remark}
	We emphasize that the ``only if'' part of the lemma holds only when all the L\'evy--Khintchine decompositions of $\mu$, $\nu$, and $\gamma$ are in the global form \eqref{eq:ULK-rep}.
\end{remark}

\begin{proof}
	$\leftarrow$. Assume $\gamma_\bullet \in \Gamma^{\mathrm{ID}}(\mu_\bullet,\nu_\bullet)$ for $\bullet = \nabla, \Delta, J$. Then $\gamma = \gamma_\nabla * \gamma_\Delta * \gamma_J$ has L\'evy exponent
	\[
	\Psi_\gamma(\xi,\eta) = \Psi_{\gamma_\nabla}(\xi,\eta) + \Psi_{\gamma_\Delta}(\xi,\eta) + \Psi_{\gamma_J}(\xi,\eta),
	\]
	and similarly $\Psi_\mu = \Psi_{\mu_\nabla} + \Psi_{\mu_\Delta} + \Psi_{\mu_J}$. Evaluating at $\eta = 0$ gives
	\[
	\Psi_\gamma(\xi,0) = \sum_{\bullet} \Psi_{\gamma_\bullet}(\xi,0) = \sum_{\bullet} \Psi_{\mu_\bullet}(\xi) = \Psi_\mu(\xi),
	\]
	where the middle equality uses $\gamma_\bullet \in \Gamma^{\mathrm{ID}}(\mu_\bullet,\nu_\bullet)$ and Lemma \ref{lem:marg-char}. Likewise $\Psi_\gamma(0,\eta) = \Psi_\nu(\eta)$. By Lemma \ref{lem:marg-char}, we conclude $\gamma \in \Gamma^{\mathrm{ID}}(\mu,\nu)$.
	
	$\rightarrow.$ Assume $\gamma \in \Gamma^{\mathrm{ID}}(\mu,\nu)$. Let $(\theta, \sigma, \Sigma)$ be the L\'evy triplet in the untruncated form \eqref{eq:ULK-rep} associated to $\gamma$. We then have 
	\begin{align*}
		\Psi_\gamma(\xi,\eta) &= i\theta^\top(\xi,\eta) - \frac12 (\xi,\eta)^\top \sigma (\xi,\eta) 
		+ \int_{\mathbb{R}^{2d}\setminus\{0\}} \bigl( e^{i(\xi,\eta)^\top(z,z')} - 1 - i(\xi,\eta)^\top(z,z') \bigr) \, d\Sigma(z,z') \\
		&= \Psi_{\gamma_\nabla}(\xi,\eta) + \Psi_{\gamma_\Delta}(\xi,\eta) + \Psi_{\gamma_J}(\xi,\eta),
	\end{align*}
	where $\Psi_{\gamma_\nabla}$, $\Psi_{\gamma_\Delta}$, $\Psi_{\gamma_J}$ correspond to the drift, diffusion, and jump components, respectively. In particular, evaluating at $\eta=0$, the jump part satisfies
	\begin{align*}
		\Psi_{\gamma_J}(\xi,0) = \int_{\mathbb{R}^{2d}\setminus\{0\}} \bigl( e^{i\xi^\top z} - 1 - i\xi^\top z \bigr) \, d\Sigma(z,z').
	\end{align*}
	
	By Lemma \ref{lem:marg-char}, $\Psi_\gamma(\xi,0) = \Psi_\mu(\xi)$. Using the additivity of the L\'evy exponent,
	\[
	\Psi_{\gamma_\nabla}(\xi,0) + \Psi_{\gamma_\Delta}(\xi,0) + \Psi_{\gamma_J}(\xi,0) = \Psi_{\mu_\nabla}(\xi) + \Psi_{\mu_\Delta}(\xi) + \Psi_{\mu_J}(\xi).
	\]
	The right-hand side is the unique decomposition of the L\'evy exponent into drift, diffusion, and jump parts in the global form. Since each term on the left-hand side is of a distinct type — $\Psi_{\gamma_\nabla}(\xi,0)$ is purely imaginary and linear in $\xi$, $\Psi_{\gamma_\Delta}(\xi,0)$ is real and quadratic, and $\Psi_{\gamma_J}(\xi,0)$, as shown above, is given by the L\'evy integral — the uniqueness of the L\'evy--Khintchine decomposition in the global form forces termwise equality: $\Psi_{\ga_\bullet}(\xi,0)=\Psi_{\mu_\bullet}$ for $\bullet\in\{\nabla,\De,J\}$. 
	The same argument with $\xi = 0$ yields $\Psi_{\gamma_\bullet}(0,\eta) = \Psi_{\nu_\bullet}(\eta)$ for $\bullet = \nabla, \Delta, J$. By Lemma \ref{lem:marg-char}, this implies $\gamma_\bullet \in \Gamma^{\mathrm{ID}}(\mu_\bullet,\nu_\bullet)$ for each $\bullet$.
\end{proof}

We now return to the proof of Theorem \ref{thm:additivity}.

\begin{proof}[Proof of Theorem \ref{thm:additivity}]
	The proof makes use of the following fact: if $\gamma \in \ID_2(\mathbb{R}^{2d})$ admits the L\'evy--Khintchine decomposition $\gamma = \gamma_{\nabla} * \gamma_{\Delta} * \gamma_{J}$, and $Q: \mathbb{R}^{2d} \to \mathbb{R}$ is a quadratic form (e.g., $Q(x,y) = c(x,y) = \frac12 |x-y|^2$), then
	\[
	\langle \gamma, Q \rangle = \langle \gamma_{\nabla}, Q \rangle + \langle \gamma_{\Delta}, Q \rangle + \langle \gamma_{J}, Q \rangle.
	\]
	This identity follows directly from the fact that the L\'evy process associated with $\gamma$ is the independent sum of its drift, diffusion, and jump components, and both the diffusion and jump parts are centered.
	
	To prove the inequality $\ge$ in \eqref{eq:additivity}, pick any $\gamma \in \Gamma^{\mathrm{ID}}(\mu,\nu)$. By Lemma \ref{lem:ga-decomp}, $\gamma_{\bullet} \in \Gamma^{\mathrm{ID}}(\mu_{\bullet},\nu_{\bullet})$ for each $\bullet \in \{\nabla, \Delta, J\}$. Using the identity above,
	\[
	\langle \gamma, c \rangle = \langle \gamma_{\nabla}, c \rangle + \langle \gamma_{\Delta}, c \rangle + \langle \gamma_{J}, c \rangle
	\ge \mathcal{C}_2^{\mathrm{ID}}(\mu_{\nabla}, \nu_{\nabla}) + \mathcal{C}_2^{\mathrm{ID}}(\mu_{\Delta}, \nu_{\Delta}) + \mathcal{C}_2^{\mathrm{ID}}(\mu_{J}, \nu_{J}).
	\]
	Taking the infimum over all $\gamma \in \Gamma^{\mathrm{ID}}(\mu,\nu)$ yields the desired inequality.
	
	For the reverse inequality, for each component $\bullet \in \{\nabla, \Delta, J\}$ let $\gamma^*_\bullet \in \Gamma^{\mathrm{ID}}(\mu_\bullet,\nu_\bullet)$ be an ID-optimal coupling of the pair $(\mu_\bullet,\nu_\bullet)$. By Lemma \ref{lem:ga-decomp}, $\ga = \ga_\nabla^**\ga_\De^**\gamma_J^*\in\Ga^\ID(\mu,\nu)$. Using the identity above, we obtain
	\begin{align*}
		\mathcal{C}_2^{\mathrm{ID}}(\mu,\nu)&\le \inn{\ga,c} = \langle \gamma^*_{\nabla}, c \rangle + \langle \gamma^*_{\Delta}, c \rangle + \langle \gamma^*_{J}, c \rangle \\
		&= \mathcal{C}_2^{\mathrm{ID}}(\mu_{\nabla}, \nu_{\nabla}) + \mathcal{C}_2^{\mathrm{ID}}(\mu_{\Delta}, \nu_{\Delta}) + \mathcal{C}_2^{\mathrm{ID}}(\mu_{J}, \nu_{J}).
	\end{align*}
	Combining the two inequalities establishes the equality \eqref{eq:additivity}.
\end{proof}
\section{Optimal Immersion Couplings of L\'evy Processes}\label{sec:4}

In this section we prove the two main results of this work, Theorems \ref{thm:opt-immersion} and \ref{main}. Throughout this section, we assume the settings of these theorems. 
Specifically, let $(\Omega, \mathcal{F}, \mathfrak{F})$ and $(\tilde\Omega, \tilde{\mathcal{F}}, \tilde{\mathfrak{F}})$ be two filtered measurable spaces, on which the $\Pi$-valued process $\{X_t\}_{t\ge 0}$ and $\tilde \Pi$-valued process $\{Y_t\}_{t\ge 0}$, together with the two families of Markov measures $\{(\mathbb{P}^x, \mathbb{E}^x)\}_{x\in\Pi}$ and $\{(\tilde{\mathbb{P}}^y, \tilde{\mathbb{E}}^y)\}_{y\in\tilde\Pi}$ are defined. Fix two initial distributions $\mu \in \mathcal{P}(\Pi)$ and $\nu \in \mathcal{P}(\tilde\Pi)$ and consider an immersion coupling of the Markov measures $\mathbb{P}^\mu$ and $\tilde{\mathbb{P}}^{\nu}$ given by
\(
(\boldsymbol{\Omega}, \boldsymbol{\mathcal{F}}, \boldsymbol{\mathfrak{F}}, \mathbf{P}, \mathbf{E}),
\)
where $(\boldsymbol{\Omega}, \boldsymbol{\mathcal{F}}, \boldsymbol{\mathfrak{F}}=\cb{\bs{\mcl{F}}_t}_{t\ge 0})$ is the product filtered measurable space. 

Recall the laws of the two Markov processes:
\[
\mu_t^x := \operatorname{law}^x(X_t) = (X_t)_{\sharp}(\mathbb{P}^x)\in\mcp(\Pi) , \qquad 
\nu_t^y := \operatorname{law}^y(Y_t) = (Y_t)_\sharp (\tilde{\mathbb{P}}^y)\in\mcl{P}(\tilde\Pi).
\]
Introduce the semigroup operators $\{T_t\}_{t\ge 0}$ and $\{\tilde T_t\}_{t\ge 0}$ associated with the two Markov processes; that is, for any bounded continuous functions $\varphi \in C_b(\Pi)$ and $\psi \in C_b(\tilde\Pi)$,
\begin{align}\label{def:semigroup}
	T_t \varphi(x) := \mathbb{E}^x[\varphi(X_t)] = \int_{\Pi} \varphi \, d\mu_t^x, \qquad 
	\tilde T_t \psi(y) := \tilde{\mathbb{E}}^y[\psi(Y_t)] = \int_{\tilde\Pi} \psi \, d\nu_t^y.	
\end{align}

\subsection{Proof of Theorem \ref{thm:opt-immersion}}
We begin by proving Theorem \ref{thm:opt-immersion}. As in the setting of the theorem, let $c: \Pi \times \tilde\Pi \to \mathbb{R}_+$ be a lower semicontinuous cost, and $\mcc_c:\mcl{P}(\Pi)\times\mcp(\tilde\Pi)\to [0,\infty]$ be the associated $c$-optimal transport cost. Denote by $\Xi:\mbr_+\times \Pi\times \tilde \Pi\to[0,\infty]$ the $c$-transport cost between the laws of the two Markov processes:
\begin{align}\label{def:Xi}
	\Xi(t, x, y) := \mathcal{C}_c(\mu_t^x, \nu_t^y).
\end{align}
By \emph{Kantorovich duality} \cite{MR2459454} and the notion of semigroups from \eqref{def:semigroup}, it holds
\begin{align}\label{eq:Kantorovich}
	\Xi(t, x, y) = \sup_{(\varphi, \psi)} \left\{ \int_{\Pi} \varphi \, d\mu_t^x + \int_{\tilde\Pi} \psi \, d\nu_t^y \right\} = \sup_{(\varphi, \psi)} \sqb{T_t\vphi(x)+\tilde T_t\psi(y)},	
\end{align}
where the supremum is taken over all such pairs of bounded continuous functions $(\varphi, \psi) \in C_b(\Pi) \times C_b(\tilde\Pi)$ with $\varphi \oplus \psi \le c$. We note $\Xi(0,x,y)=c(x,y)$. Subsequently, define the difference quotient, for $t>0$ and $(x,y)\in\Pi\times\tilde\Pi$, by
\begin{align}\label{def:Ta}
	\Ta(t,x,y):= \frac{\Xi(t,x,y)-c(x,y)}{t}
\end{align}

We begin by proving the following lemma.

\begin{lemma}\label{lem:Xi-lem}
	Assume the setting of Theorem \ref{thm:opt-immersion}, and let $(\mathbf{P}, \mathbf{E})$ be an immersion coupling of the two Markov measures $\mathbb{P}^\mu$ and $\tilde{\mathbb{P}}^{\nu}$. 
	Then for any $t, h \ge 0$,
	\[
	\mathbf{E}\bigl[ c(X_{t+h}, Y_{t+h}) \mid \boldsymbol{\mathcal{F}}_t \bigr] \ge \Xi(h, X_t, Y_t)\qquad\mathbf{P\text{-}}a.s.
	\]
\end{lemma}

\begin{proof}
	Fix $t,h\ge0$, and let $\gamma_{t,h}$ be the conditional distribution of $(X_{t+h},Y_{t+h})$ given $\boldsymbol{\mathcal{F}}_t$. For every Borel set $A\subset\Pi$, Proposition \ref{prop:immersion} and the marginal Markov property give
	\begin{align*}
		\gamma_{t,h}(A\times\tilde\Pi)
		&=\mathbf{E}\bigl[\mathbf{1}_{\{X_{t+h}\in A\}}\mid\boldsymbol{\mathcal{F}}_t\bigr]
		=\mathbb{E}^\mu\bigl[\mathbf{1}_{\{X_{t+h}\in A\}}\mid\mathcal{F}_t\bigr]\otimes1
		=\mu_h^{X_t}(A)
		\qquad\mathbf{P}\text{-a.s.}
	\end{align*}
	Similarly, for every Borel set $B\subset\tilde\Pi$,
	\(
	\gamma_{t,h}(\Pi\times B)=\nu_h^{Y_t}(B)
	\)
	$\mathbf{P}$-a.s.
	Since the state spaces are Polish, applying these identities on countable generating classes of Borel sets shows that $\gamma_{t,h}\in\Gamma(\mu_h^{X_t},\nu_h^{Y_t})$, $\mathbf{P}$-almost surely. Hence, by the definition of the transport cost,
	\[
	\mathbf{E}\bigl[c(X_{t+h},Y_{t+h})\mid\boldsymbol{\mathcal{F}}_t\bigr]
	=\int_{\Pi\times\tilde\Pi}c(u,v)\,d\gamma_{t,h}(u,v)
	\ge\mathcal{C}_c(\mu_h^{X_t},\nu_h^{Y_t})
	=\Xi(h,X_t,Y_t)
	\quad\mathbf{P}\text{-a.s.}
	\]
	This proves the desired bound.
\end{proof}

We may now prove Theorem \ref{thm:opt-immersion}.

\begin{proof}[Proof of Theorem \ref{thm:opt-immersion}]
	Consider the $\mathbb{R}$-valued process $\{S_t\}_{t\ge 0}$ defined on $(\boldsymbol{\Omega}, \boldsymbol{\mathcal{F}}, \boldsymbol{\mathfrak{F}})$ as in the theorem.
	We first note that the process is integrable: $S_t\in L^1(\mathbf{P})$ for all $t\ge 0$. Indeed, by \ref{C0}, we have the bound $0\le c(X_t,Y_t)\le f(X_t)+g(Y_t)$, which, by \eqref{eq:sup-integrability} and the marginal laws of $\mathbf{P}$, implies $c(X_t,Y_t)\in L^1(\mathbf{P})$. Next, by \ref{C1} and the definition of the half-relaxed transport derivative and the continuity of $f$ and $g$, it holds for all $(x,y)\in \Pi\times\tilde\Pi$:
	\begin{align*}
		\omega_c^-(x,y)=\liminf_{\substack{t \searrow 0 \\ (x',y') \to (x,y)}} \frac{\mathcal{C}_c(\mu_t^{x'}, \nu_t^{y'}) - c(x',y')}{t}\ge \liminf_{(x',y')\to(x,y)}[-f(x')-g(y')]= -f(x)-g(y).
	\end{align*}
	The same computation leads to $\omega_c^-(x,y)\le f(x)+g(y)$. Hence $|\omega_c^-(x,y)|\le f(x)+g(y)$. By \eqref{eq:sup-integrability} and the marginal laws of $\mathbf{P}$, $\int_0^t \omega_c^-(X_s,Y_s)ds\in L^1(\mathbf{P})$. Therefore $S_t\in L^1(\mathbf{P})$.
	
	To establish that $\{S_t\}_{t\ge 0}$ is a $(\boldsymbol{\mathfrak{F}}, \mathbf{P})$-submartingale, we need to show for any fixed $0 \le s \le t < \infty$ that
	\begin{align}\label{eq:goal4.1}
		\mathbf{E}[S_t - S_s \mid \boldsymbol{\mathcal{F}}_s] = \mathbf{E}\left[ c(X_t, Y_t) - c(X_s, Y_s) - \int_s^t \omega_c^-(X_u, Y_u) \, du \;\middle|\; \boldsymbol{\mathcal{F}}_s \right] \ge 0.
	\end{align}
	
	Fix $n \ge 1$ and let $h_n := \frac{t-s}{n}$. Define $\tau_k = \tau_k^n := s + k h_n$ for $k = 0, \dots, n$, so that $\tau_0 = s$ and $\tau_n = t$. Writing the increment as a telescoping sum and using the tower property of conditional expectations, we obtain
	\begin{align}\label{eq:4.1}
		\mathbf{E}[c(X_t, Y_t) - c(X_s, Y_s) \mid \boldsymbol{\mathcal{F}}_s] = \mathbf{E}\bigl[ \Psi_n \mid \boldsymbol{\mathcal{F}}_s \bigr],
	\end{align}
	where
	\[
	\Psi_n := \sum_{k=0}^{n-1} \mathbf{E}\left[ c(X_{\tau_{k+1}}, Y_{\tau_{k+1}}) - c(X_{\tau_k}, Y_{\tau_k}) \;\middle|\; \boldsymbol{\mathcal{F}}_{\tau_k} \right].
	\]
	By Lemma \ref{lem:Xi-lem} and \eqref{def:Ta}, it holds $\mathbf{P}$-a.s. for each $k$,
	\[
	\mathbf{E}\left[ c(X_{\tau_{k+1}}, Y_{\tau_{k+1}}) - c(X_{\tau_k}, Y_{\tau_k}) \mid \boldsymbol{\mathcal{F}}_{\tau_k} \right] \ge \Theta(h_n, X_{\tau_k}, Y_{\tau_k})\cdot h_n
	\]
	Inserting these bounds into the definition of $\Psi_n$, we have
	\[
	\Psi_n \ge \sum_{k=0}^{n-1} \Theta(h_n, X_{\tau_k}, Y_{\tau_k}) \cdot h_n \quad \mathbf{P}\text{-a.s.}
	\]
	
	Define $\tau(u) = \tau_n(u) := \max\{ \tau_k : \tau_k \le u,\; 0 \le k \le n-1 \}$. Then the sum can be rewritten as an integral:
	\begin{align}\label{eq:4.2}
		\Psi_n \ge \int_s^t \Theta(h_n, X_{\tau(u)}, Y_{\tau(u)}) \, du.
	\end{align}
	Since the two processes $\{X_t\}$ and $\{Y_t\}$ are c\`adl\`ag almost surely under $\mathbb{P}^\mu$ and $\tilde{\mathbb{P}}^{\nu}$, respectively, the coupled process $\{(X_t, Y_t)\}$ is also c\`adl\`ag under the coupling measure $\mathbf{P}$. Moreover, as $n\to\infty$, we have $\tau_n(u)\to u$. Since c\`adl\`ag paths have at most countably many discontinuities,
	\[
	\lim_{n\to\infty}(X_{\tau_n(u)},Y_{\tau_n(u)})=(X_u,Y_u)
	\quad\text{for a.e. }u\in[s,t],\;\mathbf{P}\text{-a.s.}
	\]
	By the definition of the half-relaxed transport derivative $\omega_c^-$ from \eqref{def:half-relaxed}, it follows that
	\[
	\liminf_{n \to \infty} \Theta(h_n, X_{\tau(u)}, Y_{\tau(u)}) \ge \omega_c^-(X_u, Y_u) \quad \text{for a.e. } u \in [s, t],\; \mathbf{P}\text{-a.s.}
	\]
	
	Now combining \eqref{eq:4.1} and \eqref{eq:4.2}, for every $n \ge 1$ we have
	\[
	\mathbf{E}[c(X_t, Y_t) - c(X_s, Y_s) \mid \boldsymbol{\mathcal{F}}_s] \ge \mathbf{E}\left[ \int_s^t \Theta(h_n, X_{\tau(u)}, Y_{\tau(u)}) \, du \;\middle|\; \boldsymbol{\mathcal{F}}_s \right].
	\]
	Taking $\liminf$ as $n \to \infty$ on both sides and applying Fatou's lemma for conditional expectations, we obtain
	\begin{align*}
		\mathbf{E}[c(X_t, Y_t) - c(X_s, Y_s) \mid \boldsymbol{\mathcal{F}}_s]
		&\ge \liminf_{n \to \infty} \mathbf{E}\left[ \int_s^t \Theta(h_n, X_{\tau(u)}, Y_{\tau(u)}) \, du \;\middle|\; \boldsymbol{\mathcal{F}}_s \right] \\
		&\ge \mathbf{E}\left[ \liminf_{n \to \infty} \int_s^t \Theta(h_n, X_{\tau(u)}, Y_{\tau(u)}) \, du \;\middle|\; \boldsymbol{\mathcal{F}}_s \right] \\
		&\ge \mathbf{E}\left[ \int_s^t \liminf_{n \to \infty} \Theta(h_n, X_{\tau(u)}, Y_{\tau(u)}) \, du \;\middle|\; \boldsymbol{\mathcal{F}}_s \right]\\
		&\ge \mathbf{E}\left[ \int_s^t \omega_c^-(X_u, Y_u) \, du \;\middle|\; \boldsymbol{\mathcal{F}}_s \right].
	\end{align*}
	This is precisely \eqref{eq:goal4.1}, and hence the proof is complete.
	
	The two applications of Fatou's lemma above (the second and third inequalities) are legitimate provided the integrand is bounded below by a $\mathbf{P}$-integrable function.
	By \ref{C1}, $\Theta(h, x, y) \ge -f(x) - g(y)$, so
	\[
	\int_s^t \Theta(h_n, X_{\tau(u)}, Y_{\tau(u)}) \, du \ge -\int_s^t \bigl[ f(X_{\tau(u)}) + g(Y_{\tau(u)}) \bigr] \, du.
	\]
	Since $f(X_{\tau(u)}) \le \sup_{r \in [s, t]} f(X_r)$ and similarly for $g$, we obtain the uniform lower bound
	\[
	\int_s^t \Theta(h_n, X_{\tau(u)}, Y_{\tau(u)}) \, du \ge -(t-s) \left( \sup_{r \in [s, t]} f(X_r) + \sup_{r \in [s, t]} g(Y_r) \right).
	\]
	The integrability condition \eqref{eq:sup-integrability} and the marginal laws of $\mathbf{P}$ guarantee that the right-hand side is $\mathbf{P}$-integrable, and hence the first application of Fatou's lemma (conditional expectation) is valid. The second application (inside the integral) follows immediately from the same pointwise lower bound, which is independent of $n$ and integrable in $u$.
\end{proof}

\subsection{Proof of Theorem \ref{main}}
We now turn our attention to Theorem \ref{main}. 
We assume the same setting from the previous subsection, except now the cost $c: \mathbb{R}^d \times \mathbb{R}^d \to \mathbb{R}_+$ is the quadratic cost $c(x,y) = \frac12 |x-y|^2$, and $\{X_t\}, \{Y_t\}$ are L\'evy processes on $\mathbb{R}^d$ with finite second moments under measures $\mbp^0,\tilde\mbp^0$, with generators $\mathcal{A}, \mathcal{B} \in \mathcal{G}_2^\Lambda(\mathbb{R}^d)$.

Before proceeding to the proof, we state the \emph{strong duality} for transport derivatives, which will be a key ingredient. Recall the setting from Section \ref{sec:2-opt-gen}. Let $\mathcal{A}, \mathcal{B} \in \mathcal{G}_2^\Lambda(\mathbb{R}^d)$ be two L\'evy generators with finite second moments. Recall the notion of L\'evy transport derivative $\ta^\La:\mbr^d\times\mbr^d\to\mbr$ from Definition \ref{def:levy-deriv}. We now introduce the corresponding dual quantity. In what follows, denote by $C_2(\mathbb{R}^d)$ the space of all continuous functions $f: \mathbb{R}^d \to \mathbb{R}$ satisfying a quadratic growth bound $|f(x)| \le C(1+|x|^2)$ for some $C \ge 0$, and by $C_2^2(\mathbb{R}^d) \subset C_2(\mathbb{R}^d)$ the subspace of twice continuously differentiable functions $f$ with $\|D^2 f\|_\infty < \infty$.

\begin{definition}[Dual transport derivatives]\label{def:dual-deriv}
	Let $\mathcal{A}, \mathcal{B} \in \mathcal{G}_2^\Lambda(\mathbb{R}^d)$ be two L\'evy generators with finite second moments. The \emph{dual transport derivative} of $\mathcal{A}, \mathcal{B}$ is defined as 
	\[
	\omega'(x,y) := \sup_{(\varphi,\psi)} \bigl[ \mathcal{A}\varphi(x) + \mathcal{B}\psi(y) \bigr], \qquad (x,y) \in \mathbb{R}^{2d},
	\]
	where the supremum is taken over all pairs $(\varphi,\psi) \in C_2^2(\mathbb{R}^d)^2$ such that $c - (\varphi \oplus \psi)$ attains a global minimum at $(x,y)$.
\end{definition}

The strong duality is stated as follows.
We remark that the identity below was proved in \cite{KangLim2025}. A new but elementary proof will be provided in Section \ref{sec:3}.

\begin{proposition}[Strong duality]\label{thm:strongdual}
	Let $\mathcal{A}, \mathcal{B} \in \mathcal{G}_2^\Lambda(\mathbb{R}^d)$ be two L\'evy generators with finite second moments, and let $\theta^\Lambda,\omega'$ be their L\'evy and dual transport derivatives, respectively. Then
	\[
	\theta^\Lambda(x,y) = \omega'(x,y) \qquad \text{for all } (x,y) \in \mathbb{R}^{2d}.\]
\end{proposition}

To apply Theorem \ref{thm:opt-immersion}, we need to establish Assumption \ref{asp:C} and the integrability condition \eqref{eq:sup-integrability} for the two L\'evy processes $\{X_t\}, \{Y_t\}$. We recall the notion of the finite difference $\Theta: \mathbb{R}_+ \times (\mathbb{R}^d)^2 \to \mathbb{R}$ of the transport cost from \eqref{def:Ta} between two L\'evy processes. Denote by $m, \tilde m \in \mathbb{R}^d$ the mean vectors of the two L\'evy processes (see \eqref{def:mean-vector}):
\[
m = m_\mA= \mathbb{E}^0[X_1], \qquad \tilde m = m_\mB = \tilde{\mathbb{E}}^0[Y_1].
\]

\begin{lemma}\label{lem:Ta-prelim}
	For all $t>0$ and $(x,y)\in\mathbb{R}^{2d}$, the following translation identity holds:
	\begin{align}\label{eq:Ta-translation}
		\Theta(t,x,y)
		= \Theta(t,0,0)+(x-y)^\top(m-\tilde m).
	\end{align}
	Moreover, the following bound holds:
	\begin{align}\label{ineq:Ta}
		t\,c(m,\tilde m)
		\le \Theta(t,0,0)
		\le \theta^\Lambda(0,0)+t\,c(m,\tilde m).
	\end{align}
	As $t\searrow0$, $\Theta(t,\cdot,\cdot)$ converges uniformly to $\theta^\Lambda$ on $\mathbb{R}^{2d}$.
	In particular, the half-relaxed transport derivative satisfies
	$\omega^-=\theta^\Lambda=\om'$.
\end{lemma}

\begin{proof}
	We first prove the translation identity. For the quadratic cost and $\mu,\nu\in\mathcal{P}_2(\mathbb{R}^d)$, the translation property holds:
	\[
	\mathcal{C}_c(\delta_x*\mu,\delta_y*\nu)
	=\mathcal{C}_c(\mu,\nu)+c(x,y)
	+(x-y)^\top(m_\mu-m_\nu),
	\]
	where $m_\mu,m_\nu$ are the mean vectors of $\mu,\nu$, and $\delta_x*\mu$ denotes the translation of $\mu$ by $x$.
	Applying this identity to $\mu_t^0,\nu_t^0$, whose mean vectors are $tm,t\tilde m$, and using
	$\mu_t^x=\delta_x*\mu_t^0$ and
	$\nu_t^y=\delta_y*\nu_t^0$, we obtain
	\[
	\Xi(t,x,y)
	=\Xi(t,0,0)+c(x,y)+t(x-y)^\top(m-\tilde m).
	\]
	Subtracting $c(x,y)$ and dividing by $t>0$ proves \eqref{eq:Ta-translation}.
	
	We next establish the bound \eqref{ineq:Ta}. By Jensen's inequality, every coupling $\gamma\in\Gamma(\mu_t^0,\nu_t^0)$ satisfies
	\[
	\int_{\mathbb{R}^{2d}} c(u,v)\,d\gamma(u,v)
	\ge c(tm,t\tilde m)=t^2c(m,\tilde m).
	\]
	For the upper bound, let $\mathcal{J}_*$ be an optimal coupling generator of $\mathcal{A},\mathcal{B}$ given by Proposition \ref{prop:generator_optimality}. Its time-$t$ law is a coupling of $\mu_t^0,\nu_t^0$, so
	\[
	\Xi(t,0,0)\le(e^{t\mathcal{J}_*}c)(0,0)
	= t\theta^\Lambda(0,0)+t^2c(m,\tilde m),
	\]
	where the equality follows from \eqref{eq:semigroup-exp}. Since $\Theta(t,0,0)=t^{-1}\Xi(t,0,0)$, these bounds prove \eqref{ineq:Ta}.
	
	We now prove that $\Ta(t,0,0)\to \ta^\La(0,0)$. The upper bound in \eqref{ineq:Ta} gives
	\[
	\limsup_{t\searrow0}\Theta(t,0,0)
	\le\theta^\Lambda(0,0).
	\]
	To obtain the corresponding lower bound, fix $\varepsilon>0$. By the definition of the dual transport derivative (Definition \ref{def:dual-deriv}) and the strong duality $\theta^\Lambda=\omega'$ (Proposition \ref{thm:strongdual}), choose
	$(\varphi,\psi)\in C_2^2(\mathbb{R}^d)^2$ such that
	$\varphi\oplus\psi\le c$, $\varphi(0)+\psi(0)=c(0,0)=0$, and
	\[
	(\mathcal{A}\varphi)(0)+(\mathcal{B}\psi)(0)
	>\theta^\Lambda(0,0)-\varepsilon.
	\]
	Integrating $\varphi\oplus\psi\le c$ against any coupling of $\mu_t^0,\nu_t^0$ and taking the infimum gives
	\[
	\Xi(t,0,0)\ge
	(e^{t\mathcal{A}}\varphi)(0)+(e^{t\mathcal{B}}\psi)(0).
	\]
	Using the touching condition at the origin, we obtain
	\[
	\Theta(t,0,0)\ge
	\frac{(e^{t\mathcal{A}}\varphi)(0)-\varphi(0)}{t}
	+
	\frac{(e^{t\mathcal{B}}\psi)(0)-\psi(0)}{t}.
	\]
	By Remark \ref{rem:C2-extension}, the two difference quotients converge to $(\mathcal{A}\varphi)(0)$ and $(\mathcal{B}\psi)(0)$, respectively. Hence
	\[
	\liminf_{t\searrow0}\Theta(t,0,0)
	\ge(\mathcal{A}\varphi)(0)+(\mathcal{B}\psi)(0)
	>\theta^\Lambda(0,0)-\varepsilon.
	\]
	As $\varepsilon>0$ is arbitrary, this proves
	$\Theta(t,0,0)\to\theta^\Lambda(0,0)$.
	
	Finally, combining \eqref{eq:Ta-translation} with
	\eqref{eq:theta-affine}, we have
	\[
	\Theta(t,x,y)-\theta^\Lambda(x,y)
	=\Theta(t,0,0)-\theta^\Lambda(0,0).
	\]
	Thus as $t\searrow 0$,
	\[
	\sup_{(x,y)\in\mathbb{R}^{2d}}
	\bigl|\Theta(t,x,y)-\theta^\Lambda(x,y)\bigr|
	=\bigl|\Theta(t,0,0)-\theta^\Lambda(0,0)\bigr|
	\longrightarrow0.
	\]
	Since $\theta^\Lambda$ is continuous, the definition of the half-relaxed transport derivative \eqref{def:half-relaxed} gives
	$\omega^-=\theta^\Lambda$. Together with strong duality, this yields
	$\omega^-=\theta^\Lambda=\om'$.
\end{proof}

We may now complete the proof of Theorem \ref{main}.

\begin{proof}[Proof of Theorem \ref{main}]
	As given in the theorem, let $\{X_t\}_{t\ge 0}$ and $\{Y_t\}_{t\ge 0}$ be two L\'evy processes defined on the filtered probability spaces $(\Omega, \mathcal{F}, \mathfrak{F}, \mathbb{P}^0)$ and $(\tilde\Omega, \tilde{\mathcal{F}}, \tilde{\mathfrak{F}}, \tilde{\mathbb{P}}^0)$, respectively. By Proposition \ref{levyspecialcase}, let $\mathbf{P}_*^{(0,0)}$ be the measure on the product filtered space such that $\{(X_t, Y_t)\}_{t\ge 0}$ is $c$-optimal among all L\'evy couplings, in the sense of that proposition, where $c(x,y) = \frac12 |x-y|^2$. Note that $\mathbf{P}_*^{(0,0)}$ is an immersion coupling of $\mathbb{P}^0$ and $\tilde{\mathbb{P}}^0$.
	 
	We will show that $\mathbf{P}_*^{(0,0)}$ is both locally and globally $c$-optimal. 
	
	\medskip
	
	\emph{$\mathbf{P}_*^{(0,0)}$ is locally optimal.} 
	Let $\mathcal{J}_*$ be the L\'evy generator associated with the coupled process $\{(X_t, Y_t)\}_{t\ge 0}$ under $\mathbf{P}_*^{(0,0)}$. By Dynkin's formula and the finite second moments of the coupled L\'evy process, the following process is a $(\boldsymbol{\mathfrak{F}}, \mathbf{P}_*^{(0,0)})$-martingale:
	\[
	c(X_t, Y_t) - \int_0^t (\mathcal{J}_* c)(X_u, Y_u) \, du.
	\]
	By Proposition \ref{prop:generator_optimality} and Lemma \ref{lem:Ta-prelim}, we have $\mathcal{J}_* c = \theta^\Lambda = \omega_c^-$. Hence we conclude that
	\begin{align}\label{def:cost-process}
		c(X_t, Y_t) - \int_0^t \omega_c^-(X_u, Y_u) \, du
	\end{align}
	is a $(\boldsymbol{\mathfrak{F}}, \mathbf{P}_*^{(0,0)})$-martingale. 
	
	To complete the proof of local optimality, it remains to show that for any immersion coupling $\mathbf{P}$ of $\mathbb{P}^0$ and $\tilde{\mathbb{P}}^0$, the process \eqref{def:cost-process} is a $(\boldsymbol{\mathfrak{F}}, \mathbf{P})$-submartingale. We will invoke Theorem \ref{thm:opt-immersion} by verifying Assumption \ref{asp:C} and the integrability condition \eqref{eq:sup-integrability}. We first note that the quadratic bound holds:
	\begin{align*}
		c(x,y)&\le |x|^2+|y|^2\qquad (x,y)\in\mathbb{R}^{2d}.
	\end{align*}	
	Next, by Lemma \ref{lem:Ta-prelim}, the finite difference quotient of the transport cost satisfies the following bound for all $t \in (0,1]$ and $(x,y) \in \mathbb{R}^{2d}$:
	\begin{align*}
		|\Theta(t,x,y)| &\le |(x-y)^\top(m-\tilde m)| + c(m,\tilde m) + \theta^\Lambda(0,0)\\
		&\le \frac12|x-y|^2 + \frac12|m-\tilde m|^2 + c(m,\tilde m) + \theta^\Lambda(0,0)\\
		&\le |x|^2+|y|^2 + C,
	\end{align*}
	where $C := \frac12|m-\tilde m|^2 + c(m,\tilde m) + \theta^\Lambda(0,0)$.
		Now set $f(x)=|x|^2+C$ and $g(y)=|y|^2$. We then verify that the pair $(f,g)$ satisfies Conditions \ref{C0} and \ref{C1}. Finally, to verify \eqref{eq:sup-integrability}, for every $T>0$ we have
	\begin{align*}
		\mathbb{E}^0\left[\sup_{u\in[0,T]}f(X_u)\right]
		&\le C+\mathbb{E}^0\left[\sup_{u\in[0,T]}|X_u|^2\right]<\infty,
	\end{align*}
	since L\'evy processes with finite second moments have integrable running maxima (by applying Doob's inequality to the centered process $X_t-tm$). The same estimate holds for $g$ with $Y_u$. Thus Assumption \ref{asp:C} and \eqref{eq:sup-integrability} are satisfied, and Theorem \ref{thm:opt-immersion} implies that for any immersion coupling $\mathbf{P}$, the process \eqref{def:cost-process} is a submartingale. Hence $\mathbf{P}_*^{(0,0)}$ is locally optimal.

	\emph{$\mathbf{P}_*^{(0,0)}$ is globally optimal.}
	Fix any immersion coupling $(\mathbf{P}, \mathbf{E})$ of $\mathbb{P}^0$ and $\tilde{\mathbb{P}}^0$. As shown above, $\mathbf{P}_*^{(0,0)}$ is locally optimal, and hence \eqref{def:cost-process} is a $(\boldsymbol{\mathfrak{F}}, \mathbf{P}_*^{(0,0)})$-martingale. Moreover, by Theorem \ref{thm:opt-immersion}, \eqref{def:cost-process} is a $(\boldsymbol{\mathfrak{F}}, \mathbf{P})$-submartingale for any coupling $\mathbf{P}$. Consequently, for all $0 \le s \le t$,
	\begin{align}\label{eq:conditional-sub}
		\mathbf{E}[c(X_t, Y_t) \mid \boldsymbol{\mathcal{F}}_s] \ge c(X_s, Y_s) + \mathbf{E}\left[ \int_s^t \omega_c^-(X_u, Y_u) \, du \;\middle|\; \boldsymbol{\mathcal{F}}_s \right].
	\end{align}
	
	By Lemma \ref{lem:Ta-prelim} and Proposition \ref{prop:om-affine}, the half-relaxed transport derivative admits the affine decomposition $\omega_c^-(x, y) = \tilde f(x) + \tilde g(y)$, where 
	\[
	\tilde f(x) := \frac12 \omega_c^-(0,0) + x^\top (m - \tilde m), \qquad 
	\tilde g(y) := \frac12 \omega_c^-(0,0) - y^\top (m - \tilde m).
	\]
	Since $\mathbf{P}$ is an immersion coupling of $\mathbb{P}^0$ and $\tilde{\mathbb{P}}^0$, Proposition \ref{prop:immersion} together with the Markov property yields
	\begin{align*}
		\mathbf{E}\left[ \int_s^t \omega_c^-(X_u, Y_u) \, du \;\middle|\; \boldsymbol{\mathcal{F}}_s \right]
		&= \mathbf{E}\left[ \int_s^t \bigl( \tilde f(X_u) + \tilde g(Y_u) \bigr) \, du \;\middle|\; \boldsymbol{\mathcal{F}}_s \right] \\
		&= \mathbb{E}^0\left[ \int_s^t \tilde f(X_u) \, du \;\middle|\; \mathcal{F}_s \right] + \tilde{\mathbb{E}}^0\left[ \int_s^t \tilde g(Y_u) \, du \;\middle|\; \tilde{\mathcal{F}}_s \right] \\
		&= \mathbb{E}^{X_s}\left[ \int_0^{t-s} \tilde f(X_u) \, du \right] + \tilde{\mathbb{E}}^{Y_s}\left[ \int_0^{t-s} \tilde g(Y_u) \, du \right].
	\end{align*}
	
	Substituting this into \eqref{eq:conditional-sub}, we obtain
	\[
	\mathbf{E}[c(X_t, Y_t) \mid \boldsymbol{\mathcal{F}}_s] \ge c(X_s, Y_s) + \mathbb{E}^{X_s}\left[ \int_0^{t-s} \tilde f(X_u) \, du \right] + \tilde{\mathbb{E}}^{Y_s}\left[ \int_0^{t-s} \tilde g(Y_u) \, du \right].
	\]
	Since $\sigma(X_s, Y_s) \subset \boldsymbol{\mathcal{F}}_s$, taking conditional expectation with respect to $\sigma(X_s, Y_s)$ and using monotonicity gives
	\[
	\mathbf{E}[c(X_t, Y_t) \mid X_s, Y_s] \ge c(X_s, Y_s) + \mathbb{E}^{X_s}\left[ \int_0^{t-s} \tilde f(X_u) \, du \right] + \tilde{\mathbb{E}}^{Y_s}\left[ \int_0^{t-s} \tilde g(Y_u) \, du \right].
	\]
	
	For the optimal coupling $\mathbf{P}_*^{(0,0)}$, equality holds in \eqref{eq:conditional-sub} (since $S_t$ is a martingale), and consequently the inequality above becomes an equality as well. Therefore,
	\[
	\mathbf{E}_*^{(0,0)}[c(X_t, Y_t) \mid X_s, Y_s] \le \mathbf{E}[c(X_t, Y_t) \mid X_s, Y_s]
	\]
	for every immersion coupling $\mathbf{P}$. Hence $\mathbf{P}_*^{(0,0)}$ is globally optimal.
\end{proof}

\section{Strong Duality (Proof of Proposition \ref{thm:strongdual})}\label{sec:3}

\subsection{Strong duality and a formal proof}
In this final section, we establish the \emph{strong duality} (Proposition \ref{thm:strongdual}) between the L\'evy transport derivative $\theta^\Lambda$ (Definition \ref{def:levy-deriv}) and its dual quantity $\omega'$ (Definition \ref{def:dual-deriv}).

Throughout, we denote by $C_2^2(\mbr^d)$ the space of twice continuously differentiable functions $\vphi:\mbr^d\to\mbr$ for which there exists a constant $C\ge 0$ such that
\begin{align}\label{cond:C2growth}
	|\vphi(x)|\le C(1+|x|^2),\qquad 
	|\nabla \vphi(x)|\le C(1+|x|),\qquad 
	|D^2\vphi(x)|\le C,
\end{align}
for all $x\in\mbr^d$.
In particular, $C_2^2(\mbr^d)$ contains all quadratic functions of the form
\begin{align*}
	q(x)= a + b^\top x + x^\top Q x,
\end{align*}
where $a\in \mbr$, $b\in \mbr^d$, and $Q$ is a $(d\times d)$-matrix.
For $(x,y)\in \mbr^d\times \mbr^d$, we denote by $\mcl{D}(x,y)$ the family of all pairs $(\vphi,\psi)\in C_2^2(\mbr^d)^2$ such that
\begin{align*}
	c - \vphi\oplus\psi \ \text{achieves a global minimum at } (x,y),
\end{align*}
where we recall $c(x,y)=\f 12 |x-y|^2$. 
Equivalently, $\mcl{D}(x,y)$ consists of all pairs $(\vphi,\psi)$ for which $\vphi\oplus\psi$, up to an additive constant, touches $c$ from below at $(x,y)$.
This family of test functions appears in the definition of dual transport derivatives between L\'evy generators $\mA,\mB\in\mcg_2^\La(\mbr^d)$:
\begin{align*}
	\om'(x,y) =\sup_{(\vphi,\psi)\in \mcl{D}(x,y)} [(\mA\vphi)(x)+(\mB\psi)(y)].
\end{align*}

We first establish a structural identity for the dual transport derivative, mirroring the affine form of the L\'evy transport derivative given in Proposition \ref{prop:ta-affine}.

\begin{proposition}\label{prop:om-affine}
	Let $\mA,\mB\in \levyG_2(\mbr^d)$, and let $\om'$ be the dual transport derivative associated with the pair $(\mA,\mB)$. Then $\om'$ is affine and satisfies
	\begin{align*}
		\om'(x,y)=\om'(0,0)+ (m_\mA-m_\mB)^\top (x-y),\quad\mbox{ for all }(x,y)\in\mbr^d\times\mbr^d.
	\end{align*}
\end{proposition}

\begin{proof}
	Given $\varphi\in C_2^2(\mathbb{R}^d)$ and $x_0\in\mathbb{R}^d$, denote by $\tilde\varphi=\tilde\varphi_{x_0}\in C_2^2(\mbr^d)$ its first-order Taylor remainder:
	\[
	\tilde\varphi(x):=\varphi(x+x_0)-\varphi(x_0)-\nabla\varphi(x_0)^\top x.
	\]
	By the quadratic identity for $c(x,y)=\frac12|x-y|^2$, we have $(\varphi,\psi)\in\mcl{D}(x_0,y_0)$ if and only if
	\[
	(\tilde\varphi,\tilde\psi)\in\mcl{D}(0,0),
	\qquad
	\nabla\varphi(x_0)=x_0-y_0,
	\qquad
	\nabla\psi(y_0)=y_0-x_0.
	\]
	
	For $(\varphi,\psi)\in\mcl{D}(x_0,y_0)$, translation invariance of the L\'evy generators gives
	\[
	\mA\varphi(x_0)+\mB\psi(y_0)
	=\mA\tilde\varphi(0)+\mB\tilde\psi(0)
	+\nabla\varphi(x_0)^\top m_\mA+\nabla\psi(y_0)^\top m_\mB.
	\]
	Using the gradient identities above, we obtain
	\[
	\mA\varphi(x_0)+\mB\psi(y_0)
	=\mA\tilde\varphi(0)+\mB\tilde\psi(0)
	+(m_\mA-m_\mB)^\top(x_0-y_0).
	\]
	Taking the supremum over $(\varphi,\psi)\in\mcl{D}(x_0,y_0)$, and using the above correspondence up to additive constants, which do not affect the generators, yields
	\[
	\om'(x_0,y_0)=\om'(0,0)+(m_\mA-m_\mB)^\top(x_0-y_0).
	\]
	This proves the desired identity.
\end{proof}

Before turning to the rigorous proof of Proposition \ref{thm:strongdual}, we first present a guiding heuristic, commonly used in optimization theory, for deriving duality results.
Comparing the affine structures for $\theta^\Lambda$ and $\omega'$ established in Propositions \ref{prop:ta-affine} and \ref{prop:om-affine}, it suffices to establish the identity at the single point $(x,y) = (0,0)$.
Consider the functional
\(
F:\levyG_2(\mbr^{2d}) \times [C_2^2(\mbr^d)]^2 \to \mbr
\)
defined by
\begin{align}\label{def:F.3}
	F(\mJ,\vphi,\psi)
	:= (\mJ c)(0,0)-\mJ(\vphi\oplus \psi)(0,0)
	+ \mA\vphi(0)+\mB\psi(0).
\end{align}
For fixed $(\vphi,\psi)$, the map $\mJ\mapsto F(\mJ,\vphi,\psi)$ is affine, and for fixed $\mJ$, the map $(\vphi,\psi)\mapsto F(\mJ,\vphi,\psi)$ is affine as well. Formally applying a minimax argument, one obtains
\begin{align}
	\ta^\La(0,0)&=\inf_{\mJ\in\Ga^\La(\mA,\mB)} (\mJ c)(0,0) \nonumber
	\\
	&\stackrel{(1)}{=} \inf_{\mJ \in \levyG_2(\mbr^{2d})}
	\sup_{(\vphi,\psi)\in Y} F(\mJ,\vphi,\psi) \nonumber\\
	&\stackrel{(2)}{=} \sup_{(\vphi,\psi)\in Y}
	\inf_{\mJ \in \levyG_2(\mbr^{2d})} F(\mJ,\vphi,\psi) \nonumber\\
	&\stackrel{(3)}{=} \sup_{(\vphi,\psi)\in \mcl{D}(0,0)}
	\bigl[\mA\vphi(0)+\mB\psi(0)\bigr]=\om'(0,0), \label{eq:heuristic}
\end{align}
where $Y\subset [C_2^2(\mbr^d)]^2$ is a certain convex subset. 

This heuristic argument closely parallels the classical derivation of Kantorovich duality in optimal transport \cite{MR3050280}. In that setting, the functional
\(
F:\mcl{M}(\Pi^2)\times [C_b(\Pi)]^2\to\mbr
\)
is defined by
\begin{align*}
	F(\ga,\vphi,\psi)
	:= \int_{\Pi^2} (c-\vphi\oplus \psi)\,d\ga
	+ \int_\Pi \vphi\,d\mu + \int_\Pi \psi\,d\nu,
\end{align*}
and the interchange of the infimum and supremum is rigorously justified by the Fenchel--Rockafellar duality theorem. In the present work, we instead use a minimax theorem due to Lassonde \cite{Lassonde1983}. As the compactness argument below suggests, the same principle also provides an alternative derivation of Kantorovich duality for the quadratic cost and marginals with finite second moments, using quadratically growing test functions.

Our proof proceeds by making the heuristic argument above fully rigorous. To this end, we shall establish the following points:
\begin{enumerate}[label=(\roman*)]
	\item the applicability of Lassonde's minimax theorem to the present setting;
	\item for each fixed $\mJ\in \mcg_2^\La(\mbr^{2d})$,
	\begin{align}
		\sup_{(\vphi,\psi)\in Y}
		\sqb{\mA\vphi(0)+\mB\psi(0)-\mJ(\vphi\oplus\psi)(0,0)}
		= \begin{cases}
			0, & \mJ\in \Ga^\La(\mA,\mB),\\
			+\infty, & \mJ\notin \Ga^\La(\mA,\mB);
		\end{cases}\label{ind:sup}
	\end{align}
	\item for each fixed $(\vphi,\psi)\in Y$,
	\begin{align}
		\inf_{\mJ\in \mcg_2^\La(\mbr^{2d})}
		[\mJ(c-\vphi\oplus\psi)](0,0)
		= \begin{cases}
			0, & (\vphi,\psi)\in \mcl{D}(0,0),\\
			-\infty, & (\vphi,\psi)\notin \mcl{D}(0,0).
		\end{cases}\label{ind:inf}
	\end{align}
\end{enumerate}
In the following subsections, we address each of these points in turn.

\subsection{Lassonde's minimax principle}

To justify the interchange of infimum and supremum, we use a minimax result of Lassonde \cite{Lassonde1983}. Unlike the classical Sion minimax theorem, this result does not require the minimizing space to be compact or its topology to be inherited from an ambient topological vector space. Instead, it uses a compactness condition on a sublevel set associated with finitely many maximizing parameters.

We first recall the relevant terminology. Following \cite[Definition 2]{Lassonde1983}, a \emph{convex space} is a convex subset of a vector space equipped with a topology that induces the usual Euclidean topology on the convex hull of every finite subset. No topology on the ambient vector space is required. In particular, a Hausdorff convex set on which the convex combination map
\[
[0,1]\times X\times X\ni(\lambda,x,x')\mapsto\lambda x+(1-\lambda)x'
\]
is continuous is a convex space in this sense.

A function $f:X\to\mbr$ is called \emph{quasi-convex} if, for all $x,x'\in X$ and $\lambda\in[0,1]$,
\[
f(\lambda x+(1-\lambda)x')\le\max\{f(x),f(x')\}.
\]
It is \emph{lower semicontinuous} if every sublevel set $\{x\in X:f(x)\le r\}$ is closed, and \emph{inf-compact} if every such sublevel set is compact.

We state the following result using our notation for the variables, with the space $X$ corresponding to the minimizing variable and $Y$ to the maximizing variable.

\begin{theorem}[Lassonde's minimax principle {\cite[Corollary 1.12]{Lassonde1983}}]\label{thm:minimax}
	Let $X$ be a nonempty convex space and $Y$ a nonempty convex subset of a vector space. Let $F:X\times Y\to\mbr$ satisfy the following conditions:
	\begin{itemize}
		\item for every $y\in Y$, the function $x\mapsto F(x,y)$ is quasi-convex on $X$ and lower semicontinuous on every compact subset of $X$;
		\item for every $x\in X$, the function $y\mapsto F(x,y)$ is affine on $Y$;
		\item there exist a nonempty finite set $K_0\subset Y$ and a real number
		\[
		\gamma>\sup_{y\in Y}\inf_{x\in X}F(x,y)
		\]
		such that the following set is compact:
		\[
		\left\{x\in X:\max_{y\in K_0}F(x,y)\le\gamma\right\}.
		\]
	\end{itemize}
	Then the following minimax equality holds:
	\[
	\min_{x\in X}\sup_{y\in Y}F(x,y)
	=\sup_{y\in Y}\inf_{x\in X}F(x,y).
	\]
\end{theorem}

\subsection{Weak topology on the space $\mcg_2^\La(\mbr^{d})$}

To apply Lassonde's minimax principle stated in the previous subsection, we let $\mX$ denote the space of linear operators mapping $C_2^2(\mbr^{2d})$ into measurable functions, and set
\(
X=\levyG_2(\mbr^{2d})\subset\mX.
\)
Clearly, $\levyG_2(\mbr^{2d})$ is a convex subset of $\mX$. We now equip it with a topology that makes it a convex space in Lassonde's sense and supports the required compactness and semicontinuity arguments.

We introduce a \emph{weak topology} on $\mcg_2^\La(\mbr^{2d})$, defined through the weak topology on the space $\ID_2(\mbr^{2d})$ of infinitely divisible probability measures with finite second moment. For notational simplicity, we present the construction and the main arguments in the case of L\'evy generators on $\mbr^d$. All results extend verbatim to $\mbr^{2d}$, which is the setting relevant for the applications considered in this work.

Recall that each L\'evy generator $\mA\in\levyG_2(\mbr^d)$ uniquely determines an infinitely divisible probability measure with finite second moment, namely the law at time $1$ of the associated L\'evy process. More precisely, we associate to $\mA$ the probability measure
\begin{align}\label{def:cano-map}
	\mu_1 := \delta_0 e^{\mA} =: \mathfrak{I}(\mA) \in \ID_2(\mbr^d),
\end{align}
where $\{e^{t\mA}\}_{t\ge 0}$ denotes the L\'evy semigroup generated by $\mA$. This defines a canonical mapping
\begin{align}\label{def:imap}
	\mathfrak{I}:\levyG_2(\mbr^d)\to \ID_2(\mbr^d).
\end{align}
Using this canonical map, we may induce a topology on $\levyG_2(\mbr^d)$ as the initial topology associated with $\mathfrak{I}$, provided that a topology is specified on $\ID_2(\mbr^d)$. One natural choice is the Wasserstein--$2$ metric topology; however, this topology does not possess sufficiently strong compactness properties for our purposes. Instead, we shall work with the weak topology on $\ID_2(\mbr^d)$.

\begin{definition}[Weak topology and weak convergence on $\levyG_2(\mbr^d)$]\label{def:weak-top}
	Let $\mathfrak{I}$
	be the canonical map defined in \eqref{def:imap}, \eqref{def:cano-map}, and equip $\ID_2(\mbr^d)$ with the weak topology, that is, the subspace topology inherited from the weak topology on $\mcp(\mbr^d)$. The \emph{weak topology on $\levyG_2(\mbr^d)$} is defined as the initial topology induced by $\mathfrak{I}$.
	
	In particular, a sequence $\{\mA_n\}_{n\ge 1}\subset \levyG_2(\mbr^d)$ is said to converge weakly to $\mA\in\levyG_2(\mbr^d)$, denoted by $\mA_n \xrightarrow[]{w} \mA$, if the associated infinitely divisible measures $\mu_n := \delta_0 e^{\mA_n}$ converge weakly to $\mu := \delta_0 e^{\mA}$.
\end{definition}

We note that the above weak topology on $\levyG_2(\mbr^d)$, together with the associated notion of convergence, is not new in substance. Although it is rarely stated explicitly at the level of generators, closely related modes of convergence have long been used in the study of L\'evy processes; see \cite{Sato1999,LiangSchillingWang2020}. 

\begin{remark}\label{rem:metrizable}
	The weak topology on $\levyG_2(\mbr^d)$ is metrizable. This follows from the fact that the weak topology on $\ID_2(\mbr^d)$ is induced by the \emph{bounded Lipschitz} metric:
	\begin{align*}
		\mcl{W}_{\mathrm{BL}}(\mu, \nu) := \sup_{\|\vphi\|_{\Lip} \le 1} \left| \int_{\mbr^d} \vphi \, d(\mu - \nu) \right|, 
		\quad \text{where } \|\vphi\|_{\Lip} = \|\vphi\|_{\infty} + [\vphi]_{\Lip}.
	\end{align*}
	Given that the map $\mathfrak{I}$ from \eqref{def:imap} is a bijection, we induce a metric on $\levyG_2(\mbr^d)$ by setting:
	\begin{align*}
		\mcl{W}(\mA, \mB) = \mcl{W}_{\mathrm{BL}}(\mathfrak{I}(\mA), \mathfrak{I}(\mB)).
	\end{align*}
	In particular, the weak topology on $\levyG_2(\mbr^d)$ is \emph{sequential}, meaning that a set $K \subseteq \levyG_2(\mbr^d)$ is closed if and only if for every sequence $\{\mA_n\}_{n \in \mathbb{N}} \subseteq K$ such that $\mA_n \xrightarrow{w} \mA$ in $\levyG_2(\mbr^d)$, the limit $\mA$ is contained in $K$. Likewise, a set $K\subset\mcg_2^\La(\mbr^d)$ is compact if and only if for every sequence $\{\mA_n\}_n\subset K$ there is a subsequence $\{\mA_{n_k}\}_k$ and $\mA\in K$ such that $\mA_{n_k}\xrightarrow[]{w}\mA$.
\end{remark}

\begin{proposition}
	The convex combination map
	\[
	[0,1]\times\levyG_2(\mbr^d)\times\levyG_2(\mbr^d)\ni(\lambda,\mA,\mB)\mapsto\lambda\mA+(1-\lambda)\mB
	\]
	is jointly continuous with respect to the weak topology. Hence, $\levyG_2(\mbr^d)$ is a convex space.
\end{proposition}

\begin{proof}
	The proof is based on the following elementary fact. For $\lambda \ge 0$ and $\mu \in \ID(\mathbb{R}^d)$, denote by $\mu^{*\lambda}$ the $\lambda$-th convolution power. Then the maps
	\[
	(\lambda, \mu) \mapsto \mu^{*\lambda} \in \ID(\mathbb{R}^d), \qquad 
	(\mu, \nu) \mapsto \mu * \nu \in \ID(\mathbb{R}^d)
	\]
	are continuous with respect to the weak topology. This follows directly from the Lévy continuity theorem; we omit the elementary details.
	
	Since $\levyG_2(\mbr^d)$ is a sequential space, it suffices to show that for any sequences $\lambda_n \to \lambda$, $\mA_n \xrightarrow{w} \mA$, and $\mB_n \xrightarrow{w} \mB$, the convergence
	\begin{align}\label{goal:conv}
		\lambda_n \mA_n + (1 - \lambda_n)\mB_n \xrightarrow{w} \lambda \mA + (1 - \lambda)\mB
	\end{align}
	holds in $\levyG_2(\mbr^d)$. Let $\mu_n := \mathfrak{I}(\mA_n) = \delta_0 e^{\mA_n}$ and $\nu_n := \mathfrak{I}(\mB_n) = \delta_0 e^{\mB_n}$. Since $\mA_n\xrightarrow[]{w}\mA,\mB_n\xrightarrow[]{w}\mB$, we have $\mu_n \xrightarrow{w} \mu := \mathfrak{I}(\mA)$ and $\nu_n \xrightarrow{w} \nu := \mathfrak{I}(\mB)$. Noting that $\mathfrak{I}(\lambda \mA) = (\mathfrak{I}(\mA))^{*\lambda}$, we compute:
	\begin{align*}
		\mathfrak{I}(\lambda_n \mA_n + (1 - \lambda_n)\mB_n) &= \mathfrak{I}(\lambda_n \mA_n) * \mathfrak{I}((1 - \lambda_n)\mB_n) 
		= \mu_n^{*\lambda_n} * \nu_n^{*(1 - \lambda_n)}.
	\end{align*}
	As $\lambda_n \to \lambda$, $\mu_n \xrightarrow{w} \mu$, and $\nu_n \xrightarrow{w} \nu$, the continuity of convolution powers and convolution ensures that $\mu_n^{*\lambda_n} * \nu_n^{*(1 - \lambda_n)} \xrightarrow{w} \mu^{*\lambda} * \nu^{*(1 - \lambda)}$, which is precisely $\mathfrak{I}(\lambda \mA + (1 - \lambda)\mB)$. This establishes \eqref{goal:conv}.
	
	Since the weak topology is Hausdorff, this continuity implies that its restriction to every finite convex hull coincides with the usual Euclidean topology. Thus $\levyG_2(\mbr^d)$ is a convex space in Lassonde's sense.
\end{proof}

Our next task is to establish the lower semicontinuity of the map
\(
\levyG_2(\mbr^{2d}) \ni \mJ \mapsto F(\mJ,\vphi,\psi)
\)
under the weak topology introduced above. From the definition \eqref{def:F.3}, this reduces to proving the lower semicontinuity, for any fixed $f \in C_2^2(\mbr^{2d})$ bounded from below, of the map
\[
\mJ \mapsto (\mJ f)(0).
\]

To this end, we recall a classical characterization of weak convergence for infinitely divisible measures in terms of their L\'evy triplets (see \cite[Theorem 8.7]{Sato1999}). In what follows, let $b \in C_c(\mathbb{R}^d)$ be a continuous cutoff function equal to $1$ in a neighborhood of the origin, and denote by $(\kappa, \alpha, \Theta)_b$ the L\'evy triplet of an ID measure with respect to the cutoff function $b$; see Section \ref{sec:3.1.2} for details.  

\begin{lemma}\label{lem:Sato}
	For $n \in \mathbb{N} \cup \{\infty\}$, let $\mu_n \in \ID(\mbr^d)$, and let $(\kappa_n,\alpha_n,\Theta_n)_b$ be the associated L\'evy triplet. Then $\mu_n \to \mu_\infty$ weakly if and only if the truncated triplets $(\kappa_n,\alpha_n,\Theta_n)_b$ converge to $(\kappa_\infty,\alpha_\infty,\Theta_\infty)_b$ in the following sense:
	\begin{itemize}
		\item For every $f \in C_b(\mbr^d)$ that vanishes on a neighborhood of the origin, 
		\begin{align}\label{eq:f-convergence}
			\lim_{n\to\infty} \int_{\mbr^d} f \, d\Theta_n = \int_{\mbr^d} f \, d\Theta_\infty.
		\end{align}
		
		\item For $\de>0$, define the nonnegative definite matrix
		\[
		x^\top \alpha_n^\de x := x^\top \alpha_n x + \int_{|y|\le \de} (x^\top y)^2 \, \Theta_n(dy), \quad x \in \mbr^d.
		\]
		Then, for all $x \in \mbr^d$,
		\begin{align}\label{conv:diff-mat}
			\lim_{\de \searrow 0} \limsup_{n\to\infty} \big| x^\top \alpha_n^\de x - x^\top \alpha_\infty^\de x \big| = 0.
		\end{align}
		
		\item $\kappa_n \to \kappa_\infty$ in $\mbr^d$.
	\end{itemize}
\end{lemma}

\begin{remark}\label{rem:adelta}
    The nonnegative definite matrix $\alpha_n^\delta$ defined above can be characterized as follows: for any symmetric matrix $M$, 
    \[
    \operatorname{tr}(\alpha_n^\delta M) = \operatorname{tr}(\alpha_n M ) + \int_{|y| \le \delta} y^\top M y \, \Theta_n(dy).
    \]
    Likewise, the convergence \eqref{conv:diff-mat} is equivalent to the following condition: for every symmetric matrix $M$,
    \begin{align*}
    	\lim_{\de \searrow 0} \limsup_{n\to\infty}|\tr[(\al_n^\de-\al_\infty^\de)M]| =0.
    \end{align*}
\end{remark}

We first establish the continuity of the map $\levyG_2(\mbr^d)\ni\mA \mapsto (\mA f)(0)$ for $f\in C_b^2(\mbr^d)$. 
\begin{lemma}\label{lem:weak-continuity}
	Let $f \in C_b^2(\mbr^d)$, and let $\{\mA_n\}_{n\in\mbn \cup \{\infty\}} \subset \levyG_2(\mbr^d)$. 
	Suppose that $\mA_n \xrightarrow{w} \mA_\infty$. 
	Then, for every $x \in \mbr^d$,
	\(
	(\mA_n f)(x) \;\longrightarrow\; (\mA_\infty f)(x)
	\) 
	as $n\to\infty$.
\end{lemma}

\begin{proof}
	By translation invariance of Lévy generators, it suffices to prove the convergence at a single point, say $x=0$, i.e.,
	\[
	\mA_n f(0) \to \mA_\infty f(0).
	\]
	Since $\mA_n \to \mA_\infty$ weakly, by definition, the associated infinitely divisible measures converge weakly: $\mu_n \to \mu_\infty$. Invoking Lemma~\ref{lem:Sato}, the associated truncated Lévy triplets $(\kappa_n, \alpha_n, \Theta_n)_b$ converge to $(\kappa_\infty, \alpha_\infty, \Theta_\infty)_b$ (for a fixed cutoff function $b\in C_c(\mbr^d)$ with $b\equiv1$ on $B_1(0)$) in the sense specified there. In particular, weak convergence implies the following uniform second-moment bound, see \cite[Proof of Theorem 8.7]{Sato1999}:
	\begin{align}\label{def:M}
		M := \sup_{n\in\mbn\cup\{\infty\}} \int_{\mbr^d} \min\{1, |x|^2\} \, d\Theta_n(x) < \infty.
	\end{align}

	Define the following set of positive radii:
	\[
	E := \{\delta>0 : \Theta_\infty(\{x \in \mbr^d : |x| = \delta\}) = 0\},
	\]
	i.e., the set of radii for which the limiting Lévy measure gives no mass to the corresponding sphere. Clearly, $0$ is a limit point of $E$. For any $\delta \in E$, the truncated measures $\Theta_n \mathbf{1}_{B_\delta(0)^c}$ converge weakly to $\Theta_\infty \mathbf{1}_{B_\delta(0)^c}$. By the Portmanteau lemma, for any bounded continuous $g \in C_b(\mbr^d)$,
	\begin{align}\label{eq:weak-conv}
		\lim_{n\to\infty} \int_{|x|>\delta} g(x) \, d\Theta_n(x) = \int_{|x|>\delta} g(x) \, d\Theta_\infty(x).
	\end{align}
	
	Next, fix $f \in C_b^2(\mbr^d)$ and define the modulus of continuity of $D^2 f$:
	\[
	\rho(\delta) := \sup_{|x|\le \delta} \|D^2 f(x) - D^2 f(0)\|.
	\]
	Since $f \in C_b^2(\mbr^d)$, we have $\rho(\delta) \to 0$ as $\delta \searrow 0$. By Taylor's theorem, for $|x| \le \delta$,
	\begin{align}\label{eq:taylor}
		|R_2 f(x)| := \bigl| f(x) - f(0) - x^\top \nabla f(0) - \tfrac{1}{2} x^\top D^2 f(0) x \bigr| \le \rho(\delta) |x|^2.
	\end{align}
	
	Now consider $\mA_n f(0)$, for $n\in\mbn\cup\{\infty\}$, under the L\'evy--Khintchine decomposition with respect to $(\kappa_n, \alpha_n, \Theta_n)_b$ with a fixed cutoff parameter $\de\in(0,1]$:
	\begin{align*}
		\mA_n f(0)&= \ka_n^\top \nabla f(0)+ \rb{\f 12\tr [\al_n D^2f(0)]+\int_{|x|\le\de} \Psi(x) d\Ta_n(x)}+\int_{|x|>\de}\Psi(x) d\Ta_n(x)\\
		&=: T_{n}^\nabla+ T_{n,\de}^\De+ T_{n,\de}^J,
	\end{align*}
	where $\Psi(x):=f(x)-f(0)-x^\top \nabla f(0)b(x)$.
	Then the difference $\mA_nf(0)-\mA_\infty f(0)$ is bounded as follows:
	\begin{align*}
		|\mA_n f(0) - \mA_\infty f(0)| \le |T_n^\nabla - T_\infty^\nabla| + |T_{n,\delta}^\Delta - T_{\infty,\de}^\Delta| + |T_{n,\delta}^J - T_{\infty,\de}^J|.
	\end{align*}
	
	We next consider each component separately as $n\to\infty$.  
	For the drift component, since $\kappa_n \to \kappa_\infty$, we immediately have
	\begin{align}\label{lim:drift}
		|T_n^\nabla - T_\infty^\nabla| \to 0 \quad \text{as } n \to \infty.
	\end{align}
	For the diffusion term, using the Taylor expansion $\Psi(x) = \frac12 x^\top D^2 f(0) x + R_2(x)$, we obtain
	\begin{align*}
		|T_{n,\delta}^\Delta - T_{\infty,\de}^\Delta| 
		&= \Bigl| \frac12 \operatorname{tr}[\alpha_n D^2 f(0)] + \int_{|x|\le \delta} \Psi(x) \, d\Theta_n(x) - \frac12 \operatorname{tr}[\alpha_\infty D^2 f(0)]-\int_{|x|\le \delta} \Psi(x) \, d\Theta_\infty(x) \Bigr| \\
		&\le \Bigl| \frac12 \operatorname{tr}[(\alpha_n-\alpha_\infty) D^2 f(0)] + \frac12 \int_{|x|\le \delta} (x^\top D^2 f(0) x) b(x) \, d(\Theta_n-\Ta_\infty)(x)   \Bigr| \\
		&\quad + \Bigl| \int_{|x|\le \delta} R_2(x) b(x) \, d(\Theta_n-\Theta_\infty)(x) \Bigr| \\
		&= \Bigl| \frac12 \operatorname{tr}[(\alpha_n^\delta - \alpha^\de_\infty) D^2 f(0)] \Bigr| + \Bigl| \int_{|x|\le \delta} R_2(x) b(x) \, d(\Theta_n-\Theta_\infty)(x) \Bigr|,
	\end{align*}
	where $\alpha_n^\delta$ is defined as in Lemma~\ref{lem:Sato} (see Remark \ref{rem:adelta}).  
	By \eqref{eq:taylor}, $\de\le 1$, and the uniform bound \eqref{def:M}, the remainder term is controlled by
	\[
	\Bigl| \int_{|x|\le \delta} R_2(x) b(x) \, d\Theta_n(x) \Bigr| \le \rho(\delta) \sup_n \int_{|x|\le \delta} |x|^2 \, d\Theta_n(x) \le  M \rho(\delta).
	\]
	Hence we obtain
	\begin{align}\label{item:diffusion-conv}
		\limsup_{n\to\infty} |T_{n,\delta}^\Delta - T_{\infty,\de}^\Delta| \le   \limsup_{n\to\infty} \Bigl| \frac12 \operatorname{tr}[(\alpha_n^\delta - \alpha^\de_\infty) D^2 f(0)] \Bigr| + 2M \rho(\delta).
	\end{align}
	Finally, we consider the jump part. 
	Since $\Psi\in C_b(\mbr^d)$ (as the cutoff function $b$ is compactly supported), by \eqref{eq:weak-conv}, for any $\delta \in E$,
	\begin{align}
		\limsup_{n\to\infty} |T_{n,\delta}^J - T_{\infty,\de}^J| 
		&= \limsup_{n\to\infty} \Bigl| \int_{|x|> \delta} \Psi(x) \, d(\Theta_n - \Theta_\infty)(x) \Bigr|=0.
		\label{item:jump}
	\end{align}
	
	Combining the estimates from \eqref{lim:drift}, \eqref{item:diffusion-conv}, and \eqref{item:jump}, we obtain that for all $\delta \in E\cap(0,1]$,
	\begin{align*}
		\limsup_{n\to\infty} |\mA_n f(0) - \mA_\infty f(0)| \le \limsup_{n\to\infty} \Bigl| \frac12 \operatorname{tr} \bigl[ (\alpha_n^\delta - \alpha^\de_\infty) D^2 f(0) \bigr] \Bigr| + 2 M \rho(\delta) .
	\end{align*}
	Since $0$ is a limit point of $E$, letting $\delta \searrow 0$ along a sequence in $E$, the convergence in \eqref{conv:diff-mat} (see also Remark \ref{rem:adelta}) together with $\rho(\delta)\to 0$ implies that
	\(
	\mA_n f(0) \to \mA_\infty f(0)
	\)
	as $n\to\infty$,
	thus establishing the desired pointwise convergence.
\end{proof}

We now extend the preceding result from $f\in C_b^2(\mbr^d)$ to $f\in C_2^2(\mbr^d)$, with continuity replaced by lower semicontinuity.

\begin{proposition}\label{prop:liminf-gen}
	Let $f\in C_2^2(\mbr^d)$ satisfy $\inf_{x\in\mbr^d} f(x)>-\infty$. 
	Suppose $\{\mA_n\}_{n\in\mbn\cup\{\infty\}}\subset \levyG_2(\mbr^d)$ is a sequence of L\'evy generators such that $\mA_n\xrightarrow[]{w}\mA_\infty$. 
	Then
	\[
	\liminf_{n\to\infty} \mA_n f(0)\ge \mA_\infty f(0).
	\]
\end{proposition}

\begin{proof}
	We decompose $f=f_0+f_1$, where $f_0\in C_c^2(\mbr^d)$ is a compactly supported $C^2$-function such that $f_0=f$ on $B_1(0)$, and $f_1\in C_2^2(\mbr^d)$ vanishes on $B_1(0)$. In particular, $f_1(0)=0$ and $\nabla f_1(0)=0$.
	
	By Lemma~\ref{lem:Sato} and the continuity result established previously for compactly supported test functions, we have
	\(
	\lim_{n\to\infty} \mA_n f_0(0)=\mA_\infty f_0(0).
	\)
	Consequently,
	\[
	\liminf_{n\to\infty} \mA_n f(0)
	\ge \mA_\infty f_0(0)+\liminf_{n\to\infty} \mA_n f_1(0).
	\]
	
	We now analyze the term involving $f_1$. Since $f_1$ vanishes in a neighborhood of the origin, the drift and diffusion parts of the generators $\mA_n$ and $\mA_\infty$ do not contribute, and we may write for $n\in\mbn\cup\{\infty\}$:
	\[
	(\mA_n f_1)(0)=\int_{\mbr^d} [f_1(x)-f_1(0)-x^\top\nabla f_1(0)b(x)]d\Theta_n(x)=\int_{B_1(0)^c} f_1(x)\,d\Theta_n(x)
	\]
	Choose $r\in(0,1)$ such that $\Theta_\infty(\partial B_r(0))=0$. By the weak convergence $\mA_n\xrightarrow[]{w}\mA_\infty$, the cutoff L\'evy measures $\Theta_n \mathbf{1}_{B_r(0)^c}$ converge weakly to $\Theta_\infty \mathbf{1}_{B_r(0)^c}$ (Lemma \ref{lem:Sato}). Since $f_1$ is continuous, bounded from below, and vanishes on $B_1(0)$, the Portmanteau lemma yields
	\[
	\liminf_{n\to\infty} (\mA_nf_1)(0)
	=\liminf_{n\to\infty} \int_{B_r(0)^c} f_1(x)\,d\Theta_n(x)
	\ge \int_{B_r(0)^c} f_1(x)\,d\Theta_\infty(x)
	=(\mA_\infty f_1)(0).
	\]
	Combining the above estimates, we obtain the claimed lower semicontinuity, which completes the proof.
\end{proof}

\begin{lemma}\label{lem:levy-compact}
	Let $Q \in \mathcal{S}_{>0}(\mathbb{R}^d)$ be a strictly positive definite matrix, and denote $\varphi_Q(x) = x^\top Q x \in C_2^2(\mathbb{R}^d)$. Fix $a, b \in \mathbb{R}$. The set
	\begin{align}\label{def:Gq-set}
		K_{a,b} := \bigl\{ \mathcal{A} \in \mathcal{G}_2^\Lambda(\mathbb{R}^d) : \mathcal{A} \varphi_Q(0) \le a \text{ and } |m_\mathcal{A}| \le b \bigr\}
	\end{align}
	is compact in $\mathcal{G}_2^\Lambda(\mathbb{R}^d)$ with respect to the weak topology. In particular, for any sequence $\{\mathcal{A}_n\}_n \subset K_{a,b}$, there exists a subsequence $\{\mathcal{A}_{n_k}\}_k$ and $\mathcal{A} \in K_{a,b}$ such that $\mathcal{A}_{n_k} \xrightarrow{w} \mathcal{A}$.
\end{lemma}

\begin{proof}
	If $K_{a,b} = \emptyset$, there is nothing to prove. Thus assume $K_{a,b}$ is non-empty. Then necessarily $a \ge 0$ and $b \ge 0$, since $\mathcal{A}\varphi_Q(0) \ge 0$ by the positive maximum principle (as $\varphi_Q \ge 0$) and $|m_\mathcal{A}| \ge 0$. Hence the proof below proceeds with $a, b \ge 0$.
	
	Let $\mathfrak{I}(\mathcal{A}) := \delta_0 e^{\mathcal{A}} \in \ID_2(\mathbb{R}^d)$ denote the canonical map. We first establish the uniform bound on the quadratic moment:
	\begin{align}\label{bdd:quad-gen}
		\sup_{\mathcal{A} \in K_{a,b}} (e^{\mathcal{A}} \varphi_Q)(0) 
		= \sup_{\mathcal{A} \in K_{a,b}} \int_{\mathbb{R}^d} x^\top Q x \, d(\delta_0 e^{\mathcal{A}})(x) < \infty.
	\end{align}
	Indeed, let $\{\mu_t\}_{t\ge 0}$ be the convolution semigroup associated with $\mathcal{A}$. Differentiating the identity \eqref{eq:quad-moment-2} from Lemma \ref{lem:quad-linear}, we have 
	\[
	\mathcal{A} \varphi_Q(0) = \int_{\mathbb{R}^d} \varphi_Q(x)\,d\mu_1(x) - \varphi_Q(m_\mathcal{A}).
	\]
	Thus, \eqref{eq:quad-moment-2} gives, for all $t \ge 0$,
	\begin{align}\label{eq:etA}
		e^{t\mathcal{A}} \varphi_Q(0) = t (\mathcal{A} \varphi_Q)(0) + t^2 m_\mathcal{A}^\top Q m_\mathcal{A}.
	\end{align}
	Since $\mathcal{A} \varphi_Q(0) \le a$ and $|m_\mathcal{A}| \le b$, the bound in \eqref{bdd:quad-gen} follows immediately.
	
	Since $Q$ is strictly positive definite, there exists $c > 0$ such that $x^\top Q x \ge c |x|^2$. Hence \eqref{bdd:quad-gen} implies that the family of measures $\{\delta_0 e^{\mathcal{A}} : \mathcal{A} \in K_{a,b}\}$ has uniformly bounded second moments, and is therefore tight. By Prokhorov's theorem, for any sequence $\{\mathcal{A}_n\}_n \subset K_{a,b}$, there exists a subsequence $\{\mathcal{A}_{n_k}\}_k$ such that $\delta_0 e^{\mathcal{A}_{n_k}}$ converges weakly to some infinitely divisible measure $\mu$. The uniform bound on second moments ensures that $\mu \in \ID_2(\mathbb{R}^d)$, and hence $\mu = \delta_0 e^{\mathcal{A}}$ for some $\mathcal{A} \in \mathcal{G}_2^\Lambda(\mathbb{R}^d)$. This shows that $\mathcal{A}_{n_k} \xrightarrow{w} \mathcal{A}$ in $\mathcal{G}_2^\Lambda(\mathbb{R}^d)$. The uniform second-moment bound also implies uniform integrability of the first moments, so $m_{\mathcal{A}_{n_k}}\to m_{\mathcal{A}}$ and hence $|m_{\mathcal{A}}|\le b$. Together with the lower semicontinuity of $\mathcal{A}\mapsto\mathcal{A}\varphi_Q(0)$ (Proposition \ref{prop:liminf-gen}), this shows that $\mathcal{A}\in K_{a,b}$. Hence $K_{a,b}$ is compact.
\end{proof}

\subsection{Supremum/Infimum as indicators of constraint sets (Proofs of \eqref{ind:sup}, \eqref{ind:inf})}

To apply Lassonde's minimax principle of Theorem \ref{thm:minimax}, we take $X=\levyG_2(\mbr^{2d})$, equipped with the weak topology introduced above, and let $\mcl{Y}:=[C_2^2(\mbr^d)]^2$ be the vector space of test pairs. We define the convex set $Y\subset\mcl{Y}$ by
\begin{align}\label{space:Y}
	Y := \cb{(\vphi,\psi)\in [C_2^2(\mbr^d)]^2 : 
		\inf_{(x,y)\in\mbr^{2d}} \sqb{c - \vphi \oplus \psi} > -\infty }.
\end{align}
In other words, $Y$ consists of all pairs $(\vphi,\psi)$ such that the function $c - \vphi \oplus \psi$ admits a global lower bound.  We note in particular that $\mcl{D}(0,0) \subset Y$ and $[C_b^2(\mbr^d)]^2\subset Y$.  

We begin by establishing the identity \eqref{ind:sup}, which shows that the supremum functional acts as an indicator of the constraint set $\Gamma^\Lambda(\mA,\mB)$.

\begin{lemma}\label{lem:sup-ind}
	Fix $\mJ \in \levyG_2(\mathbb{R}^{2d})$. Then the identity \eqref{ind:sup} holds:
	\[
	\sup_{(\varphi,\psi)\in Y } 
	\Big[ \mA\varphi(0) + \mB\psi(0) - \mJ(\varphi \oplus \psi)(0,0) \Big] 
	=
	\begin{cases}
		0, & \mJ \in \Gamma^\Lambda(\mA,\mB),\\
		+\infty, & \mJ \notin \Gamma^\Lambda(\mA,\mB).
	\end{cases}
	\]
\end{lemma}

\begin{proof}
	Suppose first that $\mJ \in \Gamma^\Lambda(\mA,\mB)$. By the marginal conditions (see Remark \ref{rem:C2-extension}), we have
	\[
	\mJ(\varphi \oplus \psi) = \mJ(\varphi \otimes 1 + 1 \otimes \psi) = \mA \varphi \otimes 1 + 1 \otimes \mB \psi.
	\]
	Evaluating at $(0,0)$ gives, for every $(\vphi,\psi)\in Y$:
	\[
	\mA\varphi(0) + \mB\psi(0) - \mJ(\varphi \oplus \psi)(0,0) = 0.
	\]
	Taking the supremum over all $(\varphi,\psi)$ then gives
	\[
	\sup_{(\varphi,\psi)\in Y} \Big[ \mA\varphi(0) + \mB\psi(0) - \mJ(\varphi \oplus \psi)(0,0) \Big] = 0.
	\]
	
	Now suppose $\mJ \notin \Gamma^\Lambda(\mA,\mB)$. Then at least one of the marginal conditions fails for some functions in $C_b^2(\mbr^d)$.  
	Without loss of generality, suppose there exists a test function $\varphi\in C_b^2(\mbr^d)$ such that $\mJ(\varphi \otimes 1)(0,0) \neq  \mA\varphi(0)$. Replacing $\varphi$ by $-\varphi$ if necessary, we may assume  $\mJ(\varphi \otimes 1)(0,0) <  \mA\varphi(0)$.
	Then, by scaling $\varphi$ by a positive constant $\lambda \to +\infty$, we obtain
	\[
	\sup_{\lambda > 0} \big[ \mA (\lambda \varphi)(0) - \mJ((\lambda \varphi) \otimes 1)(0,0) \big] = +\infty.
	\]
	Hence the supremum diverges, giving
	\[
	\sup_{(\varphi,\psi)\in Y} \Big[ \mA\varphi(0) + \mB\psi(0) - \mJ(\varphi \oplus \psi)(0,0) \Big] = +\infty.
	\]
	This proves the lemma. 
\end{proof}

\begin{lemma}\label{lem:inf-ind}
	Fix $(\vphi,\psi)\in Y$. Then the identity \eqref{ind:inf} holds:
	\begin{align*}
		\inf_{\mJ\in\levyG_2(\mbr^{2d})}[\mJ(c-\vphi\oplus\psi)](0,0)= \begin{cases}
			0, & c- \vphi\oplus \psi \text{ attains a global minimum at }(0,0),\\
			-\infty, & \mbox{else}.\\
		\end{cases}
	\end{align*}
\end{lemma}

\begin{proof}
	Let \( f := c - \varphi \oplus \psi \in C_2^2(\mathbb{R}^{2d}) \). Since Lévy operators annihilate constants, adding constants to \(\varphi\) and \(\psi\) does not affect the value of \(\mathcal{J} f(0,0)\). We may therefore assume, without loss of generality, that \(\varphi(0) = \psi(0) = 0\), and consequently \(f(0,0) = 0\).
	
	Suppose $f$ attains a global minimum at $(0,0)$. Then $f\ge 0$ on $\mathbb{R}^{2d}$ and $(0,0)$
	is a global minimiser of $f$. By the positive maximum principle for Lévy
	generators, we have
	\[
	(\mathcal J f)(0,0)\ge 0
	\quad\text{for all } \mathcal J\in\levyG_2(\mathbb{R}^{2d}).
	\]
	Since the zero operator belongs to $\levyG_2(\mathbb{R}^{2d})$ and satisfies
	$(\mathcal J f)(0,0)=0$, the first equality follows.
	
	Suppose $f$ does not attain a global minimum at $(0,0)$. 
	Then there exists a nonempty open
	set $U\subset\mathbb{R}^{2d}$ on which $f<0$. Choose a finite nonzero
	measure $\Theta$ supported on a compact set $K\subset U$ and define the
	pure-jump Lévy generator $\mJ \in \levyG_2(\mbr^{2d})$:
	\[
	(\mathcal J g)(x,y)
	:= \int_{\mathbb{R}^{2d}}\bigl[g(x'+x,y'+y)-g(x,y)\bigr]\,d\Theta(x',y'),
	\qquad (x,y)\in\mathbb{R}^{2d}.
	\]
	Since $f(0,0)=0$ and $\Theta(K)>0$, we have
	\[
	(\mathcal J f)(0,0) = \int_{K} f(x',y')\,d\Theta(x',y') < 0.
	\]
	For $\lambda>0$, set $\mathcal J_\lambda := \lambda \mathcal J$. Then
	$\mathcal J_\lambda\in\levyG_2(\mathbb{R}^{2d})$ and $\lim_{\la\to\infty}(\mJ_\la f)(0,0)=-\infty$.
	Consequently,
	\[
	\inf_{\mathcal J\in\levyG_2(\mathbb{R}^{2d})} (\mathcal J f)(0,0) = -\infty,
	\]
	which completes the proof.
\end{proof}

\subsection{Proof of Proposition \ref{thm:strongdual}}

We now proceed to the proof of Proposition \ref{thm:strongdual}.

\begin{proof}[Proof of Proposition \ref{thm:strongdual}]
	As explained earlier, it suffices to establish the equality $\ta^\La=\om'$ at the point $(x,y)=(0,0)$. 
	Recall that $X=\levyG_2(\mbr^{2d})$ is equipped with the weak topology introduced in Definition \ref{def:weak-top}. The space $C_2^2(\mbr^d)$ consists of $C^2$ functions satisfying the growth conditions \eqref{cond:C2growth}. Let $Y\subset [C_2^2(\mbr^d)]^2$ be the convex set defined in \eqref{space:Y}. 
	
	Let $F:X\times Y\to\mbr$ be the functional defined in \eqref{def:F.3}. The formal computation leading to the identity $\ta^\La(0,0)=\om'(0,0)$ is given in \eqref{eq:heuristic}. It therefore remains to justify rigorously each step of this computation, in particular the equalities labelled (1)--(3) in \eqref{eq:heuristic}. 
	The identities corresponding to steps (1) and (3) are established in Lemmas \ref{lem:sup-ind} and \ref{lem:inf-ind}, respectively. We are thus left to justify step (2), namely the interchange of infimum and supremum. This will follow from an application of the minimax principle, Theorem \ref{thm:minimax}.
	
	We first verify that the functional $F$ satisfies the assumptions of Theorem \ref{thm:minimax}. 
	By construction, the map $(\mJ,\vphi,\psi)\mapsto F(\mJ,\vphi,\psi)$ is separately affine, and therefore convex in $\mJ$ and concave in $(\vphi,\psi)$.
	We next verify the required lower semicontinuity. 
	Fix $(\vphi,\psi)\in Y$. Then the function $f:=c-\vphi\oplus\psi$ belongs to $C_2^2(\mbr^{2d})$ and admits a uniform lower bound $f\ge -M$ for some $M\ge0$. 
	By Proposition \ref{prop:liminf-gen}, for any sequence $\mJ_n\xrightarrow{w}\mJ$ in $X$,
	\begin{align*}
		\liminf_{n\to\infty}F(\mJ_n,\vphi,\psi)
		&= \liminf_{n\to\infty}\sqb{\mJ_n(c-\vphi\oplus\psi)(0,0)+\mA\vphi(0)+\mB\psi(0)}\\
		&= \liminf_{n\to\infty}\mJ_n(c-\vphi\oplus\psi)(0,0)+\mA\vphi(0)+\mB\psi(0)\\
		&\ge \mJ(c-\vphi\oplus\psi)(0,0)+\mA\vphi(0)+\mB\psi(0)
		=F(\mJ,\vphi,\psi).
	\end{align*}
	Since $\levyG_2(\mbr^{2d})$ is sequential (Remark \ref{rem:metrizable}), this shows the map $X\ni\mJ\mapsto F(\mJ,\vphi,\psi)$ is lower semicontinuous.

		It remains to verify the compactness condition in Theorem \ref{thm:minimax}. Choose any $\mathcal{J}_0\in\Gamma^\Lambda(\mathcal{A},\mathcal{B})$, for example the generator of the independent coupling. By the marginal identities, $F(\mathcal{J}_0,\varphi,\psi)=(\mathcal{J}_0c)(0,0)$ for every $(\varphi,\psi)\in Y$. Hence
	\[
	\sup_{(\varphi,\psi)\in Y}\inf_{\mathcal{J}\in X}F(\mathcal{J},\varphi,\psi)
	\le (\mathcal{J}_0c)(0,0)<\infty.
	\]
	Fix a constant $\gamma>(\mathcal{J}_0c)(0,0)$.
	
	Let $E = \{\pm e_k : 1 \le k \le d\} \subset \mathbb{R}^d$ be the set of all signed standard basis vectors of $\mathbb{R}^d$. Define $\phi, \ell_e \in C_2^2(\mathbb{R}^d)$ and the finite set $\mathcal{H}$ of pairs of functions by
	\begin{align}\label{def:vphi-e}
		\phi(x) &:= -\frac12 |x|^2,
		\qquad
		\ell_e(x) := -e^\top x, \nonumber\\
		\mathcal{H} &:= \{(\phi,\phi)\}
		\cup \{(\phi+\ell_e,\phi),\;(\phi,\phi+\ell_e) : e \in E\}.
	\end{align}
	The set $\mathcal{H}$ has $4d+1$ elements. In particular,
	\begin{align}\label{eq:id4}
		c(x,y) - \phi(x) - \phi(y)
		= \frac12 |x-y|^2 + \frac12 |x|^2 + \frac12 |y|^2
		=: \Phi_Q(x,y),
	\end{align}
	where $\Phi_Q(x,y) = (x,y)^\top Q (x,y)$ is the quadratic form of the strictly positive definite matrix
	\[
	Q = \frac12 \begin{bmatrix}
		2I & -I \\ -I & 2I
	\end{bmatrix}.
	\]
	For the other pairs in $\mathcal{H}$, the functions $c-\varphi\oplus\psi$ are respectively $\Phi_Q(x,y)+e^\top x$ and $\Phi_Q(x,y)+e^\top y$. Since $Q$ is strictly positive definite, these functions are bounded below. Hence $\mathcal{H}$ is a finite subset of $Y$ (as defined in \eqref{space:Y}).
	
	We now claim that the sublevel set
	\[
	\mathcal{K}_\gamma
	:= \left\{ \mathcal{J} \in \mathcal{G}_2^\Lambda(\mathbb{R}^{2d}) : F(\mathcal{J}, \varphi, \psi) \le \gamma \text{ for all } (\varphi,\psi) \in \mathcal{H} \right\}
	\]
	is compact. Note that $\mathcal{J}_0\in\mathcal{K}_\gamma$, so this set is nonempty. By Lemma \ref{lem:levy-compact}, it suffices to show that there exist constants $a, b \ge 0$ such that
	\begin{align}\label{eq:compact-inclu}
		\text{if } \mathcal{J} \in \mathcal{K}_\gamma, \text{ then } \mathcal{J}(\Phi_Q)(0,0) \le a \text{ and } |m_\mathcal{J}| \le b.
	\end{align}
	Indeed, this implies $\mathcal{K}_\gamma \subset K_{a,b}$, where $K_{a,b} \subset \mathcal{G}_2^\Lambda(\mathbb{R}^{2d})$ is the compact set defined in Lemma \ref{lem:levy-compact} (applied in dimension $2d$). Since each map $\mathcal{J}\mapsto F(\mathcal{J},\varphi,\psi)$ is lower semicontinuous, $\mathcal{K}_\gamma$ is closed, and is therefore compact.
	
	We now verify \eqref{eq:compact-inclu}. Using $F(\mathcal{J}, \phi, \phi) \le \gamma$ and \eqref{eq:id4}, we have
	\begin{align}\label{step:4.2}
		F(\mathcal{J}, \phi, \phi)
		= \mathcal{J}(\Phi_Q)(0,0) + \mathcal{A}\phi(0) + \mathcal{B}\phi(0) \le \gamma.
	\end{align}
	By the positive maximum principle, $\mathcal{J}(\Phi_Q)(0,0) \ge 0$. Thus
	\[
	a := \gamma - \mathcal{A}\phi(0) - \mathcal{B}\phi(0) \ge 0,
	\qquad
	\mathcal{J}(\Phi_Q)(0,0) \le a.
	\]
	Next, using $F(\mathcal{J}, \phi+\ell_e, \phi) \le \gamma$, we obtain
	\begin{align*}
		F(\mathcal{J}, \phi+\ell_e, \phi)
		&= \mathcal{J}(\Phi_Q)(0,0)
		+ \mathcal{A}\phi(0) + \mathcal{B}\phi(0)
		- \mathcal{J}(\ell_e \oplus 0)(0,0)
		+ \mathcal{A}\ell_e(0)\\
		&= \mathcal{J}(\Phi_Q)(0,0)
		+ \mathcal{A}\phi(0) + \mathcal{B}\phi(0)
		+ (e,0)^\top m_\mathcal{J}
		- e^\top m_\mathcal{A}
		\le \gamma.
	\end{align*}
	Since $\mathcal{J}(\Phi_Q)(0,0) \ge 0$, this gives
	\[
	(e,0)^\top m_\mathcal{J}
	\le a + e^\top m_\mathcal{A}
	\le a + |m_\mathcal{A}|_\infty.
	\]
	Applying the same argument to the pair $(\phi,\phi+\ell_e)$, we similarly obtain
	\[
	(0,e)^\top m_\mathcal{J}
	\le a + |m_\mathcal{B}|_\infty.
	\]
	Since these bounds hold for all signed basis vectors $e \in E$, every component of $m_\mathcal{J} \in \mathbb{R}^{2d}$ is bounded in absolute value by $a + |(m_\mathcal{A},m_\mathcal{B})|_\infty$. Thus $|m_\mathcal{J}| \le b$ for
	\(
	b := \sqrt{2d}\,\bigl(a + |(m_\mathcal{A},m_\mathcal{B})|_\infty\bigr).
	\)
	This proves \eqref{eq:compact-inclu} and completes the verification of the compactness condition.
	
	All hypotheses of Theorem \ref{thm:minimax} are now verified. Applying that theorem justifies step (2) in \eqref{eq:heuristic}, and hence yields $\theta^\Lambda(0,0)=\omega'(0,0)$. By Propositions \ref{prop:ta-affine} and \ref{prop:om-affine}, the equality holds at every $(x,y)\in\mathbb{R}^{2d}$.
\end{proof}


\bibliography{refs}

@incollection{MR3050280,
	author    = {Ambrosio, Luigi and Gigli, Nicola},
	title     = {A User's Guide to Optimal Transport},
	booktitle = {Modelling and Optimisation of Flows on Networks},
	editor    = {Piccoli, Benedetto and Rascle, Michel},
	series    = {Lecture Notes in Mathematics},
	volume    = {2062},
	pages     = {1--155},
	publisher = {Springer},
	address   = {Berlin, Heidelberg},
	year      = {2013},
	doi       = {10.1007/978-3-642-32160-3_1},
	url       = {https://doi.org/10.1007/978-3-642-32160-3_1},
}

@article{Chaintron_2022_1,
	title      = {Propagation of chaos: A review of models, methods and applications. {I}. {M}odels and methods},
	volume     = {15},
	issn       = {1937-5077},
	url        = {http://dx.doi.org/10.3934/krm.2022017},
	doi        = {10.3934/krm.2022017},
	number     = {6},
	journal    = {Kinetic and Related Models},
	publisher  = {American Institute of Mathematical Sciences (AIMS)},
	author     = {Chaintron, Louis-Pierre and Diez, Antoine},
	year       = {2022},
	pages      = {895-1015}
}

@book{MR2459454,
	author     = {Villani, C\'{e}dric},
	title      = {Optimal Transport: Old and New},
	series     = {Grundlehren der mathematischen Wissenschaften},
	volume     = {338},
	publisher  = {Springer},
	address    = {Berlin, Heidelberg},
	year       = {2009},
	pages      = {xxii+973}
}

@book{Sato1999,
	author       = {Sato, Ken-iti},
	title        = {L\'evy Processes and Infinitely Divisible Distributions},
	series       = {Cambridge Studies in Advanced Mathematics},
	volume       = {68},
	publisher    = {Cambridge University Press},
	address      = {Cambridge},
	year         = {1999},
	pages        = {486},
	isbn         = {0521553024},
}

@article{Chen1994,
	author       = {Chen, Mu-Fa},
	title        = {Optimal Markovian Couplings and Applications},
	journal      = {Acta Mathematica Sinica (English Series)},
	volume       = {10},
	number       = {3},
	pages        = {260--275},
	year         = {1994},
	doi          = {10.1007/BF02560717},
}

@article{LimTeoh2025,
	author       = {Lim, Tau Shean and Teoh, Chao Dun},
	title        = {Abstract Formulation of Mean-Field Models and Propagation of Chaos},
	journal      = {arXiv preprint},
	year         = {2025},
	volume       = {arXiv:2508.02224},
	url          = {https://arxiv.org/abs/2508.02224},
}

@article{KangLim2025,
	author       = {Wei Yang Kang and Tau Shean Lim},
	title        = {On Optimal Markovian Couplings of Lévy Processes},
	journal      = {arXiv preprint arXiv:2509.23086},
	year         = {2025},
	url          = {https://arxiv.org/abs/2509.23086},
	note         = {82 pages}
}

@article{KendallMajkaMijatovic2024,
	author       = {Kendall, Wilfrid S. and Majka, Mateusz B. and Mijatovi{\'c}, Aleksandar},
	title        = {Optimal Markovian coupling for finite-activity L\'evy processes},
	journal      = {Bernoulli},
	volume       = {30},
	number       = {4},
	pages        = {2821--2845},
	year         = {2024},
	doi          = {10.3150/23-BEJ1696},
	url          = {https://projecteuclid.org/journals/bernoulli/volume-30/issue-4/Optimal-Markovian-coupling-for-finite-activity-L%C3%A9vy-processes/10.3150/23-BEJ1696.short},
}

@article{Zhang2000,
	author       = {Zhang, Shaoyi},
	title        = {Existence and Application of Optimal Markovian Coupling with Respect to Non-Negative Lower Semi-Continuous Functions},
	journal      = {Acta Mathematica Sinica, English Series},
	volume       = {16},
	number       = {2},
	year         = {2000},
	pages        = {261--270},
	doi          = {10.1007/s101140000049},
}

@incollection{Chen2020OptimalCouplings,
	author       = {Chen, Mu-Fa},
	title        = {Optimal Couplings and Application to Riemannian Geometry},
	booktitle    = {Probability Theory and Mathematical Statistics},
	editor       = {{\relax edited by contributors}},
	publisher    = {Walter de Gruyter GmbH},
	year         = {2020},
	pages        = {121--142},
	doi          = {10.1515/9783112319321-010},
}

@article{Griffeath1975,
	author = {Griffeath, D.},
	title = {A maximal coupling for Markov chains},
	journal = {Zeitschrift für Wahrscheinlichkeitstheorie und Verwandte Gebiete},
	volume = {31},
	number = {2},
	pages = {95--106},
	year = {1975},
	publisher = {Springer}
}

@article{LindvallRogers1986,
	author = {Lindvall, T. and Rogers, L. C. G.},
	title = {Coupling of Multidimensional Diffusions by Reflection},
	journal = {The Annals of Probability},
	volume = {14},
	number = {3},
	pages = {860--872},
	year = {1986},
	publisher = {Institute of Mathematical Statistics}
}

@article{doeblin1937theorie,
	title={Expos\'e de la th\'eorie des cha\^\i nes simples constantes de Markov \`a un nombre fini d'\'etats},
	author={Doeblin, W.},
	journal={Rev. Math. Union Interbalcanique},
	volume={1},
	pages={77--105},
	year={1937}
}

@book{levin2017markov,
	title={Markov Chains and Mixing Times},
	author={Levin, David A. and Peres, Yuval and Wilmer, Elizabeth L.},
	edition={2nd},
	year={2017},
	publisher={American Mathematical Society}
}

@inproceedings{propp1996exact,
	title={Exact sampling with coupled Markov chains and applications to statistical mechanics},
	author={Propp, James G. and Wilson, David B.},
	booktitle={Random Structures \& Algorithms},
	volume={9},
	pages={223--252},
	year={1996}
}

@book{stroock2006multidimensional,
	title={Multidimensional Diffusion Processes},
	author={Stroock, Daniel W. and Varadhan, S. R. Srinivasa},
	year={2006},
	publisher={Springer},
	series={Classics in Mathematics},
	address={Berlin},
	isbn={978-3-540-28998-2}
}

@article{kendall1986nonnegative,
	title={Nonnegative {R}icci curvature and the {B}rownian coupling property},
	author={Kendall, Wilfrid S.},
	journal={Stochastic Processes and their Applications},
	volume={20},
	number={2},
	pages={245--258},
	year={1986},
	publisher={Elsevier},
	doi={10.1016/0304-4149(85)90112-2}
}

@book{wang2013harnack,
	title={Harnack Inequalities and Applications for Stochastic Partial Differential Equations},
	author={Wang, Feng-Yu},
	series={Springer Briefs in Mathematics},
	publisher={Springer},
	address={New York},
	year={2013},
	doi={10.1007/978-1-4614-8535-3}
}

@article{kendall2015coupling,
	author    = {Wilfrid S. Kendall},
	title     = {Coupling, local times, immersions},
	journal   = {Bernoulli},
	volume    = {21},
	number    = {2},
	pages     = {1014--1046},
	year      = {2015},
	doi       = {10.3150/14-BEJ596}
}

@article{LiangSchillingWang2020,
	author    = {Liang, M. and Schilling, R. L. and Wang, J.},
	title     = {A unified approach to coupling SDEs driven by L\'evy noise and some applications},
	journal   = {Bernoulli},
	volume    = {26},
	number    = {1},
	pages     = {664--693},
	year      = {2020},
	doi       = {10.3150/19-BEJ1148},
}

@book{Bertoin1996,
    author = {Bertoin, Jean},
    title = {L\'evy Processes},
    series = {Cambridge Tracts in Mathematics},
    volume = {121},
    publisher = {Cambridge University Press},
    year = {1996},
    pages = {265},
    isbn = {0-521-56243-0}
}

@article{JackaMijatovic2015,
	author    = {Saul D. Jacka and Aleksandar Mijatovi{\'c}},
	title     = {Coupling and tracking of regime-switching martingales},
	journal   = {Electronic Journal of Probability},
	volume    = {20},
	pages     = {1--39},
	year      = {2015},
	number    = {38},
	doi       = {10.1214/ejp.v20-2307},
	mrnumber  = {3338511},
	zblnumber = {1333.60085}
}

@article{Lassonde1983,
	author  = {Lassonde, Marc},
	title   = {On the use of {KKM} multifunctions in fixed point theory and related topics},
	journal = {Journal of Mathematical Analysis and Applications},
	volume  = {97},
	number  = {1},
	pages   = {151--201},
	year    = {1983},
	doi     = {10.1016/0022-247X(83)90244-5}
}

@article{BremaudYor1978,
	author  = {Br{\'e}maud, Pierre and Yor, Marc},
	title   = {Changes of filtrations and of probability measures},
	journal = {Zeitschrift f{\"u}r Wahrscheinlichkeitstheorie und Verwandte Gebiete},
	volume  = {45},
	number  = {4},
	pages   = {269--295},
	year    = {1978},
	doi     = {10.1007/BF00537538}
}
\bibliographystyle{abbrv}

\end{document}